\documentclass[a4paper,11pt]{amsart}
\usepackage{amsmath,amsfonts,amssymb,graphics,pgfplots}
\usepackage{cite}
\usepackage{hyperref}
\usepackage{color}
\usepackage[normalem]{ulem} 
\usepackage{enumitem}
\usepackage{bbm}
\usepackage{mathtools}
\usepackage{a4wide}
\usepackage{multirow}
\usepackage{booktabs}
\usepackage{tikz}
\usepackage[foot]{amsaddr}
 
\newcommand \id{\mathbbm 1}

\newtheorem{theorem}{Theorem}[section]
\newtheorem{lemma}[theorem]{Lemma}
\newtheorem{proposition}[theorem]{Proposition}
\newtheorem{corollary}[theorem]{Corollary}
\newtheorem{assumption}[theorem]{Assumption}{\bf}{\it}
{\normalfont\bfseries}{\itshape}

\newtheorem{definition}[theorem]{Definition}
\theoremstyle{remark}
\newtheorem{remark}[theorem]{Remark}

\begin{document}

\title[Discretisation schemes for level sets of planar Gaussian fields]{Discretisation schemes for level sets \\ of planar Gaussian fields}
\author{D. Beliaev}
\address{Mathematical Institute, University of Oxford}
\email{belyaev@maths.ox.ac.uk}
\author{S. Muirhead}
\email{stephen.muirhead@kcl.ac.uk}
\subjclass[2010]{}
\thanks{The authors were supported by the Engineering \& Physical Sciences Research Council (EPSRC) Fellowship EP/M002896/1 held by Dmitry Belyaev. They would like to thank Vincent Beffara, Damien Gayet, Vincent Tassion and Igor Wigman for helpful comments and suggestions.}
\date{\today}

\begin{abstract}
Smooth random Gaussian functions play an important role in mathematical physics, a main example being the random plane wave model conjectured by Berry to give a universal description of high-energy eigenfunctions of the Laplacian on generic compact manifolds. Our work is motivated by questions about the geometry of such random functions, in particular relating to the structure of their nodal and level sets.

We study four discretisation schemes that extract information about level sets of planar Gaussian fields. Each scheme recovers information up to a different level of precision, and each requires a maximum mesh-size in order to be valid with high probability. The first two schemes are generalisations and enhancements of similar schemes that have appeared in the literature \cite{BG16, MW07}; these give complete topological information about the level sets on either a local or global scale. As an application, we improve the results in \cite{BG16} on Russo-Seymour-Welsh estimates for the nodal set of positively-correlated planar Gaussian fields. The third and fourth schemes are, to the best of our knowledge, completely new. The third scheme is specific to the nodal set of the random plane wave, and provides global topological information about the nodal set up to `visible ambiguities'. The fourth scheme gives a way to approximate the mean number of excursion domains of planar Gaussian fields.
\end{abstract}

\maketitle

\tableofcontents

\section{Introduction}
\label{sec:intro}

Let $\Psi: \mathbb{R}^2 \rightarrow \mathbb{R}$ be a planar Gaussian field, that is, a random function whose finite-dimensional distributions are Gaussian random variables. We shall throughout assume that~$\Psi$ is stationary and normalised to have zero mean and unit variance at each point. This implies that $\Psi$ may be defined through its positive-definite correlation kernel $\kappa: \mathbb{R}^2 \rightarrow [-1, 1]$, satisfying $\kappa(0) = 1$ and, for each $s,t \in \mathbb{R}^2$,
\[  
\kappa(s-t) := \mathbb{E} [ \Psi(s) \Psi(t) ] . 
\]
The main objects of study in this paper are the \textit{level sets} of $\Psi$, that is, the random sets
\[  \mathcal{N}_\ell:= \left\{ s \in \mathbb{R}^2 : \Psi(s) = \ell \right\} ,  \quad \ell \in \mathbb{R}.  \]
Throughout the paper we assume that $\kappa$ is $C^6$ at the origin and $\kappa_{vv} \neq 0$ for each unit vector $v \in S^1$,  where for a function $f:\mathbb{R}^2 \rightarrow \mathbb{R}$ and a vector $v \in \mathbb{R}^2$ we use $f_v$ to denote the derivative of $f$ in the direction $v$. This ensures that, for a fixed level $\ell \in \mathbb{R}$, the level set $\mathcal{N}_\ell$ almost surely consists of a collection of disjoint simple closed curves. We refer to the components of $\mathcal{N}_\ell$ as the \textit{level lines} and the components of $\mathbb{R}^2 \setminus \mathcal{N}_\ell$ as the \textit{excursion domains}. In the special case $\ell = 0$, we refer to $\mathcal{N} := \mathcal{N}_0$ as the \textit{nodal set}, and the level lines and the excursion domains as the \textit{nodal lines} and \textit{nodal domains} respectively.

A \textit{discretisation scheme for a level set} is a method of extracting information about the level set through discrete observations of the field. To illustrate, suppose the aim is to assess, for a large fixed box, whether there exists an excursion domain that crosses the box horizontally. By choosing a suitable lattice and a sufficiently fine mesh, we would expect to be able to determine this event with high probability by sampling the value of the random field at the vertices of the lattice. 

It is not hard to see that, as the size of the box increases, a progressively finer mesh must be used in order to control the errors that arise in this procedure; this is since any event depending on a fixed scale becomes overwhelmingly likely to occur somewhere inside the box. Our main aim is to quantify the optimum \textit{scale} at which the mesh-size must decrease as the size of the box grows. As illustrated by our results, the optimum scale depends on the exact property of the level set to be extracted.

\subsection{Level sets of planar Gaussian fields}

Understanding the geometry of the level sets of planar Gaussian fields has been of great interest to mathematical physicists over the last 30 years. Local properties of the level sets, such as their total length in large boxes, are generally well-understood and easy to compute explicitly by direct integral methods. Global properties, such as the number of excursion domains, or crossing events for excursion domains, are much more subtle to understand. For a general review of properties of level sets of Gaussian random fields, see \cite{NS15}.

In regards to the global properties, a breakthrough result of Nazarov and Sodin \cite{NS09} has established that, under some weak conditions on the kernel $\kappa$, the number of nodal domains of~$\Psi$ satisfies a law of large numbers (although not considered in \cite{NS09}, an analysis of the proof shows that this holds also for the excursion domains at any level.) To make this precise, for each bounded domain~$D$ let $N(D)$ denote the number of nodal domains in $D$, i.e.\ the number of components of $D \setminus \mathcal{N}$. Then, under certain conditions on $\kappa$, there exists a constant $c_{NS} = c_{NS}(\kappa) \ge 0$ such that, for any smooth bounded domain $D$, 
\begin{align}
\label{eq:ns}
N(sD) / \text{Area}(sD) \to c_{NS}  \quad \text{almost surely and in mean},
\end{align}
where we use $sD$ to denote the scaled set $\{sx : x \in D\} \subseteq \mathbb{R}^2$. On the other hand, even in the case that $c_{NS}$ is known to exist, there is no known way to compute its value. Indeed, the value of $c_{NS}$ is not known explicitly for any (non-degenerate) Gaussian field.  

The connectivity properties of the level sets of planar Gaussian fields are also challenging to understand. One important question is whether the level sets are almost surely bounded. This was confirmed in the special case of the nodal set of positively-correlated planar Gaussian fields -- i.e.\ the case that $\ell = 0$ and $\kappa \ge 0$ -- in a result of Alexander \cite{Alex96}, but the general case remains open. A recent result of Beffara and Gayet \cite{BG16} has provided, for the first time, more explicit control over the connectivity of the nodal set of positively-correlated planar Gaussian fields, in the form of a Russo-Seymour-Welsh (RSW) estimate.

\begin{definition}[RSW estimate]
\label{def:rsw}
A random set $\mathcal{S} \subseteq \mathbb{R}^2$ \textbf{satisfies the RSW estimate} if, for each smooth bounded domain $D \subseteq \mathbb{R}^2$ and disjoint smooth boundary arcs $\gamma$ and $\gamma'$ on $\partial D$, there exists a constant $c = c(D, \gamma, \gamma') > 0$ such that, for $s > 0$ sufficiently large,
\[   \mathbb{P}\left( \text{there exists a component of } sD \cap \mathcal{S} \text{ that intersects } s\gamma \text{ and } s\gamma' \right)     > c  .\]
\end{definition}

Conditions under which the nodal set $\mathcal{N}$ (and hence also the nodal domains $\mathbb{R}^2 \setminus \mathcal{N}$) satisfies the RSW estimate in \eqref{def:rsw} were given in \cite{BG16}:

\vspace{0.2cm}
\noindent \textbf{Theorem $4.9$ of {\cite{BG16}}.} 
\textit{ Fix $\delta > 0$. Let $\Psi$ be a stationary planar Gaussian field that is almost surely~$C^4$ and such that the distribution of $\nabla \Psi (0)$ is non-degenerate. Suppose further that the correlation kernel $\kappa$ is invariant under reflection through the horizontal axis and under rotation by~$\pi/2$, that $\kappa(x) \ge 0$, and that  $\kappa(x) = o(|x|^{- \alpha - \delta})$ as $|x| \to \infty$ for $\alpha = 144 + 128 \log_{4/3}(3/2) \approx 325$. Then the nodal set $\mathcal{N}$ satisfies the RSW estimate in Definition~\ref{def:rsw}. } 
\vspace{0.2cm}

As an application of our results, we significantly weaken the necessary decay exponent under which the RSW estimate is known to hold, from $\alpha \approx 325$ to $\alpha = 16$ (see Theorem~\ref{thm:rsw} below).

A particularly important planar Gaussian field is the \textit{random plane wave}, which is the stationary, normalised Gaussian field with the correlation kernel
\begin{align}
\label{eq:rpw}
\kappa(s) :=  J_0(k |s|), 
\end{align}
where $J_0$ is the zeroth Bessel function, $k > 0$ is a scale parameter which encodes the \textit{frequency} or \textit{inverse wave-length} of the plane wave, and $|\cdot|$ denotes the standard $\ell_2$ distance. The random plane wave is almost surely smooth, and is the canonical stationary, isotropic Gaussian element in the Hilbert space of planar functions satisfying the Helmholtz equation
\begin{align}
\label{eq:helm}
\Delta \Psi + k^2 \Psi =0 .
\end{align}
The random plane wave is of particular importance because it is conjectured by Berry \cite{Berry77} to be a universal model for the high-energy eigenfunctions of the Laplacian in domains with chaotic dynamics. It is also known to be a universal scaling limit of many ensembles of eigenfunctions of the Laplacian on Riemannian manifolds. 

Solutions of equation \eqref{eq:helm} have been extensively studied, and there are many \textit{deterministic} results that can be applied to the level sets of the random plane wave. For instance, it is known that there are universal positive constants $c_1$ and $c_2$ such that the nodal set $\mathcal{N}$ intersects every disc of radius $c_1/k$, and each nodal domain contains a disc of radius $c_2/k$. 

Bogomolny and Schmit conjectured in \cite{BS02} that the nodal set of the random plane wave is well-approximated by a small perturbation of the square lattice with mesh-size $2\pi/k$, and in particular the nodal domains can be modelled by critical percolation clusters on the square lattice. Based on this idea they made a precise conjecture about the constant $c_{NS}$ in \eqref{eq:ns} for the random plane wave. Recent computer simulations \cite{BK13,Konrad12} have given very strong evidence that this prediction is slightly inaccurate. On the other hand, these simulations also provided evidence that some {\em global} properties of the nodal domains and critical percolation clusters match very well. In particular, this is true for the crossing probabilities discussed above.

\subsection{Discretisation schemes for level sets of random fields}

Discretisation schemes provide an important tool with which to study the level sets of random fields. To see why, consider that many random fields of interest are extremely rigid, for instance the random plane wave is real analytic. This means that care needs to be taken when working with them, in particular, one cannot condition on the event that the field takes certain values in any open domain without determining the whole field. Conditioning instead on values in a discrete lattice gives a possible way to circumvent this problem.

A second importance of discretisation schemes comes from numerical methods, which are extensively used to study the geometric properties of random fields. These are intrinsically discrete, and hence it is important to have good control on the errors that may arise. As an example of what can go wrong, observe that for any large box there is a positive probability that the random plane wave takes positive values at all lattice points in the box, which suggests a positive probability of having one giant nodal domain; as we just saw, this is deterministically prohibited.

Various discretisation schemes for level sets of planar random fields have previously appeared in the literature. An early work was \cite{MW07}, which developed such a scheme for general planar random fields. This is very similar to our first scheme -- assessing the validity of the discretisation on the local scale, see Theorem~\ref{thm:main1} -- and is based around controlling the event that the level set $\mathcal{N}_\ell$ crosses an edge twice (what we call a \textit{double-crossing}; see Section \ref{sec:overview}). On the other hand, \cite{MW07} does not, in general, give a bound on the maximum mesh-size for which the scheme is valid; this is given only in the case of one specific Gaussian field.

More close to our results is the discretisation scheme for planar Gaussian fields in \cite[Theorem 1.5]{BG16}. Again this is very similar to our first scheme in Theorem \ref{thm:main1}, and is also based around controlling double-crossings. Although \cite[Theorem 1.5]{BG16} does give a general bound on the maximum mesh-size, the provided bound is much more restrictive than the one we give in Theorem~\ref{thm:main1}. In particular, the scheme in \cite[Theorem 1.5]{BG16} requires that the mesh-size $\varepsilon$ decays as $\varepsilon = o(s^{-8 - \delta})$ for some $\delta > 0$, where $s$ is the scale of the domain in which the level sets are discretised. This can be compared to our scheme in Theorem \ref{thm:main1}, which is valid if $\varepsilon = o(s^{-2-\delta})$.

In \cite{BG16}, the discretisation scheme was combined with a result of Tassion \cite{Tas16} to establish that the RSW estimate in Definition \ref{def:rsw} holds for the nodal set of positively-correlated planar Gaussian fields with sufficiently fast decay of correlations. Since we improve the discretisation scheme in \cite[Theorem 1.5]{BG16}, we are able to widen the applicability of the RSW estimate; see Remarks \ref{rem:bg1} and \ref{rem:bg2} and Theorem~\ref{thm:rsw}.

\subsection{Main results}

Our main results are a series of four discretisation schemes that extract information about the level sets of planar Gaussian fields. Each scheme recovers information to a different level of precision, and each requires a different maximum mesh-size in order to be valid with high probability. We shall present the schemes in decreasing order of precision. Note that we have chosen to state each of our results as an asymptotic statement. In each case, explicit quantitative bounds on the rate of convergence may be recovered from our proofs (see Section \ref{sec:proofs}; note the decay rates of these bounds, up to constants, depend only on the mesh-size $\varepsilon$).

Before we present our results, we state the conditions on the correlation kernel under which our results hold, and give some general definitions.

\subsubsection{Conditions on the correlation kernel}
Our results will be valid under certain smoothness and non-degeneracy assumptions on the Gaussian field, expressed through conditions on $\kappa$ at the origin. These conditions are extremely mild, and will be satisfied in most applications. In particular, it is easy to check that the random plane wave satisfies these conditions.

 \begin{assumption}
\label{assumpt:degen}
Suppose that the following hold:
\begin{enumerate}
 \item (Smoothness) The correlation kernel $\kappa$ is $C^6$ at the origin;
 \item (Non-degeneracy of first derivatives) For all unit vectors $v \in S^1$, $\kappa_{v v}(0) < 0 $.
\end{enumerate}
\end{assumption}
 
 We briefly discuss the relationship between the conditions on $\kappa$ in Assumption \ref{assumpt:degen} and the resulting properties of $\Psi$; see Section \ref{sec:gaussfields} for precise statements. First, condition (1) ensures that~$\Psi$ is three-times differentiable and twice continuously differentiable almost surely. Second, since for each $s \in \mathbb{R}^2$ and $v_1, v_2 \in S^1$,
\[  \mathbb{E} \left[ \Psi_{v_1}(s) \Psi_{v_2}(s) \right] = -\kappa_{v_1 v_2}(0) , \]
condition (2) ensures that the distribution of $\nabla \Psi$ is non-degenerate. Together these conditions are sufficient to ensure that, for a fixed level $\ell \in \mathbb{R}^2$, the components of $\mathcal{N}_\ell$ are almost surely simple closed curves.

Alternatively, as is done in \cite{NS15} for instance, the conditions in Assumption \ref{assumpt:degen} could also be reformulated in terms of the \textit{spectral measure} $\rho = \rho(\kappa)$ defined by
\[ \kappa(s) = \int_{\mathbb{R}^2} e^{2 \pi i  \langle s, \mu \rangle} \, d \rho(\mu) , \quad s \in \mathbb{R}^2 .\]
For instance, $(1)$ could be the replaced by the condition that
\[ \int_{\mathbb{R}^2} |\mu|^6 \, d \rho(\mu) < \infty ,\]
and $(2)$ could be the replaced with the condition that $\rho$ is not supported on any line.

\subsubsection{Level sets and their discretisation}
We begin by introducing the discretisation of the level set~$\mathcal{N}_\ell$ on which our schemes are based. We also make precise the sense in which we shall consider the level set to be well-approximated by its discretisation.

Let $\mathcal{L} = (\mathcal{V}, \mathcal{E}, \mathcal{F})$ be a periodic lattice in $\mathbb{R}^2$, with vertex set $\mathcal{V}$, edge set $\mathcal{E}$, and face/cell set~$\mathcal{F}$; we do not assume all faces in $\mathcal{F}$ are equal, but denote by $d(\mathcal{L})$ the largest diameter of any $f \in \mathcal{F}$. Let a bounded domain $D \subseteq \mathbb{R}^2$ be called \textit{polygonal} if its boundary $\partial D$ consists of a finite number of straight edges. Let a polygonal domain $P \subseteq \mathbb{R}^2$ be called $\mathcal{L}$\textit{-compatible} if its boundary edges are the union of edges in $\mathcal{E}$. For each $\varepsilon > 0$ and bounded domain $D \subseteq \mathbb{R}^2$, let $P^{\varepsilon}(D)$ be the largest polygonal subdomain $P \subseteq D$ such that $P$ is $\varepsilon \mathcal{L}$-compatible. This is well-defined by the periodicity of~$\mathcal{L}$.

We now introduce the discretisation of the level set $\mathcal{N}_\ell$ for a fixed $\ell \in \mathbb{R}$, based on the lattice~$\mathcal{L}$ at a mesh-size $\varepsilon > 0$. Define the signed excursion domains
\begin{equation}
 \label{def:signedregions}
  \mathcal{S}^+ := \left\{ s \in \mathbb{R}^2 : \Psi(s) > \ell \right\}  \quad \text{and} \quad  \mathcal{S}^- := \left\{ s \in \mathbb{R}^2 : \Psi(s) < \ell \right\}  , 
  \end{equation}
and let $\mathcal{P}^\varepsilon = (\mathcal{P}^\varepsilon_v)_{v \in \varepsilon \mathcal{V}} \in \{+1, -1\}^{\mathcal{V}}$ be the percolation process on $\varepsilon \mathcal{V}$ induced by these regions (i.e.\ $\mathcal{P}^\varepsilon_v = 1$ if $v \in \mathcal{S}^+$, and similarly for $\mathcal{S}^-$); this is well-defined, since almost surely $\varepsilon \mathcal{V} \in \mathcal{S}^+ \cup \mathcal{S}^-$. Consider the dual lattice~$\mathcal{L}^\ast$, with vertex set $\mathcal{F}$, face set $\mathcal{V}$, and edge set~$\mathcal{E}^\ast$. Define the discretised level set $\mathcal{N}_\ell^\varepsilon \subseteq \varepsilon \mathcal{L}^\ast$ by prescribing that an edge $e^\ast \in \varepsilon \mathcal{E}^\ast$ belongs to~$\mathcal{N}_\ell^\varepsilon$ if and only if the edge $e \in \varepsilon \mathcal{E}$ that is dual to $e^\ast$ has endpoints of opposite sign in $\mathcal{P}^\varepsilon$; see Figure \ref{fig:disc} below. 

\begin{figure}[ht]
\begin{center}
\includegraphics[scale=1]{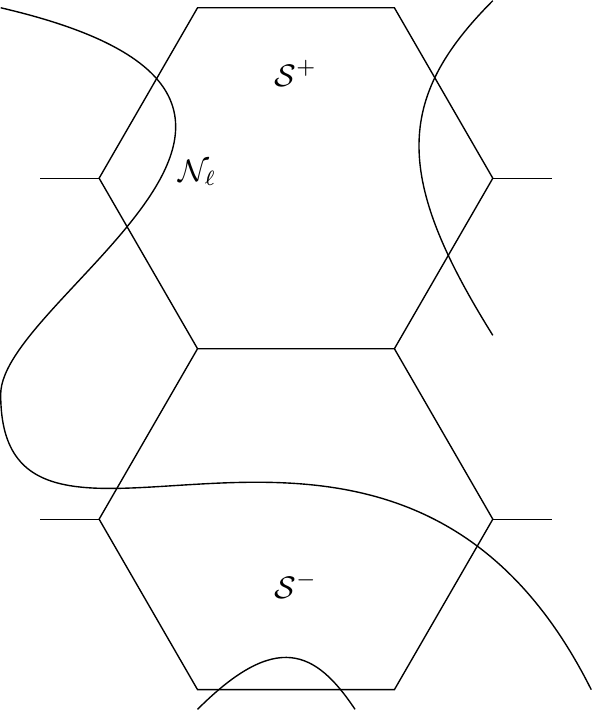}
\includegraphics[scale=1]{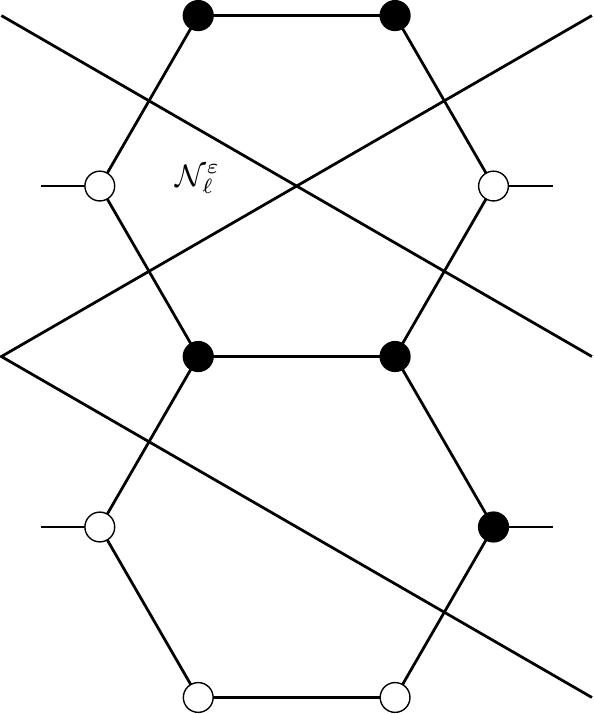}
\end{center}
\caption{An example of a level set $\mathcal{N}_\ell$ (left) and its discretisation $\mathcal{N}_\ell^\varepsilon$ (right); the black and white circles represent the induced percolation process $\mathcal{P}^\varepsilon$ (black indicating $+1$, white indicating $-1$). Here $\mathcal{L}$ is the hexagonal lattice.}
\label{fig:disc}
\end{figure}

Finally, we introduce the sense in which we consider the level set $\mathcal{N}_\ell$ to be well-approximated by its discretisation~$\mathcal{N}_\ell^\varepsilon$. For each $\varepsilon > 0$, bounded domain $D \subseteq \mathbb{R}^2$, and sets $\mathcal{M}_1, \mathcal{M}_2 \subseteq \mathbb{R}^2$ that do not intersect the vertices of~$P^\varepsilon(D)$, we say that $\mathcal{M}_1$ and $\mathcal{M}_2$ are $\varepsilon$\textit{-homeomorphic in}~$D$ if there exists a homeomorphism $h : P^\varepsilon(D)  \to P^\varepsilon(D)$ that maps $\mathcal{M}_1 \cap P^\varepsilon(D)$ onto $\mathcal{M}_2 \cap P^\varepsilon(D)$, that fixes the vertices $\mathcal{V} \cap P^\varepsilon(D)$, and such that
\[ 
|h(s) - s | <    3 d(\mathcal{L}) \, \varepsilon \quad \text{ for each } s \in \mathcal{M}_1 \cap P^\varepsilon(D) . 
\] 
Note that the constant $3 d(\mathcal{L})$ is not optimal for our results to hold, but it is convenient as an upper bound.

\begin{figure}[ht]
\begin{center}
\includegraphics[scale=1]{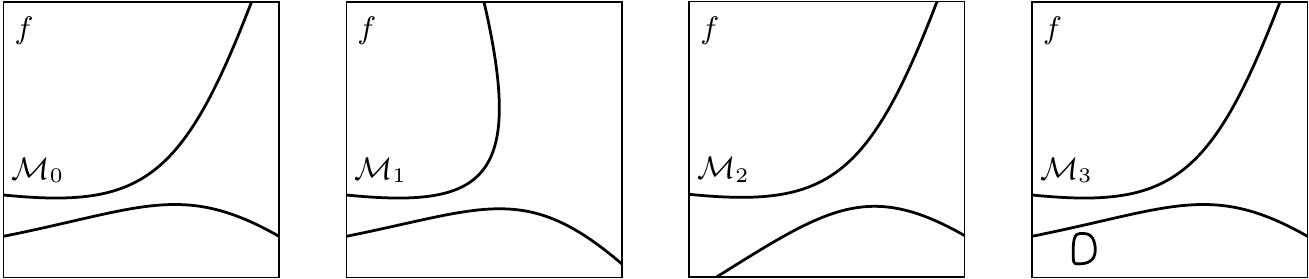}
\end{center}
\caption{A set $\mathcal{M}_0$ contained in a face $f$ of the unit square lattice, and an example of a set $\mathcal{M}_1$ that is $1$-homeomorphic to $\mathcal{M}_0$ in $f$, and sets $\mathcal{M}_2$ and~$\mathcal{M}_3$ that are not (the former since there is a vertex of $f$ lying in different connected components of $f \setminus \mathcal{M}_0$ and $f \setminus \mathcal{M}_2$, the latter since $f \setminus \mathcal{M}_3$ has more connected components than $f \setminus \mathcal{M}_0$).}
\label{fig:con}
\end{figure} 
 
Our discretisation schemes assess how well the level set $\mathcal{N}_\ell$ is approximated by its discretisation~$\mathcal{N}_\ell^\varepsilon$ inside a domain~$D$ by applying the notion of an $\varepsilon$-homeomorphism on two scales: the local scale and the global scale. On the local scale, we assess whether $\mathcal{N}_\ell$ and $\mathcal{N}_\ell^\varepsilon$ are $\varepsilon$-homeomorphic in each face $f \in P^\varepsilon(D) \cap \varepsilon \mathcal{F}$. On the global scale, we assess whether $\mathcal{N}_\ell$ and $\mathcal{N}_\ell^\varepsilon$ are $\varepsilon$-homeomorphic in~$D$. Whether either of these holds depends, in general, on the fineness of the mesh; see Figure \ref{fig:disc2}.

\begin{figure}[ht]
\begin{center}
\includegraphics[scale=1]{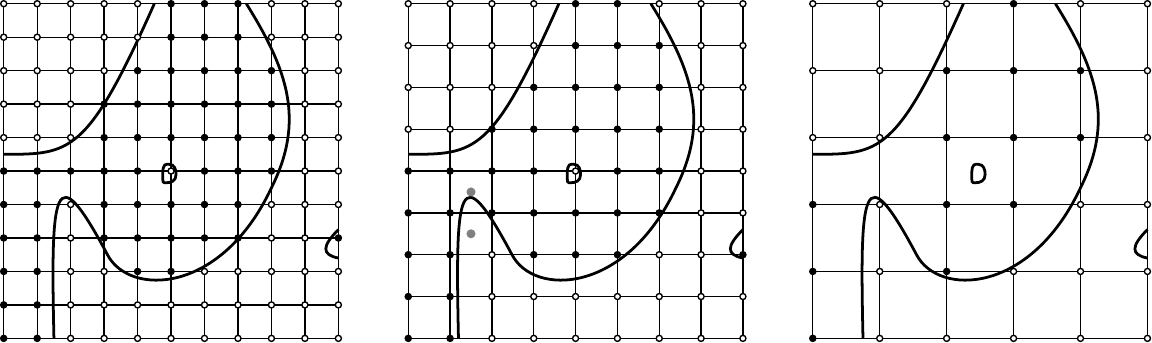}\\
\vspace{0.3cm}
\includegraphics[scale=1]{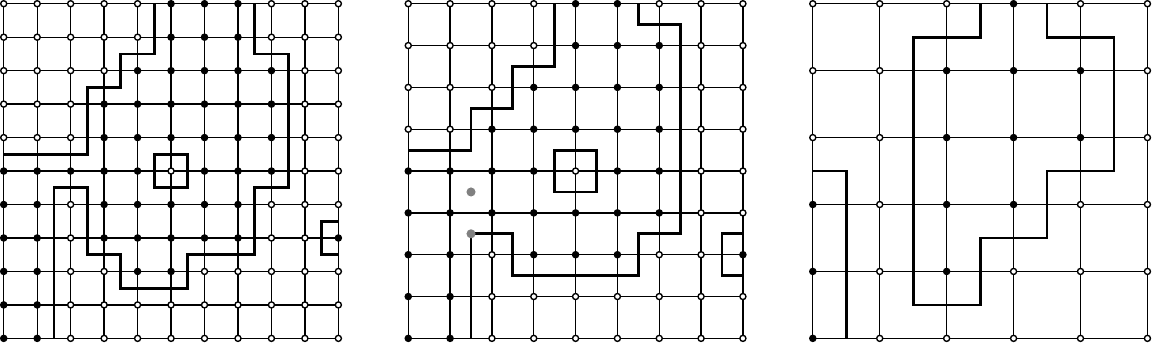}
\end{center}
\caption{Examples of a level set $\mathcal{N}_\ell$ and its discretisation $\mathcal{N}_\ell^\varepsilon$ inside a square domain $D$ for mesh-sizes $\varepsilon > 0$ in which (going left to right) (i) $\mathcal{N}_\ell$ is $\varepsilon$-homeomorphic to $\mathcal{N}_\ell^\varepsilon$ in each face $f \in D \cap \varepsilon \mathcal{L}$, (ii) $\mathcal{N}_\ell$ is $\varepsilon$-homeomorphic to $\mathcal{N}_\ell^\varepsilon$ in $D$ but not in the two faces marked with grey circles, and (iii) $\mathcal{N}_\ell$ is not $\varepsilon$-homeomorphic to $\mathcal{N}_\ell^\varepsilon$ in $D$. Here~$\mathcal{L}$ is the square lattice.}
\label{fig:disc2}
\end{figure} 

\subsubsection{Discretisation on the local scale}

The first discretisation scheme gives complete topological information about the level set $\mathcal{N}_\ell$ within each cell in the lattice; see for instance the left panel of Figure \ref{fig:disc2}. It requires the finest mesh in order to be valid with high probability. 

\begin{theorem}[Discretisation on the local scale]
\label{thm:main1}
Suppose that $\kappa$ satisfies Assumption~\ref{assumpt:degen}. Fix a level $\ell \in \mathbb{R}$ and a bounded domain~$D \subseteq \mathbb{R}^2$. Let $\varepsilon = \varepsilon_s > 0$ be a sequence such that $\varepsilon = o(s^{-2})$ as $s \to \infty$. Then, as $s \to \infty$,
\[  \mathbb{P} \left( \mathcal{N}_\ell   \text{ and } \mathcal{N}_\ell^{\varepsilon} \text{ are } \varepsilon\text{-homeomorphic in each } f \in P^{\varepsilon}(s D) \cap \varepsilon \mathcal{F} \, \right)  \to 1 . \]
\end{theorem}

\begin{remark}
\label{rem:bg1}
The discretisation scheme in Theorem \ref{thm:main1} is an enhancement of the discretisation scheme in \cite[Theorem 1.5]{BG16}, since it implicitly controls the event that the level set $\mathcal{N}_\ell$ crosses an edge $e \in \varepsilon \mathcal{E}$ twice (i.e.\ a double-crossing; see Section \ref{sec:overview}), and since it is valid for a significantly coarser mesh. This enhancement may be used to improve the main result in \cite{BG16}, although since we improve this result further using the next discretisation scheme, we postpone this discussion (see Remark~\ref{rem:bg2} and Theorem \ref{thm:rsw}).
\end{remark}

\subsubsection{Discretisation on the global scale}

The second discretisation scheme gives complete topological information about the level set $\mathcal{N}_\ell$ on a global, rather than local, scale. This is useful for assessing events involving crossings of level sets and excursion domains on macroscopic scales, where it is only important that $\mathcal{N}_\ell$ and $\mathcal{N}_\ell^\varepsilon$ have the same topology globally, rather than within each cell individually; see for instance the central panel of Figure \ref{fig:disc2}.

\begin{theorem}[Discretisation on the global scale]
\label{thm:main2}
Suppose that $\kappa$ satisfies Assumption \ref{assumpt:degen}. Fix $\delta > 0$, a level $\ell \in \mathbb{R}$ and a smooth bounded domain $D \subseteq \mathbb{R}^2$. Let $\varepsilon = \varepsilon_s > 0$ be a sequence such that $\varepsilon = o(s^{-1 - \delta})$ as $s \to \infty$. Then, as $s \to \infty$,
\[ \mathbb{P} \left( \mathcal{N}_\ell  \text{ and } \mathcal{N}_\ell^{\varepsilon}   \text{ are } \varepsilon \text{-homeomorphic in } sD \right) \to 1 . \]
\end{theorem}

\begin{remark}
\label{rem:bg2}
As a corollary of Theorem \ref{thm:main2}, we improve the main result of \cite{BG16} on RSW estimates for the nodal set of positively-correlated planar Gaussian fields. In particular, we significantly weaken the required decay exponent of the correlation kernel, from $\alpha \approx 325$ to $\alpha = 16$ (we also weaken slightly the required non-degeneracy conditions, but this is not as important). So as not to disrupt the exposition of our results, and since the proof is not self-contained (it requires knowledge of \cite{BG16}), we defer our discussion of the proof of this corollary to Appendix \ref{appendix3}.

On the other hand, it is very likely that the optimal decay exponent under which the RSW estimate holds in general is actually much lower, perhaps as low as $\alpha = 1$. This appears challenging to prove, and likely requires new ideas. 
\end{remark}

\begin{theorem}[RSW estimate for the nodal set of positively-correlated planar Gaussian fields]
\label{thm:rsw}
Fix $\delta > 0$. Suppose that $\kappa$ satisfies Assumption \ref{assumpt:degen}, and is also invariant under reflection through the horizontal axis and under rotation by $\pi/2$. Suppose further that $\kappa(x) \ge 0$, and $\kappa(x) = o(|x|^{- 16 - \delta})$ as $|x| \to \infty$. Then the nodal set $\mathcal{N}$ satisfies the RSW estimate in Definition~\ref{def:rsw}. It follows that the nodal domains $\mathbb{R}^2 \setminus \mathcal{N}$ also satisfy this estimate.
\end{theorem}

\subsubsection{Discretisation on the global scale up to visible ambiguities}

The third discretisation scheme applies only to the nodal set of the random plane wave, and is based on the observation that, because of the strong rigidity of the plane wave, errors in the discretisation of the nodal set are with high probability contained within local regions which display a certain visible signature in their discretisation (see Remark \ref{r:heurrpw} for a heuristic explanation of this). We call such local regions `visible ambiguities', and if we are willing to accept that the discretisation is only correct up to these ambiguities, the mesh can be coarsened considerably. This is potentially useful in numerical simulations, since it allows the plane wave to be initially evaluated on a coarse mesh, before a finer mesh is used on a few local regions to resolve the visible ambiguities (see Remark \ref{r:dynamic}).

For clarity of presentation, we only state our result in the case of the regular hexagonal lattice, although we believe that similar results hold, with suitable modification, for any periodic lattice (see Remark \ref{rem:anylattice}).

Let us first make precise the concept of a visible ambiguity. Fix the level $\ell = 0$. For a regular hexagon $H$, we say that $H$ \textit{displays a Type $1$ error pattern} if its vertices, going clockwise, change sign in the percolation process $\mathcal{P}^\varepsilon$ more than twice. For adjacent regular hexagons $H_1$ and $H_2$ with common edge $e$, we say that $H_1$ and $H_2$ \textit{display a Type $2$ error pattern} if: 
\begin{enumerate}
\item The endpoints of the edge $e$ have the same sign in the percolation process $\mathcal{P}^\varepsilon$; and
\item Both $H_1$ and $H_2$ have vertices of the opposite sign in the percolation process $\mathcal{P}^\varepsilon$ to the endpoints of $e$.
\end{enumerate}
Remark that, in the context of the discretised nodal set $\mathcal{N}^\varepsilon := \mathcal{N}_0^\varepsilon$, the presence of an error pattern means either that the discretised nodal lines intersect and so cannot be a true representation of the topology of the nodal set (Type 1), or come so close together that there is a high probability that they are not a true representation (Type 2); see Figure \ref{fig:err} below.
For each $\varepsilon > 0$, we define the set of visible ambiguities $\pi^\varepsilon \subseteq \varepsilon \mathcal{F}$ to be the connected components of the union over faces $f  \in \varepsilon \mathcal{F}$ that display a Type $1$ error pattern and adjacent faces $f_1, f_2 \in \varepsilon \mathcal{F}$ that display a Type $2$ error pattern (these components will usually consist of at most two adjacent faces, but may contain more). 

\begin{figure}[ht]
\begin{center}
\includegraphics[scale=1]{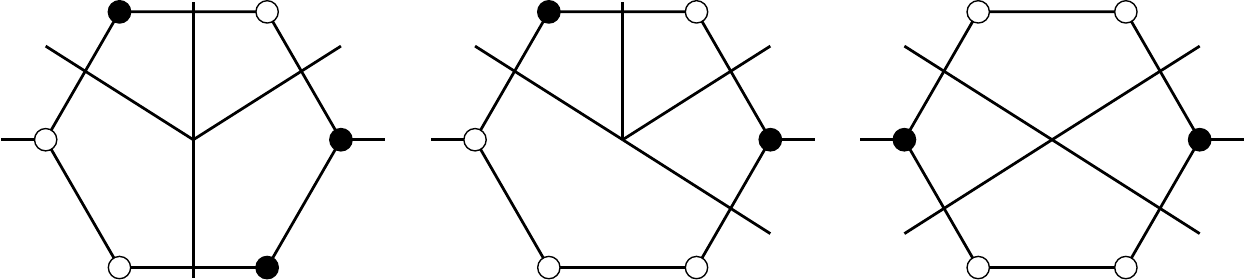}\\
\vspace{0.3cm}
\includegraphics[scale=1]{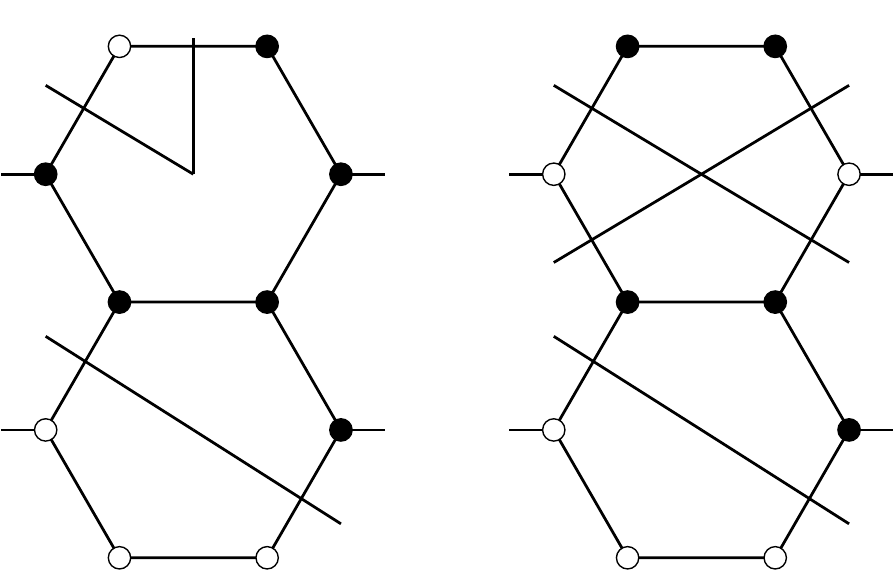}
\end{center}
\caption{Examples of Type $1$ error patterns (top row) and Type $2$ error patterns (bottom row; note, the figure on the right also contains a Type $1$ error pattern); the black and white circles represent the percolation process~$\mathcal{P}^\varepsilon$ and the lines represent the induced discretised nodal set $\mathcal{N}^\varepsilon$.}
\label{fig:err}
\end{figure}

We now make explicit the sense in which we compare the nodal set and its discretisation `up to visible ambiguities'. For a bounded domain $D$ and a set $\mathcal{S}$ such that $|\mathcal{S} \cap \partial D| \in 2 \mathbb{N}$, define a set $\bar{\mathcal{S}} \subseteq D$ to be a \textit{resolution of} $\mathcal{S}$ \textit{in} $D$ if: (i) the components of $\bar{\mathcal{S}}$ that intersect $\partial D$ are a planar matching of the boundary points $\mathcal{S} \cap \partial D$; and (ii) each component of $\bar{\mathcal{S}}$ that does not intersect $\partial D$ is a simple closed curve; see Figure \ref{fig:res}. For each component $\bar{\mathcal{F}} \in \pi^\varepsilon$, we define the collection $\mathcal{N}^\varepsilon(\bar{\mathcal{F}})$ of \textit{resolutions of the ambiguity at} $\bar{\mathcal{F}}$ to be the set of all possible modification of $\mathcal{N}^\varepsilon$ formed by substituting $\mathcal{S} = \mathcal{N}^\varepsilon \cap \bar{\mathcal{F}}$ with a set $\bar{\mathcal{S}} \subseteq \bar{\mathcal{F}}$ that is a resolution of $\mathcal{S}$ in $\bar{\mathcal{F}}$. We define the \textit{discretised nodal sets with resolved ambiguities} $\mathbf{M}^{\varepsilon}$ to be the collection $\{\mathcal{N}^\varepsilon(\bar{\mathcal{F}})\}_{\bar{\mathcal{F}} \in \pi^\varepsilon}$ of all possible such resolutions. For each $\varepsilon > 0$ and bounded domain $D \in \mathbb{R}^2$, we say that the nodal sets $\mathcal{N}$ and $\mathcal{N}^\varepsilon$ are $\varepsilon$\textit{-homeomorphic in} $D$ \textit{up to visible ambiguities} if there exists a $\mathcal{M} \in \mathbf{M}^{\varepsilon}$ such that $\mathcal{N}$ and $\mathcal{M}$ are $\varepsilon$-homeomorphic in $D$.

\begin{figure}[ht]
\begin{center}
\includegraphics[scale=1]{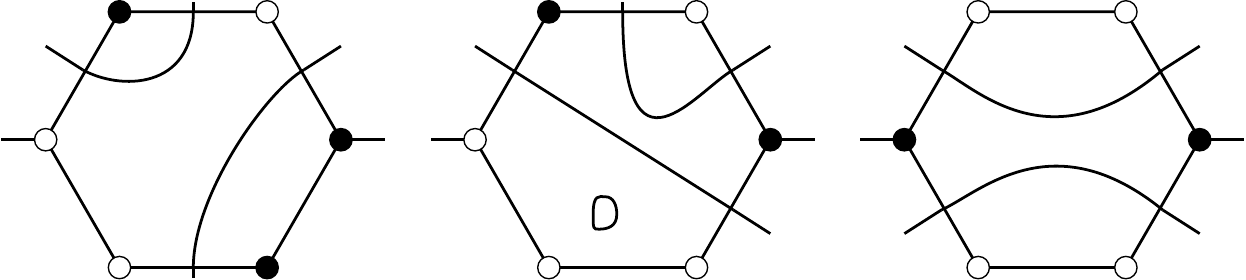}\\
\vspace{0.3cm}
\includegraphics[scale=1]{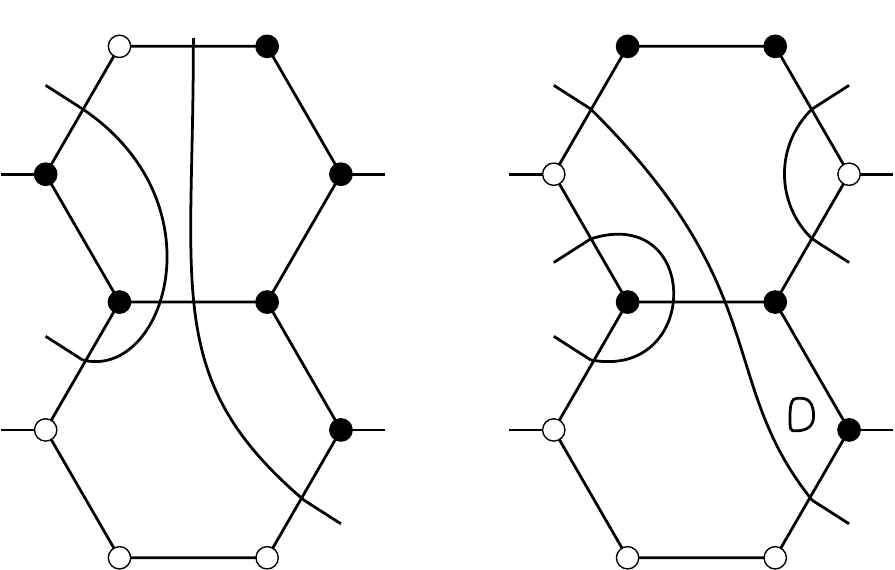}
\end{center}
\caption{Possible resolutions of the discretised nodal sets in the examples in Figure \ref{fig:err}.}
\label{fig:res}
\end{figure}

\begin{theorem}[Discretisation on the global scale up to visible ambiguities]
\label{thm:main3}
Suppose that $\kappa$ satisfies \eqref{eq:rpw} and that $\mathcal{L}$ is a regular hexagonal lattice. Fix a smooth bounded domain $D \subseteq \mathbb{R}^2$. Let $\varepsilon = \varepsilon_s > 0$ be a sequence such that $\varepsilon = o(s^{-1/2})$ as $s \to \infty$. Then, as $s \to \infty$,
\[ \mathbb{P} \left( \mathcal{N}  \text{ and } \mathcal{N}^\varepsilon    \text{ are } \varepsilon\text{-homeomorphic in } sD \text{ up to visible ambiguities}  \right) \to 1 . \]
\end{theorem}

\begin{remark}
 \label{r:dynamic}
Theorem \ref{thm:main3} demonstrates that the nodal set of the random plane wave can be reconstructed using far fewer observations of (the sign of) the field than are needed in the general case. In particular, if the mesh is allowed to be chosen \textit{dynamically}, then roughly $s^3$ observations suffice to reconstruct the nodal set in a domain of radius $s$ (i.e.\ the $s^{3+\delta}$ vertices of a fixed mesh on the scale $\varepsilon = s^{-1/2-\delta/2}$, then a further $s^2$ vertices chosen dynamically to resolve, with a mesh on the scale $\varepsilon = s^{-1}$ (c.f.\ Theorem \ref{thm:main2}), the roughly $s$ visible ambiguities that arise). This can be contrasted with the roughly $s^4$ points needed to reconstruct a level set in the general case under the scheme in Theorem~\ref{thm:main2}. It is an interesting question whether a dynamic scheme might be able to lower this bound in the general case.
\end{remark}

\begin{remark}
\label{r:heurrpw}
We briefly explain why the nodal set of the random plane wave has discretisation errors that are easily identifiable as potential ambiguities. Consider that, for a general smooth function, errors in the discretisation of its nodal set are likely to result from either small nodal domains or low-lying saddle points. For the random plane wave, however, small nodal domains are deterministically ruled out. Moreover, because the wave satisfies the Helmholtz equation~\eqref{eq:helm}, any low-lying saddle point has principal axes that intersect at approximately right-angles, unless both eigenvalues of the Hessian $\nabla^2 \Psi$ at the saddle point are very small (which is highly unlikely due to eigenvalue repulsion, see Section \ref{sec:small}). This ensures that a low-lying saddle, and hence a possible discretisation error, is very likely to be easily identifiable by observing the field at nearby lattice vertices. 

Observe that the situation is different for other level lines, since there will be a much larger range of angles that will be attained by saddles lying close to this level. In particular, the angle between the principal axes of the saddle point will be close to $0$ whenever \textit{one} of the eigenvalues of the Hessian $\nabla^2 \Psi$ at the saddle point, rather than \textit{both} of them, is small. This is a much more likely event, and as a result no improvement can be made over the discretisation scheme in Theorem \ref{thm:main2}.
\end{remark}

\begin{remark}
The exact property of regular hexagonal lattices that is crucial to our result is stated in Lemma \ref{lem:hex}. To give a brief description, what is important is that a cone at a near-right-angle must induce an error pattern somewhere in the lattice (i.e.\ defining the percolation process via membership of the interior/exterior of the cone). Note that this property does not hold for, say, the square lattice, since if the cone was centred at the exact centre of a face, and was not exactly right-angled, then it may not induce an error pattern anywhere in the lattice; see Figure \ref{fig:hex}.
\end{remark}

\begin{remark}
\label{rem:anylattice}
Although we do not establish it rigorously, it is easy to see from our proofs that a similar result to Theorem \ref{thm:main3} should hold for any periodic lattice $\mathcal{L}$, only with a different, perhaps more complicated, definition of the Type $2$ error patterns used to identify the set of visible ambiguities in collections of nearby cells (in all cases, Type $1$ error patterns are defined identically; this only has content if $\mathcal{L}$ is not a triangulation). To illustrate, let us give suitable definitions of Type $2$ error patterns for the square lattice and the regular triangular lattice -- i.e.\ the lattice in which each face is an equilateral triangle.

For the square lattice, a set of four horizontally or vertically adjacent faces $f_1, f_2, f_3, f_4 \in \varepsilon \mathcal{F}$ display a Type $2$ error patten if: 
\begin{enumerate} 
\item The edge that is common to the faces $f_2$ and $f_3$ has endpoints that have the opposite sign in the percolation process $\mathcal{P}^\varepsilon$; and
\item Both $f_1 \cup f_2$ and $f_3 \cup f_4$ have vertices of the opposite sign in $\mathcal{P}^\varepsilon$.
\end{enumerate}

For the regular triangular lattice, the set of Type $2$ error patterns can be defined on sets of $12$ faces by grouping the triangular faces into pairs of hexagons (in each of the three ways of doing this), and applying the same definition as for the hexagonal lattice.

In general, for each lattice $\mathcal{L}$ there exists a \textit{visibility parameter} $\mu > 0$ such that we may define a certain set of Type $2$ error patterns on collections of faces that lie inside balls of radius~$\mu$, for which the discretisation up to ambiguities holds by analogy to Theorem~\ref{thm:main3}.
\end{remark}

\begin{figure}[ht]
\begin{center}
\includegraphics[scale=1]{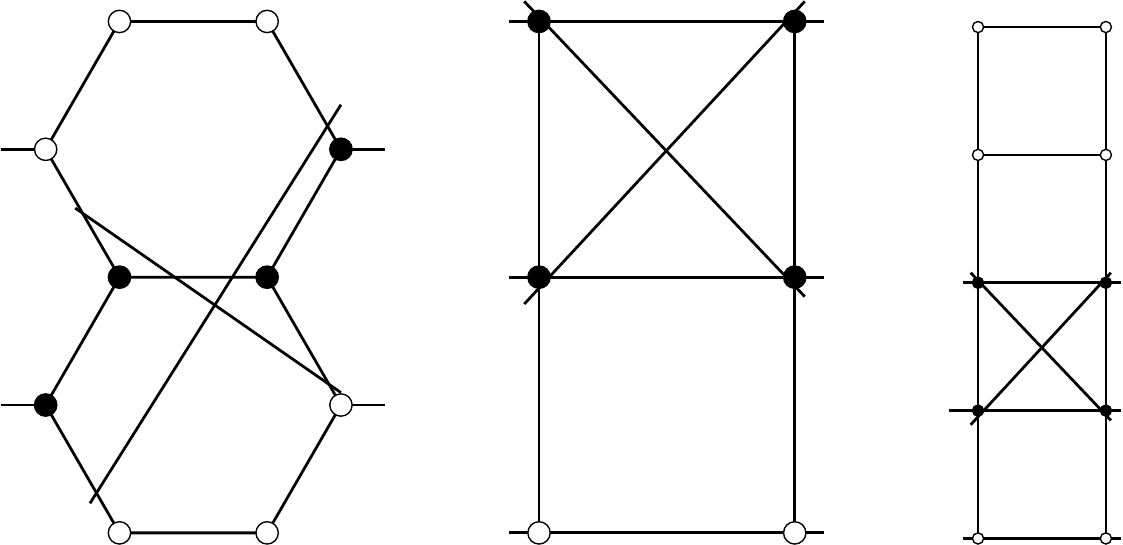}
\end{center}
\caption{A cone at near-right-angle induces an error pattern (in this case Type 2) on a regular hexagonal lattice (left), whereas it may not on a square or triangular lattice (centre), unless we widen the definition of Type 2 errors as described in Remark \ref{rem:anylattice} (right).}
\label{fig:hex}
\end{figure}

\subsubsection{Discretisation for estimating the Nazarov-Sodin constant}

The final discretisation scheme gives a way to estimate the constant $c_{NS} = c_{NS}(\kappa)$ defined in \eqref{eq:ns}. As discussed above, under some general conditions on $\kappa$ this constant is known to exist, but it has not yet been computed explicitly for any non-degenerate Gaussian field. Hence, there is much interest in developing ways of numerically approximating this constant. 

Our discretisation scheme gives a way to compare the number of nodal domains of planar Gaussian fields with the number of excursion domains induced by the discretised nodal set $\mathcal{N}^\varepsilon$. To this end, for each $\varepsilon > 0$ and bounded domain $D$, recall the definition of the $\varepsilon \mathcal{L}$-compatible set $P^\varepsilon(D)$, and let $N^{\varepsilon}(D)$ denote the number of components of the set $P^{\varepsilon}(D) \setminus \mathcal{N}^\varepsilon$.

\begin{theorem}[Discretisation for estimating the constant $c_{NS}$]
\label{thm:main4}
Suppose that $\kappa$ satisfies Assumption \ref{assumpt:degen}, and further assume that $c_{NS} = c_{NS}(\kappa)$ exists and is finite. Fix a smooth bounded domain $D \subseteq \mathbb{R}^2$. Let $\varepsilon = \varepsilon_s > 0$ be a sequence such that $\varepsilon = o(1)$ as $s \to \infty$. Then, as $s \to \infty$,
\begin{align}
\label{eq:nsapprox}
 \left|  \mathbb{E} [N^{\varepsilon}(sD) ] / \text{Area}(s D)  - c_{NS} \right| \to 0 . 
 \end{align}
\end{theorem}

\begin{remark}
For applications, it may be useful to have quantitative bounds on the rate of convergence in \eqref{eq:nsapprox}. The proof of Theorem~\ref{thm:main4} yields the following bound. For each $\delta \in (0, 1)$, there exists a constant $c = c(\kappa, \mathcal{L}, D, \delta) > 0$ such that, for each $s > 0$ and $\varepsilon \in (0, 1)$,
\[  \left| \mathbb{E} [N^{\varepsilon}(sD) ] - \mathbb{E}[N(sD)] \right| / \text{Area}(s D) < c \varepsilon^{2 - \delta} . \]
\end{remark}


\smallskip
\section{Overview of proofs and outline of paper}
\label{sec:overview}

The validity of our discretisation schemes are established by controlling certain `bad' events involving the level sets, each requiring the mesh-size to be set at a maximum level in order to rule them out with overwhelming probability. In this section we give a brief description of these events, and explain how we control them to prove the main results; this information is summarised in Table \ref{tab:1}.

The events that we control are the following (see Figure \ref{fig:badevents} below):
\begin{enumerate}
\item \textbf{Double-crossings} -- An edge $e \in \varepsilon \mathcal{E}$ is said to have a \textit{double-crossing} if $|\mathcal{N}_\ell \cap e| \ge 2$; 
\item \textbf{Small excursion domains} -- An excursion domain $S$ (i.e.\ a component of $\mathbb{R}^2 \setminus \mathcal{N}_\ell$) is called \textit{small} if its boundary $\partial S$ intersects at most one edge in $\varepsilon \mathcal{E}$;
\item \textbf{Four-crossings} -- A face $f \in \varepsilon \mathcal{F}$ is said to have a \textit{four-crossing} if four or more edges $e \in \partial f$ are crossed by $\mathcal{N}_\ell$ (this only has content if $\mathcal{L}$ is not a triangulation);
\item \textbf{Tubular-crossings} -- An edge $e \in \varepsilon \mathcal{E}$ is said to have a \textit{tubular-crossing} if, denoting by $f_1, f_2 \in \varepsilon \mathcal{F}$ the cells adjacent to $e$, there is a component of $(f_1 \cup f_2) \setminus \mathcal{N}_\ell$ that separates the endpoints of $e$ (this implies that $e$ has a double-crossing); and
\item \textbf{Invisible errors} (only defined for the regular hexagonal lattice) -- An edge $e \in \varepsilon \mathcal{E}$ is said to give rise to an \textit{invisible error} if it has a tubular-crossing and, denoting by $f_1, f_2 \in \varepsilon \mathcal{F}$ the cells adjacent to $e$, the faces $f_1$ and $f_2$ do not display a Type $2$ error pattern. In fact, for technical reasons we actually use a slightly modified version of `tubular-crossing' in our definition; we explain the necessary modification in Section~\ref{sec:proofs}.
\end{enumerate}

To emphasise the novelty of our approach, we remark that in the discretisation schemes developed in \cite{BG16} and \cite{MW07} only the first of these error events were controlled.

\begin{figure}[ht]
\begin{center}
\includegraphics[scale=1]{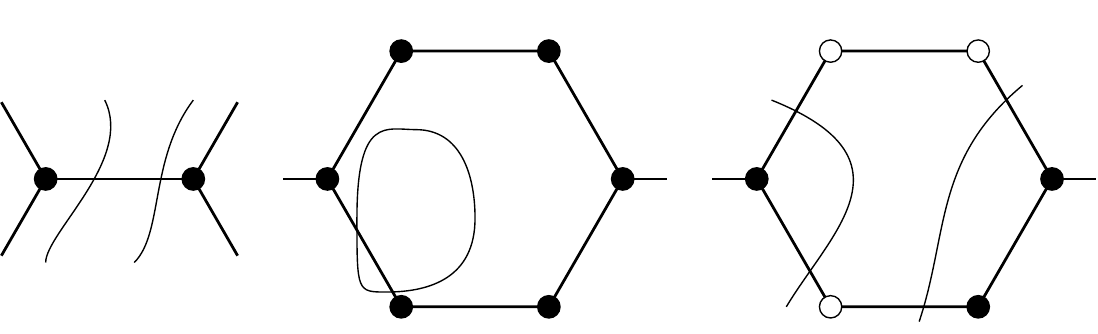}\\
\vspace{0.3cm}
\includegraphics[scale=1]{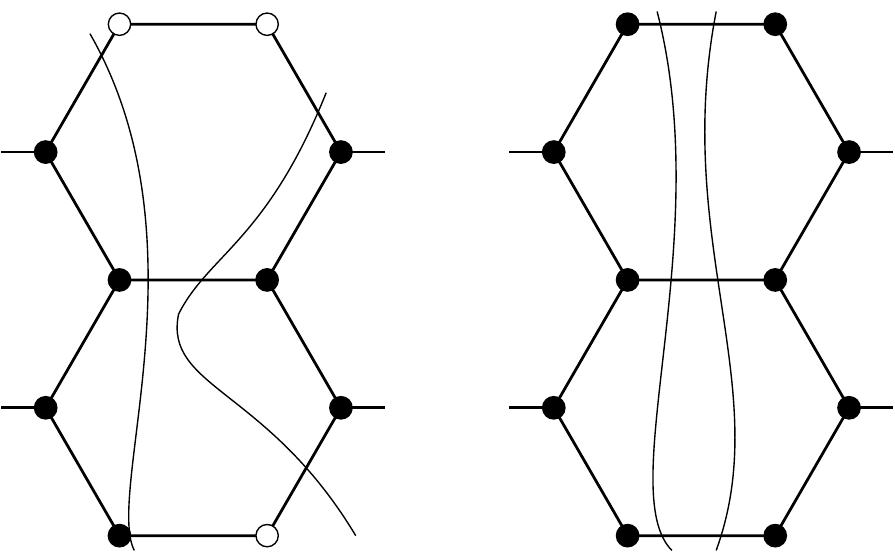}
\end{center}
\caption{An example of each of the `bad' events; from left-to-right and top-to-bottom, these depict a double-crossing, a small excursion domain, a four-crossing, a tubular-crossing and an invisible error. Here $\mathcal{L}$ is the hexagonal lattice.}
\label{fig:badevents}
\end{figure}

For the local discretisation scheme (Theorem \ref{thm:main1}), it is sufficient to control the first three events, since it is easy to see that $\mathcal{N}_\ell$ and $\mathcal{N}_\ell^{\varepsilon}$ being $\varepsilon$-homeomorphic in $f$ is equivalent to the  events (1)--(3) not occurring in $f$ or its boundary edges. To control these events inside $sD$ we must use a mesh-size $\varepsilon = o(s^{-2})$. This is essentially due to double-crossings; to control small excursion domains and four-crossings the mesh-size $\varepsilon = o(s^{-1-\delta})$ is sufficient.

For the global discretisation scheme (Theorem \ref{thm:main2}), we no longer seek to control double-crossings but instead use the following additional insight: a double-crossing that results in a discrepancy in the global topology of the level set must either (i) result in a tubular-crossing, or (ii) be located on the boundary $\partial P^\varepsilon(sD)$. Hence, in addition to small excursion domains and four-crossings, it is sufficient to control tubular-crossings, which require the mesh-size $\varepsilon = o(s^{-1-\delta})$, and double-crossings on the boundary, which require the mesh-size $\varepsilon = o(s^{-1/2})$.

For the global discretisation scheme up to visible ambiguities (Theorem \ref{thm:main3}), we use the fact that discrepancies in the global topology of the nodal lines must either be correctable by resolving a visible ambiguity (e.g.\ four-crossings that lead to Type $1$ errors, or tubular-crossings that lead to Type $2$ errors), or must be the result of invisible errors (note that small nodal domains are deterministically ruled out). Hence we only need to control double-crossings on the boundary and invisible errors, both of which requires the mesh-size $\varepsilon = o(s^{-1/2})$. Again, we note that we actually use a slight modification of the definition of tubular-crossings for this task; see Section \ref{sec:proofs}.

For the discretisation scheme for estimating the constant $c_{NS}$ (Theorem \ref{thm:main4}), we control the exact same events as for the global discretisation scheme (Theorem \ref{thm:main2}). This is natural, since the number of excursion domains is a function of the global topology. On the other hand, since our approximation is in expectation, we need to control the \textit{multiplicity} of these events, rather than simply their occurrence.

\begin{table}[h!]
  \centering
  \caption{Summary of `bad' events}
  \label{tab:1}
  \begin{tabular}{ccc}
    \textbf{Event} & \textbf{Schemes controlled by} & \textbf{Max.\! mesh-size} \\
    \toprule
    Double-crossings & Theorem \ref{thm:main1} &  $\varepsilon = s^{-2}$ \\
    \midrule
    Small excursion domains & Theorems \ref{thm:main1}, \ref{thm:main2} and \ref{thm:main4} & \multirow{3}{*}{$\varepsilon = s^{-1 - \delta}$}\\
    Four-crossings &  Theorems \ref{thm:main1}, \ref{thm:main2} and \ref{thm:main4}   \\
    Tubular-crossings & Theorems \ref{thm:main2} and \ref{thm:main4}   \\
    \midrule
    Double-crossings on boundary & Theorems \ref{thm:main2}, \ref{thm:main3} and \ref{thm:main4}  &  \multirow{2}{*}{$\varepsilon = s^{-1/2}$}\\
    Invisible errors  & Theorem \ref{thm:main3} 
  \end{tabular}
\end{table}

The rest of the paper is structured as follows. In Section \ref{sec:gauss}, we collect preliminary properties of Gaussian vectors and fields; most of these results are standard, but they will serve as a significant input into our proofs. In Section~\ref{sec:kacrice}, we undertake an analysis of events involving crossings of the level set and critical points near the level set, based around several Kac-Rice arguments; these arguments rely on certain matrix computations, and in order not to disrupt the flow of the paper we defer these computations to Appendix \ref{appendix2}. In Section \ref{sec:pert}, which is completely deterministic, we consider the effect of small perturbations on the nodal set of a generic function; the results in this section are only relevant to our analysis of the random plane wave, but are crucial in exploiting the extra control we have in that case, allowing us to reduce the study of the global topology of the nodal set to the study of invisible errors. Finally, in Section~\ref{sec:proofs}, we combine the above ingredients to establish the main results.

\smallskip

\section{Gaussian estimates}
\label{sec:gauss}

In this section we collect some preliminary Gaussian estimates. 

\subsection{Basic properties of Gaussian vectors}
In the first part, we state some basic properties of Gaussian vectors. Proposition \ref{prop:cond} is standard, whereas Lemmas \ref{lem:prod} and \ref{lem:gau} follow easily from H\"{o}lder's inequality and the compactness of the set of normalised covariances matrices.

\begin{proposition}[Gaussian regression]
\label{prop:cond}
Let $(X, Y)$ be a centred Gaussian vector of dimension $m+n$ with covariance 
\[ \Sigma = \left[  
\begin{array}{cc}
\Sigma_X & \Sigma_{XY} \\
\Sigma_{XY}^T & \Sigma_Y \\
\end{array} \right] .\]
Suppose that $\Sigma_Y$ is invertible, and fix a vector $\mathbf{y} \in \mathbb{R}^n$. Then, conditionally on $Y = \mathbf{y}$, the vector~$X$ is Gaussian with 
\[\mathbb{E}[X | Y = \mathbf{y}] = \Sigma_{XY} \Sigma_Y^{-1} \mathbf{y} \quad \text{and} \quad  \text{Cov}[ X | Y = \mathbf{y}] =  \Sigma_X - \Sigma_{XY} \Sigma_Y^{-1} \Sigma_{XY}^T .\]
\end{proposition}

\begin{lemma}[Bound on the product of Gaussian random variables]
\label{lem:prod}
Let $n \in \mathbb{N}$. Then there exists a $c := c(n) > 0$ such that, for each Gaussian vector $X := (X_1, \ldots, X_n)$ with mean $\mu = (\mu_i)_{1 \le i \le n}$ and covariance $\Sigma := (s_{ij})_{1 \le i,j \le n}$, it holds that
\[ \mathbb{E} \bigg[  \prod_{1 \le i \le n} \left| X_i \right|  \bigg] < c \max_{1 \le i \le n}  \max\{ |\mu_i|^n , s_{ii}^{n/2} \} .\]
\end{lemma}

\begin{lemma}[Bound on the maximum of a Gaussian vector under conditioning]
\label{lem:gau}
Let $n \in \mathbb{N}$ and $\ell \in \mathbb{R}$. Then there exists a $c := c(n, \ell) > 0$ such that, for each centred Gaussian vector $(X, Y)$ of dimension $n+1$ with covariance $\Sigma := (s_{ij})_{1 \le i,j \le n+1}$, it holds that, for each $\delta > 0$,
\[   \mathbb{E}\left[  \max_{1 \le i \le n} |X_i|^2  \id_{\{|Y - \ell| < \delta\}}  \right] <  c \, \delta \max_{1 \le i \le n} s_{ii}.\]
\end{lemma}

\subsection{Gaussian fields and their derivatives}
\label{sec:gaussfields}
We now state some properties of Gaussian fields and their derivatives.  Let $\Psi : \mathbb{R}^2 \to \mathbb{R}$ be a stationary, normalised Gaussian field with correlation kernel $\kappa: \mathbb{R}^2 \to [-1, 1]$. The following lemmas are standard; see \cite{Adler, AW}.

\begin{lemma}[Smoothness and non-degeneracy]
\label{lem:reg}
The following hold:
\begin{enumerate}
\item Suppose that $\kappa$ is $C^{2k}$ at the origin for $k \in \mathbb{N}$. Then almost surely $\Psi$ is $k$-times differentiable and $(k-1)$-times continuously differentiable. 
\item Suppose that $\kappa$ is $C^{2k}$ at the origin and, for unit vectors $v_1, \ldots v_k \in (S^1)^k$,
\[ (-1)^{k} \kappa_{v_1 v_1 \ldots v_k v_k}(0) > 0 .  \]
Then $\Psi_{v_1 \ldots v_k}(0)$ has a non-degenerate Gaussian distribution.
\item Suppose that $\kappa$ satisfies Assumption \ref{assumpt:degen}. Then, for any $s \in \mathbb{R}^2$, $\nabla \Psi(s)$ has a non-degenerate Gaussian distribution.
\item Fix a level $\ell \in \mathbb{R}$ and suppose that $\kappa$ satisfies Assumption \ref{assumpt:degen}. Then almost surely the components of the level set $\mathcal{N}_\ell$ are simple closed curves.
\end{enumerate}
\end{lemma}

\begin{lemma}[Derivatives of different parity are independent]
\label{lem:oddeven}
Suppose that $\kappa$ is $C^{2k}$ at the origin. Then, for any $m \in 2\mathbb{N}$ and $n \in 2\mathbb{N} + 1$ such that $m+n \le k$, and any unit vectors $v_1, \ldots , v_m, w_1, \ldots w_n \in (S^1)^{m+n}$,
\[  \mathbb{E}\left[ \Psi_{v_1\ldots v_m}(0)  \Psi_{w_1 \ldots w_n}(0) \right]  = 0.  \]
If $m=0$, the above result holds with the convention $\Psi_{v_1 \ldots v_m}(0) := \Psi(0)$.
\end{lemma}

\begin{lemma}[Covariances between the field and its derivatives]
\label{lem:cov}
Let $s \in \mathbb{R}$ be a point and $v \in S^1$ a unit vector. Suppose that $\kappa$ is $C^2$ at the origin and $C^1$ at $s$. Then,
\[  \mathbb{E}\left[ \Psi(s) \Psi_v(0) \right]  = -\kappa_v(s)    \qquad \text{and} \qquad    \mathbb{E}\left[\Psi_v(0) \Psi_v(0) \right]  = -\kappa_{vv}(0) . \]
\end{lemma}

\begin{lemma}[Derivative bounds]
\label{lem:derbounds}
Suppose that $\kappa$ is $C^{2k}$ at the origin for $k \in \mathbb{N}$. Fix $r > 0$. Then there exists a $c = c(\kappa, k, r) > 0$ such that 
\begin{equation}
\label{eq:der}
 \mathbb{E}[ \|\Psi\|_{C^k(B(r))} ]   < c.
\end{equation}
\end{lemma}

The following proposition is less standard; we give a full proof in Appendix~\ref{appendix1}. Note that in many important cases -- e.g.\ isotropic $\kappa$ or $\kappa \in L_1(\mathbb{R}^2)$ -- the proposition follows immediately from the resulting strict positive-definiteness of $\kappa$ (see \cite[Theorem 3.8]{Sun93} and \cite[Theorem 6.11]{Wendland} respectively, where for the latter also note that, under Assumption \ref{assumpt:degen}, $\kappa$ is continuous).

\begin{proposition}[Local three-point non-degeneracy]
\label{prop:three}
Suppose that $\kappa$ satisfies Assumption~\ref{assumpt:degen}. Then there exists a $\delta > 0$ such that, for any distinct $s = (s_i)_{1 \le i \le 3} \in (\mathbb{R}^2)^3$ satisfying $\max_{i,j} |s_i - s_j| < \delta$ and that are not co-linear, the distribution of $\Psi(s)$ is non-degenerate.
\end{proposition}

\subsection{Small ball estimate for derivatives of the random plane wave}
\label{sec:small}
To close this section, we give a small ball estimate for the first and second derivatives of the random plane wave; this will be necessary in completing the proof of Theorem \ref{thm:main3}.

Let $\Psi: \mathbb{R}^2 \to \mathbb{R}$ be the random plane wave, i.e.\ the stationary, normalised Gaussian field whose correlation kernel satisfies \eqref{eq:rpw} for frequency parameter $k > 0$. Let $\lambda_1$ and $\lambda_2$ denote the eigenvalues of the Hessian matrix $\nabla^2 \Psi(0)$. The following proposition essentially expresses the order-one repulsion between the eigenvalues $\lambda_1$ and $\lambda_2$; the existence of this repulsion is not restricted to the plane wave (it holds for any sufficiently smooth isotropic Gaussian field) but we state it only in this case. We give a full proof in Appendix \ref{appendix1}.

\begin{proposition}[Small ball estimate for derivatives of the random plane wave]
\label{prop:smalldevest}
There exists a $c = c(k) > 0$ such that, for each $\delta_1, \delta_2, \delta_3 > 0$,
\[ \mathbb{P} \left(  |\nabla \Psi(0)| < \delta_1, \max\{ |\lambda_1| , |\lambda_2| \} < \delta_2 , |\lambda_1 + \lambda_2| < \delta_3  \right)  < c \delta_1 \delta_2^3 \delta_3 .\]
\end{proposition}

\smallskip

\section{Kac-Rice arguments}
\label{sec:kacrice}

In this section we develop several Kac-Rice arguments which underpin the validity of our discretisation schemes. Throughout this section let $\Psi : \mathbb{R}^2 \to \mathbb{R}$ be a stationary, normalised Gaussian field with correlation kernel $\kappa : \mathbb{R}^2 \to [-1, 1]$ satisfying Assumption \ref{assumpt:degen}, and let $\ell \in \mathbb{R}$ be a fixed level. By Lemma \ref{lem:reg}, this implies that the components of the level set~$\mathcal{N}_\ell$ are almost surely simple closed curves, and also that the distribution of $\nabla \Psi(s)$ is non-degenerate for any $s \in \mathbb{R}^2$.

For $k \in \mathbb{N}$ and a collection of points $s = (s_1, \ldots, s_k) \in (\mathbb{R}^2)^k$, let $\theta^-(s)$ and $\theta^+(s)$ denote, respectively, the size of the smallest and largest interior angle among all possible triangles formed by points in $s$; if $k \le 2$ or the points in $s$ are not distinct, set $\theta^-(s) := 0$ and $\theta^+(s) := \pi$. The variable $\theta^-(s)$ gives a quantitative measure of the degeneracy of $s$, and in addition $\theta^+(s)$ identifies degeneracy that is due to co-linearity.
\subsection{Kac-Rice formulae}

Kac-Rice formulae are a general way to express expectations involving the zeros of a random process in terms of certain integrals. The exact form of a given Kac-Rice formula will depend on the precise quantity that is desired to be computed. We state two Kac-Rice formulae, one for controlling the intersection of the level set with line segments, and another for controlling the set of critical points; since these are standard, we omit the proof.

\begin{proposition}[One-dimensional Kac-Rice formula]
\label{prop:1dkr}
Fix $k \in \mathbb{N}$ and $\mu^-, \mu^+ \in [0, \pi)$. Let $L := (L_i)_{1 \le i \le k}$ be a set of line segments in $\mathbb{R}^2$. For $1 \le i \le k$, let $v_i \in S^1$ be a unit vector in the direction of $L_i$. Suppose that the distribution of $\Psi(s)$ is non-degenerate for each $s = (s_1, \ldots , s_k) \in L$ with $\{s_i\}_{1 \le i \le k}$ distinct and satisfying $\theta^-(s) \ge \mu^-$ and $\theta^+(s) \le \mu^+$. Define the random variable
\[ N := \left| \left\{  (s_1 \in L_1 \cap \mathcal{N}_\ell, \ldots , s_k \in L_k \cap \mathcal{N}_\ell) :  \{s_i\}_{1 \le i \le k} \text{ distinct, } \theta^-(s) \ge \mu^-,  \theta^+(s) \le \mu^+  \right\}  \right|  . \]
Then
\[  \mathbb{E} \left[ N \right] = \int_{s \in L} \varphi_s(\ell)  \,  \mathbb{E} \left[ \prod_{1 \le i \le k}  \left|  \Psi_{v_i}(s_i) \right| \, \Big|  \, \Psi(s_1) = \ldots = \Psi(s_k) = \ell  \right]  \id_{ \{  \theta^-(s) \ge \mu^-, \, \theta^+(s) \le \mu^+  \} } \, ds, \]
where $\varphi_s(\ell)$ denotes the density at $\ell$ of the $k$-dimensional Gaussian vector $\Psi(s)$.
\end{proposition}

\begin{proposition}[Two-dimensional Kac-Rice formula]
\label{prop:2dkr}
Fix $c > 0$ and a bounded domain $D \subseteq \mathbb{R}^2$. Let $\mathcal{C}$ be the set of critical points of $\Psi$, i.e.\ $\mathcal{C} := \{ s \in \mathbb{R}^2 : |\nabla \Psi(s)| = 0 \}$. Define the random variable
\[ N := \left| \left\{  s \in D \cap \mathcal{C} : |\Psi(s) - \ell| < c  \right\}  \right|  . \]
Then
\[  \mathbb{E} \left[ N \right] = \int_{s \in D} \varphi_s(0)  \,  \mathbb{E} \left[  \left| \nabla^2 \Psi (s) \right| \id_{\{|\Psi(s) - \ell| < c\}} \Big|  \left|\nabla \Psi(s) \right| =  0  \right]  \, ds , \]
where $\varphi_s(0)$ denotes the density at zero of the two-dimensional Gaussian vector $\nabla \Psi(s)$.  
\end{proposition}

\subsection{Level set crossings}
We first use the one-dimensional Kac-Rice formula to give bounds for events involving the intersection of level sets with line segments. 

\subsubsection{Single crossings}
Let $L$ be a line segment in $\mathbb{R}^2$. For each $\varepsilon > 0$, let $N^\varepsilon$ denote the number of intersections of $\mathcal{N}_\ell$ with the scaled line segment $\varepsilon L$, i.e.\ $N^\varepsilon := \left| \left\{ x \in L : \varepsilon x \in \mathcal{N}_\ell   \right\} \right|$.

\begin{proposition}[Single crossing of one line segment]
\label{prop:kr1}
There exists a $c = c(\kappa, \ell) > 0$ such that, for each $\varepsilon > 0$,
\[ \mathbb{E} \left[ N^\varepsilon  \right]  = c |L| \varepsilon .\]
\end{proposition}
\begin{proof}
Let $v \in S^1$ be a unit vector in the direction of the line segment $L$. By the one-dimensional Kac-Rice formula in Proposition \ref{prop:1dkr} (setting $k := 1, \mu^- := 0, \mu^+ := \pi, L_1  := L$),
\begin{align*}
\mathbb{E} \left[ N^\varepsilon  \right]    = \varepsilon (2 \pi)^{-1/2}  e^{-\ell^2} \int_{s  \in L}  \,  \mathbb{E} \left[  \left|  \Psi_v(\varepsilon s) \right|  \big|   \Psi(\varepsilon s) = \ell  \right]  \, ds .
\end{align*}
Applying Lemma \ref{lem:oddeven}, $\Psi_v(\varepsilon s)$ is independent of $\Psi(\varepsilon s)$. Moreover, by stationarity, the distribution of $\Psi_v(x)$  is identical for each $x$. Hence 
\[  \mathbb{E} \left[ N^\varepsilon  \right]    = \varepsilon (2 \pi)^{-1/2} e^{-\ell^2}  \,|L|  \, \mathbb{E} \left[  \left|  \Psi_v(0) \right|    \right]  ,\]
which proves the result.
\end{proof}

\subsubsection{Double-crossings}

Fix $d > 0$ and let $L := (L_i)_{i=1,2}$ be a pair of line segments in $\mathbb{R}^2$ contained within the ball $B(d)$, with $m:=\max_{i = 1,2} |L_i|$. For each $\varepsilon > 0$, define the random variable
\[ N^\varepsilon := \left| \left\{ (s_1 \in \varepsilon L_1 \cap \mathcal{N}_\ell, s_2 \in \varepsilon L_2 \cap \mathcal{N}_\ell) :  s_1 \neq s_2 \right\} \right| .\]

\begin{proposition}[Double-crossing of two line segments]
\label{prop:kr2}
There exists a $c = c(\kappa, \ell, d) > 0$ such that, for each $\varepsilon \in (0, 1)$,
\[ \mathbb{E} \left[ N^\varepsilon \right]  < c m^2  \varepsilon .\]
Moreover, if $L_1$ and $L_2$ lie on a common line, the bound can be replaced with $c m^2 \varepsilon^3$.
\end{proposition}

\begin{proof}
Remark that, by a linear rescaling of the plane, we may reduce to the case that $\kappa_{vv}(0) = -2$ for all unit vectors $v \in S^1$; this simplification will be useful when we appeal to the computation in Appendix~\ref{appendix2}. We may also choose $\varepsilon$ sufficiently small so that, for each $s = (s_1, s_2) \in \varepsilon L$ with $s_i$ distinct, the distribution of $\Psi(s)$ is non-degenerate; this is possible since $\kappa_{vv}(0) \neq 0$.

Let $v_1$ and $v_2$ be unit vectors in the direction of the line segments $L_1$ and $L_2$ respectively. By the one-dimensional Kac-Rice formula in Proposition \ref{prop:1dkr} (setting $k := 2,  \mu^- := 0, \mu^+ := \pi$) and using the fact that the density of a two-dimensional Gaussian vector with covariance matrix $\Sigma$ is at most $(2 \pi)^{-1} |\Sigma|^{-1/2}$,
\begin{align*}
&  \mathbb{E} \left[ N^\varepsilon   \right]    < \varepsilon^2 (2 \pi)^{-1}  \int_{s = (s_1, s_2) \in L_1 \times L_2 } \! \! \! \! \! \! \! \! |A^{\varepsilon, s}|^{-\frac12}  \,  \mathbb{E} \left[  \left|  \Psi_{v_1}(\varepsilon s_1) \Psi_{v_2}(\varepsilon s_2) \right|  \big|   \Psi(\varepsilon s_1) = \Psi(\varepsilon s_2) = \ell   \right]  \, ds_1 \, ds_2 ,
\end{align*}
where  
\[ A^{\varepsilon, s} :=  \left[ \begin{array}{cc}
1 & \kappa(\varepsilon (s_1-s_2)) \\
\kappa(\varepsilon (s_1-s_2)) & 1 \\
\end{array} \right] . \]
Fix for a moment $s = (s_1, s_2) \in L^2$ such that $s_1 \neq s_2$. Applying Proposition \ref{prop:cond} and Lemma~\ref{lem:cov}, conditionally on $\Psi(\varepsilon s_1) = \Psi(\varepsilon s_2) = \ell$ the vector
\[  \left( \Psi_{v_1}(s_1) , \Psi_{v_2}(s_2) \right) \]
has a Gaussian distribution with mean $\ell \mu^{\varepsilon, s}$, where
\[ \mu^{\varepsilon, s}  := (\mu^{\varepsilon, s} _i)_{i = 1,2} =  B^{\varepsilon, s} (A^{\varepsilon, s})^{-1}  (1, 1)^T , \]
and covariance
\[   D^{\varepsilon, s} = (d^{\varepsilon, s}_{ij})_{1 \le i ,j \le 2} := C - B^{\varepsilon, s} (A^{\varepsilon, s})^{-1} (B^{\varepsilon, s})^T ,  \]
for
\[  B^{\varepsilon, s} :=  \left[ \begin{array}{cc}
0 & \kappa_{v_1}(\varepsilon (s_1 - s_2)) \\
- \kappa_{v_2}(\varepsilon (s_1 - s_2))  & 0 \\
\end{array} \right] \]
and $C = (c_{ij})_{1 \le i,j \le 2}$ satisfying $c_{ii} = 2$. Applying Lemma \ref{lem:prod}, we deduce that there exists a $c_0 = c_0(\ell) > 0$ such that
\[  |A^{\varepsilon, s}|^{-\frac12}  \,  \mathbb{E} \left[  \left|  \Psi_v(\varepsilon s_1) \Psi_v(\varepsilon s_2) \right| \big|   \Psi(\varepsilon s_1) = \Psi(\varepsilon s_2) = \ell   \right]  <  c_0 \left( \frac{ \max_{i}  \max\{ (\mu^{\varepsilon, s}_i)^4, (d^{\varepsilon, s}_{ii})^2 \} }{|A^{\varepsilon, s}| }\right)^{\frac12} . \]

We first treat the case that $L_1$ and $L_2$ do not lie on a common line. Observe that, since $A^{\varepsilon, s}$ and $B^{\varepsilon, s}$ are positive definite, $d^{\varepsilon, s}_{ii} \le 2$. Moreover, by Taylor's theorem, there exists a constant $c_1 = c_1(\kappa, d) > 0$ such that 
\[ \kappa_{v_1}(\varepsilon (s_1 - s_2)) < c_1 \varepsilon |s_1 - s_2| .\]
Computing $B^{\varepsilon, s} (A^{\varepsilon, s})^{-1} (1, 1)^T$ and $|A^{\varepsilon, s}|$ explicitly, there exists a $c_2 = c_2(\kappa, d) > 0$ such that
\[ \left( \frac{ \max_{i}  \max\{ (\mu^{\varepsilon, s}_i)^4, (d^{\varepsilon, s}_{ii})^2 \} }{|A^{\varepsilon, s}| }\right)^{\frac12} \le  \frac{  \max\left\{ (c_1 \varepsilon)^2 \left( 1 - \kappa(\varepsilon (s_1-s_2))   \right)^2 , 2 \right\} }{ \left( 1 - \kappa(\varepsilon (s_1-s_2))^2 \right)^{\frac12} } < \frac{c_2}{ \varepsilon |s_1 - s_2| } ,  \]
where the last inequality holds by Taylor's theorem. Integrating over $s \in L_1 \times L_2$, we have the result. 

Turning now to the case that $L_1$ and $L_2$ lie on a common line, we may choose $v_1 = v_2 =: v$. Remark that, by Taylor's theorem, there exists a constant $c_1 = c_1(\kappa, d) > 0$ (i.e.\ independent of $s$) and a $C^1$ function $f^s : \mathbb{R}^+ \to \mathbb{R}$ satisfying, for each $x \in (0, 1)$,
\[  \left| f^s(x) - 1 + x^2 \right| < c_1 x^4  ,  \quad  \left| (f^s)'(x) - 2 x \right| < c_1  x^3,   \]
such that
\[  \kappa(\varepsilon (s_1-s_2))  =  f^s(\varepsilon |s_1 - s_2| )   \qquad \text{and} \qquad  \kappa_v(\varepsilon (s_1 - s_2))  =  (f^s)'(\varepsilon  |s_1 - s_2|  )  .\]
Applying the computation in Proposition \ref{prop:append1}, there exists a $c_2 = c_2(\kappa,  d) > 0$, such that, for each $\varepsilon \in (0, 1)$,
\[ \left( \frac{ \max_{i}  \max\{ (\mu^{\varepsilon, s}_i)^4, (d^{\varepsilon, s}_{ii})^2 \} }{|A^{\varepsilon, s}| }\right)^{\frac12}   < c_2 \varepsilon .\]
Integrating over $s \in L$, we have the result.
\end{proof}

\subsubsection{Triple-crossings}
Fix $d > 0$ and an angle $\mu \in [0, \pi)$. Let $L := (L_i)_{1 \le i \le 3}$ be a triple of line segments in $\mathbb{R}^2$ contained within the ball $B(d)$. For each $\varepsilon > 0$, define the random variable
\[ N^\varepsilon := \left| \left\{ (s_1 \in \varepsilon L_1 \cap \mathcal{N}_\ell, \ldots , s_3 \in \varepsilon L_3 \cap \mathcal{N}_\ell) : \{s_i\}_{1 \le i \le 3} \text{ distinct, } \theta^-(s) \ge \varepsilon^{3/2}, \theta^+(s) \le \mu \right\} \right| .\] 
Remark that $N^\varepsilon$ counts triple-crossings that are not too degenerate; such triple-crossings pose problems for our asymptotic analysis, and we prefer to deal with these in other ways. 

\begin{proposition}[Triple-crossing of three line segments]
\label{prop:kr3}
Fix $\delta > 0$. Then there exists a $c = c(\kappa, \ell, d, \mu, L, \delta) > 0$ such that, for each $\varepsilon \in (0, 1)$,
\[ \mathbb{E} \left[  N^\varepsilon  \right]  < c  \varepsilon^{4 - \delta} .\] 
\end{proposition}

\begin{remark}
Unlike the condition $\theta^+(s) \le \mu$, which is crucial, we do not believe that the constraint $\theta^-(s) \ge \varepsilon^{3/2}$ is necessary for Proposition \ref{prop:kr3} to hold; we impose it because it greatly simplifies the proof (see in particular the computation in Proposition \ref{prop:append2}). On the other hand, the result would be too weak for our purposes if we chose any stronger constraint; see Proposition~\ref{prop:contriple}. 

\end{remark}
\begin{proof}
As in the proof of Proposition \ref{prop:kr2}, by a linear rescaling of the plane we may reduce to the case that $\kappa_{vv}(0) = -2$ for all unit vectors $v \in S^1$. We may also choose $\varepsilon$ sufficiently small so that, for each $s = (s_1, s_2, s_3) \in \varepsilon L$  with $\{s_i\}_{1 \le i \le k}$ distinct and not co-linear, the distribution of $\Psi(s)$ is non-degenerate; this is possible by Proposition \ref{prop:three}.

Let $v_1, v_2$ and $v_3$ be unit vectors in the direction of the line segments $L_1, L_2$ and $L_3$ respectively. By the one-dimensional Kac-Rice formula in Proposition \ref{prop:2dkr} (setting $k := 3, \mu^- := \varepsilon^{3/2}, \mu^+ := \mu$) and using the fact that the density of a three-dimensional Gaussian vector with covariance matrix $\Sigma$ is at most $(2 \pi)^{-3/2} |\Sigma|^{-1/2}$,
\begin{align*}
  \mathbb{E} \left(  N^{\varepsilon}  \right)   < (2 \pi)^{-\frac32}  \varepsilon^3  \int_{s \in L} \! \! |A^{\varepsilon, s}|^{-\frac12}  \,  \mathbb{E} \bigg[  \prod_{1 \le i \le 3}  \left| \Psi_{v_i}(\varepsilon s_i) \right|  \big|  \Psi(\varepsilon s_i) = \ell \bigg] \mathbf{1}_{\{  \theta^-(s) \ge  \varepsilon^{3/2} ,  \theta^+(s) \le \mu \}}   \, ds ,
\end{align*}
where
\[ A^{\varepsilon, s} :=  \left( \kappa(\varepsilon (s_i - s_j) ) \right)_{1 \le i,j \le 3} .\]
Fix for a moment $s = (s_1, s_2, s_3)$ with $s_i$ are distinct. Applying Proposition \ref{prop:cond} and Lemma~\ref{lem:cov}, conditionally on $\Psi(\varepsilon s_1) = \Psi(\varepsilon s_2) = \Psi(\varepsilon s_3) = \ell$ the vector
\[  \left( \Psi_{v_1}(s_1) ,  \Psi_{v_2}(s_2) ,  \Psi_{v_3}(s_3)  \right) \]
has a Gaussian distribution with mean $\ell \mu^{\varepsilon, s} $, where
\[ \mu^{\varepsilon, s}  := (\mu^{\varepsilon, s} _i)_{1 \le i \le 3} = B^{\varepsilon, s} (A^{\varepsilon, s})^{-1} (1, 1, 1)^T, \]
and covariance
\[   D^{\varepsilon, s} = (d^{\varepsilon, s}_{ij})_{1 \le i ,j \le 3} := C -  B^{\varepsilon, s} (A^{\varepsilon, s})^{-1} (B^{\varepsilon, s})^T  ,  \]
for
\[ B^{\varepsilon, s} := ( \kappa_{v_j}(\varepsilon (s_i - s_j))  )_{1 \le i,j \le 3}  \]
and $C = (c_{ij})_{1 \le i, j \le 3}$ satisfying $c_{ii} = 2$. Applying Lemma \ref{lem:prod}, we deduce that there exists a $c_0 = c_0(\ell) > 0$ such that,
\[     |A^{\varepsilon, s}|^{-\frac12}  \,  \mathbb{E} \bigg[  \prod_{1 \le i \le 3}  \left| \Psi_{v_i}(\varepsilon s_i) \right|  \big|  \Psi(\varepsilon s_i) = \ell \bigg]     \le c_0  \left( \frac{ \max_i  \max\{ (\mu^{\varepsilon, s}_i)^6,   (d^{\varepsilon, s, v}_{ii })^3 \}  }{|A^{\varepsilon, s}|} \right)^{\frac12}. \]
For $1 \le i \neq j \le 3$, let $\eta_{ij}$ be the angle formed between the line segment $s_i s_j$ and the vector~$v_j$.  Remark that, by Taylor's theorem, there exist constants $c_1 = c_1(\kappa, d) > 0$ (i.e.\ independent of~$s$)$, e^s_{ij} = e^s_{ij}(\kappa)$ and functions $f^s_{ij}, g^s_{ij}: \mathbb{R}^+ \to \mathbb{R}$ satisfying, 
\[ |e^s_{ij} - e^s_{ik} | < c_1 |\eta_{ij} - \eta_{ik}| ,\]
and, for each $x \in (0, 1)$,
\[  \left| f^s_{ij}(x) - 1 + x^2 + e_{ij}^s x^4  \right| < c_1 x^6 \ , \quad f^s_{ij} = f^s_{ji}  \]
and
\[ \left| g^s_{ij}(x) +  \cos(\eta_{ij}) \left( 2x + 4 e_{ij}^s x^3 \right)   \right| < c_1 x^5   ,\]
such that
\[    \kappa(\varepsilon (s_i-s_j))   =  f^s_{ij}(\varepsilon |s_i-s_j| ) \qquad \text{and} \qquad  \kappa_{v_j}(\varepsilon (s_i - s_j))  = g^s_{ij}(\varepsilon |s_i - s_j|)   .\]
Applying the computation in Proposition \ref{prop:append2}, there exists a $c_2 = c_6(\kappa, d, \mu, L) > 0$ such that, for each $\varepsilon \in (0, 1)$,
\[ \left( \frac{ \max_i  \max\{ (\mu^{\varepsilon, s}_i)^6,   (d^{\varepsilon, s, v}_{ii })^3 \}  }{|A^{\varepsilon, s}|} \right)^{\frac12}  \times  \theta^{-}(s) < c_2 \varepsilon .\]
To complete the proof, notice that there exists a $c_3 = c_3(d, \delta) > 0$ such that, for $\varepsilon \in (0, 1)$,
\[ \int_{s \in L} \frac{1}{\theta^-(s)} \mathbf{1}_{\{\theta^-(s) >  \varepsilon^{3/2} \}}   \, ds  <  c_3 \varepsilon^{-\delta},  \]
and we have the result.
\end{proof}

\subsection{Critical points near a level}

We now use the two-dimensional Kac-Rice formula to give a bound on the number of critical points for which the $\Psi$ has a value that is near the level $\ell$. Let $D$ be a bounded domain. For each $s, \delta > 0$, define the random variable
\[  N^{s, \delta} := \left| \left\{ x \in sD : \nabla \Psi (x) = 0, |\Psi(x) - \ell| <  \delta \right \} \right| .\]

\begin{proposition}[Critical points near a level]
\label{prop:krmax}
There exists a $c = c(\kappa, D) > 0$ such that, for each $s > 0$ and $\delta > 0$, 
\[ \mathbb{E}\left[  N^{s, \delta}  \right] < c s^2 \delta .\]
\end{proposition}

\begin{proof}
Similarly to in the proofs of Proposition \ref{prop:kr2} and \ref{prop:kr3}, by a linear rescaling of the plane we may reduce to the case that $\kappa_{vv}(0) = -2$ for all unit vectors $v \in S^1$. By the two-dimensional Kac-Rice formula in Proposition \ref{prop:2dkr}, and using the fact that the density of a two-dimensional Gaussian vector with covariance matrix $\Sigma$ is at most $(2 \pi)^{-1} |\Sigma|^{-1/2}$,
\[ \mathbb{E} \left[ N^{s, \delta}  \right]  = s^2 (2 \pi)^{-1}  \int_{x \in D} |A|^{-1/2}  \,  \mathbb{E} \left[   |  \nabla^2 \Psi(s x)|  \id_{\{ |\Psi(x) - \ell| <  \delta \}} \big|   |\nabla \Psi(s x)|  = 0   \right]  \, dx ,\]
where  $A := 2 \id_{2 \times 2}$. Applying Lemma \ref{lem:oddeven}, the random vector $\nabla \Psi(s x)$ is independent of $(\Psi(s x), \nabla^2 \Psi (s x))$. Moreover, by stationarity, the distribution of $(\Psi(x), \nabla^2 \Psi (x))$ is identical for each $x$. Hence 
\[  \mathbb{E} \left[ N^{s, \delta}  \right]  =  \frac{s^2}{4 \pi}  \text{Area}(D) \times \mathbb{E} \left[    | \nabla^2 \Psi(0)|  \id_{\{ |\Psi(0) - \ell| < \delta \}}   \right]  .\]
Since $|\nabla^2 \Psi(0)| \le \max\{ |\Psi_{xx}(0)|, |\Psi_{yy}(0)|\}^2$, applying estimate Lemma \ref{lem:gau} gives the result. 
\end{proof}

\smallskip

\section{Perturbation analysis}
\label{sec:pert}

In this section we develop a perturbation analysis that will be crucial in our study of the nodal set of the random plane wave. This section is completely deterministic, and may be of independent interest, although a reader that is not concerned with the special case of the random plane wave may prefer to skip ahead.

Define a \textit{regular conic section} $\mathcal{C}$ to be any one of the following objects: (i) the empty set; (ii) a line; (iii) a circle; or (iv) a hyperbola whose principal axes intersect at right-angles. The aim of this section is to give sufficient criteria under which a solution to the Helmholtz equation~\eqref{eq:helm} has a nodal set that is locally well-approximated by a regular conic section.

First we define the sense in which we consider a set to be `well-approximated' by a regular conic section (the results that follow actually hold in a stronger sense, but this is sufficient for our purposes). For each $\delta > 0$ and set $\mathcal{S} \subseteq \mathbb{R}^2$, define the $\delta$\textit{-thickening of} $\mathcal{S}$ to be the closed set of all points within a distance $\delta$ from $\mathcal{S}$. For each $\delta > 0$ and domain $D \subseteq \mathbb{R}^2$, we say that a set $\mathcal{S} \subseteq \mathbb{R}^2$ satisfies the $\delta$\textit{-conic property in} $D$ if there exists a regular conic section $\mathcal{C}$ and a homeomorphism $h: D  \to D$ mapping $\mathcal{S} \cap D$ onto $ \mathcal{C} \cap D$ such that: (i) $\mathcal{S} \subseteq \mathcal{C}^\delta$; (ii) $|\mathcal{S} \cap \partial D| = |\mathcal{C} \cap \partial D|$; and (iii) $|h(s, 1) - s| < \delta$ for each $s \in  \mathcal{S} \cap \partial D$. 

See Figure~\ref{fig:cone} for an illustration of the $\delta$-conic property. Remark that it requires \textit{both} that~$\mathcal{S}$ lies inside a $\delta$-thickening of a regular conic section, and also that the boundary points of $\mathcal{S}$ are close to the boundary points of the conic section.

\begin{figure}[ht]
\begin{center}
\includegraphics[scale=1]{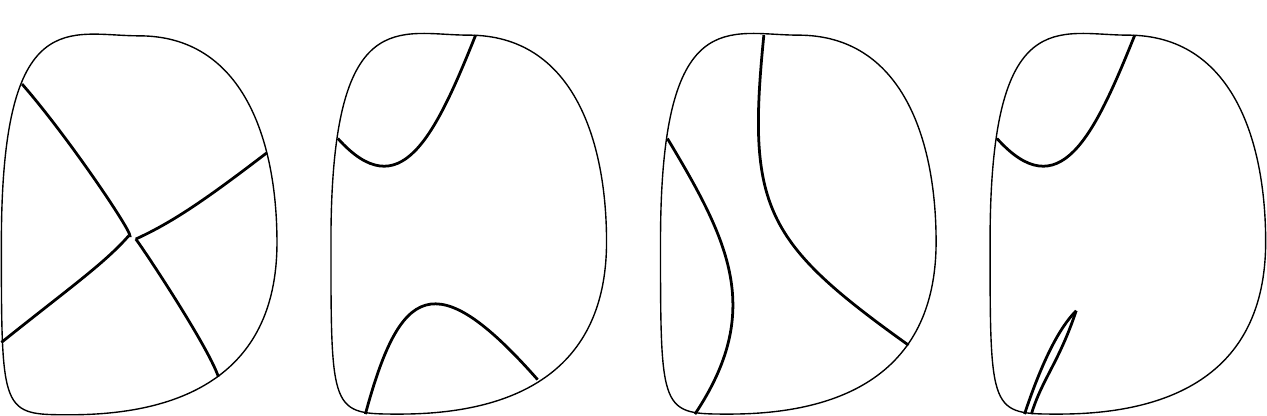}
\end{center}
\caption{Examples of sets $\mathcal{S}$ that satisfy the $\delta$-conic property in a domain $D$ for small $\delta> 0$ (the two figures on the left), and which do not satisfy the $\delta$-conic property in $D$ for small $\delta > 0$ (the two figures on the right; the first fails because $\mathcal{S}$ does not lie in the $\delta$-thickening of a regular conic section; the second fails because the boundary points $\mathcal{S}$ do not match the conic section.}
\label{fig:cone}
\end{figure}

\subsection{Perturbations of general functions}

In the next three lemmas, define functions $f, g: \mathbb{R}^2 \to \mathbb{R}$ and let $h:= f+g$, with $\mathcal{N}$ the nodal set of~$h$. 

\begin{lemma}[Nodal set of a perturbed linear function]
\label{lem:lin}
Suppose that $f$ is linear with $|\nabla f| = 1$, and let $\delta := \|g\|_{C^1(B(2))}$. Then there exists a $c > 0$ such that, for sufficiently small $\delta > 0$, either $D = B(1)$ or $D = B(1+c\delta)$ satisfy the $c \delta$-conic property.
\end{lemma}
\begin{proof}
Without loss of generality we may assume that $f(x, y) = y + d$ for some constant $d \ge 0$. Remark that the nodal set $\mathcal{N}$  is contained in the $\delta$-thickened line 
\[ \mathcal{C}^\delta := \{(x, y) : | y + d |  \le  \delta  \} .\]
If $d > 1 + \delta$ then $|\mathcal{C}^\delta \cap B(1)| = 0$, and so it suffices to consider the case $d \le 1 + \delta$ and to set $D := B(1+4\delta)$. Remark that, for sufficiently small $\delta$, the set $\mathcal{C}^\delta \cap \partial D$ consists of two disjoint arcs, and moreover
\[ \inf |  f'|_{\mathcal{C}^\delta \cap \partial D} | \ge |\nabla f| \sqrt{   1 - \left( \frac{1+2\delta}{1+4\delta} \right)^2  } > \delta \ge \|g\|_{C^1(\mathcal{C}^\delta \cap \partial D)}   . \]
Hence $h$ is not stationary along either of the arcs in $\mathcal{C}^\delta \cap \partial D$, and so $\mathcal{N} \cap \partial D$ contains exactly one point in each arc. Since these arcs have length at most $c \delta$ for some $c > 0$, the $c \delta$-conic property is satisfied. 
\end{proof}

\begin{lemma}[Nodal set of a perturbed quadratic function: The elliptic case]
\label{lem:quad1}
Suppose that $f$ is quadratic with Hessian matrix $\nabla^2$ that has eigenvalues $\lambda_1$ and $\lambda_2$ of the same sign such that $|\lambda_1| = |\lambda_2| = 2$. Let $\delta^2 := \|g\|_{C^1(B(2))}$. Then there exists a $c > 0$ such that, for sufficiently small $\delta > 0$, there is a simply-connected domain $D$ with $B(1) \subseteq D \subseteq B(1 + c\delta)$ such that $\mathcal{N}$ satisfies the $ c \delta$-conic property in $D$.
\end{lemma}
\begin{proof}
Without loss of generality we may assume that 
\[ f(x, y) = (x - x_0)^2 + (y - y_0)^2 - d \]
for some $x_0, y_0 \in \mathbb{R}^2$ and $d \in \mathbb{R}$. Remark that the nodal set $\mathcal{N} $ is contained in the set 
\[ \mathcal{C}^\delta := \left\{(x, y) : | (x - x_0)^2 + (y - y_0)^2 | \in (d - \delta^2, d + \delta^2 ) \right\}   ,\]
which is contained in a $\delta$-thickened circle for sufficiently small $\delta$. If $d < 2 \delta^2$, then $\mathcal{C}^\delta$ is contained in a ball of radius $2\delta$, and so at least one of $D = B(1)$ or $D = B(1 + 4\delta)$ is such that $|\mathcal{C}^\delta \cap \partial D| = 0$. Suppose instead that $d \ge 2 \delta^2 $. Then, on the set $\mathcal{C}^\delta$, we have for sufficiently small $\delta > 0$, 
\[ |\nabla f| \ge 2 \sqrt{2 \delta^2 - \delta^2} = 2 \delta > \delta^2 =: \|g\|_{C^1(B(2))}. \]
Hence we may find a $D$ such that $|\mathcal{N} \cap \partial D| \le 2$, by either taking $D = B(1)$ or the smallest set that contains $B(1)$ such that $\partial D$ passes through $\mathcal{C}^\delta$ along gradient flow lines of $f$. Since these gradient flow lines have length at most $c \delta$ for some $c > 0$, such a set is in $B(1 + c \delta)$. Moreover, since $\mathcal{N} \cap \partial D$ contains at most one point in each gradient flow line that is traversed in this procedure, the $c \delta$-conic property is satisfied.
\end{proof}

\begin{lemma}[Nodal set of a perturbed quadratic function: The hyperbolic case]
\label{lem:quad2}
Suppose that $f$ is quadratic with Hessian matrix $\nabla^2$ that has eigenvalues $\lambda_1$ and $\lambda_2$ of opposite signs such that $|\lambda_1| = |\lambda_2| = 2$. Let $\delta^2 := \|g\|_{C^1(B(2))}$. Then there exists a $c > 0$ such that, for sufficiently small $\delta > 0$, there exists a simply-connected domain $D$ with $B(1) \subseteq D \subseteq B(1 + c\delta)$ such that $\mathcal{N}$ satisfies the $c \delta$-conic property in $D$.
\end{lemma}
\begin{proof}
Without loss of generality we may assume that 
\[ f(x, y) = -(x - x_0)^2 + (y - y_0)^2 + d \]
for some $x_0, y_0 \in \mathbb{R}^2$ and $d \ge 0$. Remark that the nodal set $\mathcal{N} $ is contained in the set 
\[ \mathcal{C}^\delta := \left\{(x, y) : | (x - x_0)^2 - (y - y_0)^2 | \in (d - \delta^2, d + \delta^2 ) \right\}   ,\]
which is contained in a $\delta$-thickened regular conic section for sufficiently small $\delta$. Moreover, outside the ball $\mathcal{B} := B( (x_0, y_0), \delta^2)$,
\[ |\nabla f| >  2 \delta^2 > \delta^2 =: \|g\|_{C^1(B(2))}. \]
Now, either $B(1)$ does not intersect $\mathcal{B}$, in which case we may either take $D = B(1)$ or a $D$ such that $\partial D$ passes through $\mathcal{C}^\delta$ outside $\mathcal{B}$ and along gradient flow lines of~$f$; in both these cases $|\mathcal{N} \cap \partial D| \le 2$. Or else $B(1)$ does intersect $\mathcal{B}$, and we may take $D$ such that $\partial D$ passes through $\mathcal{C}^\delta$ outside $\mathcal{B}$ and along gradient flow lines of $f$, and such that $|\mathcal{N} \cap \partial D| = 4$. Since $\mathcal{N} \cap \partial D$ contains at most one point in each gradient flow line that is traversed in this procedure, and since the gradient glow lines have length at most $c \delta$ for some $c > 0$, in both these cases $D \subseteq B(1 + c\delta)$ and the $c\delta$-conic property is satisfied.
\end{proof}

The upshot of Lemmas \ref{lem:lin}--\ref{lem:quad2} is the following general criteria for a $C^3$ function $\Psi : \mathbb{R}^2 \to \mathbb{R}$ to have a nodal set that locally satisfies the $\delta$-conic property for some small $\delta$. 

\begin{corollary}
\label{cor:pert}
There exists a $c > 0$ such that, for sufficiently small $\delta > 0$, the following holds. Let $\Psi: \mathbb{R}^2 \to \mathbb{R}$ be a $C^3$ function, with $\mathcal{N}$ its nodal set. Let $\lambda_1$ and $\lambda_2$ denote the eigenvalues of the Hessian matrix $\nabla^2 \Psi(0)$. Suppose that, for some $\varepsilon > 0$, either
\[  \frac{\|\Psi\|_{C^2(B(2 \varepsilon))}}{ |\nabla \Psi(0)| } < \frac{\delta}{4 \varepsilon} \]
or $\lambda_1$ and $\lambda_2$ are both non-zero and
\[  \left| 1 - \frac{|\lambda_1|}{|\lambda_2|} \right|  < 2\delta \quad \text{and} \quad   \frac{\|\Psi\|_{C^3(B(2 \varepsilon))}}{\max\{ |\lambda_1|, |\lambda_2| \} }    < \frac{\delta^2}{16 \varepsilon}     . \]
Then there exists a simply-connected domain $D$ with $B(\varepsilon) \subseteq D \subseteq B((1 + c\delta)\varepsilon)$ such that $\mathcal{N}$ satisfies the $c \delta \varepsilon$-conic property in~$D$.
\end{corollary}

\begin{proof}
Assume that the first condition holds, and define
\[  h(x, y) :=  \frac{\Psi(\varepsilon x, \varepsilon y) }{\varepsilon |\nabla \Psi(0)|  }  \ , \quad f(x, y) := \frac{ \Psi(0) +  (\varepsilon x, \varepsilon y)  \cdot \nabla \Psi(0) }{\varepsilon |\nabla \Psi(0)|  } \]
and $g(x, y) := h(x, y) - f(x, y)$. Remark that the nodal set of $h$ is $\varepsilon^{-1} \mathcal{N}$. Remark also that $f$ is linear with $|\nabla f(0)| = 1$, and further, by Taylor's theorem,
\[  \|g\|_{C^1(B(2))} \le   (2\varepsilon)^2 \|\Psi\|_{C^2(B(1))}  \frac{1}{\varepsilon |\nabla \Psi(0)|  }   <  \delta  .   \]
By Lemma \ref{lem:lin}, and after rescaling by $\varepsilon$, we deduce the result.

Assume instead that the second condition holds. Without loss of generality we may assume the the eigenvectors corresponding to $\lambda_1$ and $\lambda_2$ are in the direction of the $x$ and $y$ axes respectively, and that $|\lambda_1| \ge |\lambda_2|$. Let $\mu := \sqrt{|\lambda_1/\lambda_2|}$, and note that the second condition together with the inequality $\sqrt{1-x} < 1- x/2$, valid for $x \in (0,1)$, implies that $\mu \in [1, 1+\delta)$. Similarly to before, define 
\[ h(x, y) := \frac{\Psi(\varepsilon x, \mu \varepsilon y)}{ \varepsilon^2  |\lambda_1| /2}\]
and let $f(x, y)$ and $g(x, y)$ be respectively the second-order Taylor polynomial of $h$ and its remainder. Then, $f$ is quadratic whose Hessian matrix has eigenvalues $\mu_1, \mu_2$ satisfying $|\mu_1| = |\mu_2| = 2$. Moreover, by Taylor's theorem,
\[  \|g\|_{C^2(B(2))} \le  (2\varepsilon)^3  \|\Psi\|_{C^3(B(1))} \frac{1}{\varepsilon^2  |\lambda_1| /2 } <   \delta^2  .   \]
Writing $\mathcal{N}^h$ for the nodal set of $h$, by Lemmas \ref{lem:quad1} and \ref{lem:quad2} we deduce that there exists a $c_0 > 0$ and a simply-connected domain $D$ satisfying $B(1) \subseteq D \subseteq B(1 + c_0\delta)$ such that $\mathcal{N}^h$ satisfies the $c_0 \delta$-conic property in $D$. The proof is then concluded by applying Lemma \ref{lem:cone} below, remarking that the nodal set $\mathcal{N}$ results from applying the rescaling $(x, y) \mapsto \varepsilon (x,  \mu y)$ to $\mathcal{N}^h$. 
\end{proof}

\begin{lemma}
\label{lem:cone} 
Fix $c_0 > 0$. Then there exists a $c = c(c_0) > 0$ such that, for sufficiently small $\delta > 0$ and any $\varepsilon > 0$ the following holds. Define $\mu \in [1, 1 + \delta)$ and let $T^\varepsilon_\mu : \mathbb{R}^2 \to \mathbb{R}^2$ be the transformation $(x, y) \mapsto \varepsilon (x, \mu y)$. Suppose that $\mathcal{S}$ satisfies the $c_0 \delta$-conic property in the set $D \subseteq B(1 + c_0 \delta)$. Then $T^\varepsilon_\mu(\mathcal{S})$ satisfies the $c \delta \varepsilon$-conic property in the set $T^\varepsilon_\mu(D) \subseteq B((1 + c\delta)\varepsilon)$.
\end{lemma}

\begin{proof}
This is clear from the definition of the $\delta$-conic property.
\end{proof}

\subsection{Application of the perturbation analysis to plane waves}
\label{sec:pert2}

We now show how to apply the above perturbation analysis to plane waves, i.e.\ solutions to the Helmholtz equation~\eqref{eq:helm}. In particular, we deduce sufficient conditions under which the nodal set locally satisfies the $\delta$-conic property for some small $\delta>0$.
 
We begin by explaining the heuristics of the approach. Let $\Psi: \mathbb{R}^2 \to \mathbb{R}$ satisfy equation \eqref{eq:helm} for frequency parameter $k > 0$, with nodal set $\mathcal{N}$. Let $\lambda_1$ and $\lambda_2$ be the eigenvalues of the Hessian matrix $\nabla^2 \Psi(0)$. Corollary \ref{cor:pert} suggests that, outside an error event on which both $\nabla \Psi (0)$ \textit{and} the eigenvalues $\lambda_1$ and $\lambda_2$ are very small (relative to the~$C^3$ norm of $\Psi$ near the origin), the nodal set of $\Psi$ satisfies the $\delta$-conic property as long as the ratio $|\lambda_1/\lambda_2|$ is close to one. To make the link with plane waves, observe in addition to the above that unless~$\Psi(0)$ is also small, the nodal set will be empty. On the other hand, if $\Psi(0)$ is small, then since \eqref{eq:helm} implies that $\Psi(0)$ is proportional to $\lambda_1 + \lambda_2$, the ratio $|\lambda_1/\lambda_2|$ must be close to one (outside the aforementioned error event that $\lambda_1$ and $\lambda_2$ are both very small), and so the $\delta$-conic property holds. Formalising the above, we deduce the following result.

\begin{corollary}[Sufficient conditions for nodal set of a plane wave to be locally well-approximated by a regular conic section]
\label{cor:pert2}
There exists a $c = c(k) > 0$ such that, for sufficiently small $\delta > 0$, the following holds. For each $\varepsilon > 0$, if the conditions
\[  |\nabla \Psi(0)| <  c \varepsilon  \|\Psi\|_{C^2(B(2\varepsilon))}    \,  , \    \max\{ |\lambda_1|, |\lambda_2|  \} <  c \varepsilon \|\Psi\|_{C^3(B(2 \varepsilon))}    \]
and
\[    |\lambda_1 + \lambda_2|  < c \varepsilon^2 \|\Psi\|_{C^2(B(2\varepsilon))}   , \] 
do not all hold, then there exists a simply-connected domain $D$ with $B(\varepsilon) \subseteq D \subseteq B((1 + c\delta)\varepsilon)$ such that $\mathcal{N}$ satisfies the $c \delta \varepsilon$-conic property in~$D$.
\end{corollary}

\begin{proof}
Let $c_0 > 0$ denote the constant appearing in Corollary \ref{cor:pert} and fix $\delta > 0$ sufficiently small that the conclusion of that corollary is available.

We show that there exists a $c = c(k, \delta) > 0$ such that, if any of
\begin{align}
\label{eq:cond}
  |\nabla \Psi(0)| \ge  c \varepsilon  \|\Psi\|_{C^2}   \, , \ \max\{ |\lambda_1|, |\lambda_2|  \}  \ge  c \varepsilon \|\Psi\|_{C^3}      \quad \text{or} \quad   |\lambda_1 + \lambda_2|  \ge c \varepsilon^2 \|\Psi\|_{C^2(B(2)} , 
\end{align}
hold, then either $\mathcal{N}$ does not intersect $B( ( 1 + c_0\delta) \varepsilon)$, or at least one of the two conditions of Corollary~\ref{cor:pert} must be satisfied for this choice of $\delta$.  By the conclusion of Corollary \ref{cor:pert} this gives the result.

Suppose that the first statement in \eqref{eq:cond} holds. Then, the first condition of Corollary \ref{cor:pert} is satisfied immediately if we take $c > c_1 := 4/\delta$. So henceforth we may assume that
\[ |\nabla \Psi(0)| <  c_0 \varepsilon  \|\Psi\|_{C^2(B(2\varepsilon))} .\]
Now, either $\mathcal{N}$ does not intersect $B( \varepsilon(1 + c_0\delta))$, or by Taylor's theorem we deduce that, for $c_2 > 2 c_1$,
\[   |\Psi(0)| <  c_2 \varepsilon^2  \|\Psi\|_{C^2(B(2\varepsilon))} .  \]
Since $\Psi$ satisfies the Helmholtz equation \eqref{eq:helm}, this also implies that, for $c_3 > k c_2$, 
\[   |\lambda_1 + \lambda_2| <  c_3 \varepsilon^2  \|\Psi\|_{C^2(B(2\varepsilon))} .   \]
Hence, since $\varepsilon < 1$ and $\|\Psi\|_{C^2(B(2))} \le \|\Psi\|_{C^3(B(2\varepsilon))}$, it remains to consider the case that
\begin{align*}
   \max\{ |\lambda_1|, |\lambda_2|  \}  \ge  c \varepsilon \|\Psi\|_{C^3(B(2\varepsilon))}      \quad \text{and} \quad   |\lambda_1 + \lambda_2|  < c_2 \varepsilon^2 \|\Psi\|_{C^3(B(2\varepsilon))}  . 
\end{align*}
Choosing $c > c_2$, this implies that that $\lambda_1$ and $\lambda_2$ are non-zero and of opposite sign. Hence using the inequality
\[ \left| 1 - \frac{ |x|}{|y|} \right| \le \frac{|x+y|}{\max\{|x|, |y| \} - |x + y|} , \]
valid for non-zero $x$ and $y$ of opposite sign, the second condition of Corollary \ref{cor:pert} is then satisfied if~$c$ is chosen large enough.
\end{proof}

\smallskip

\section{Proof of the main results}
\label{sec:proofs}

In this section we complete the proofs of the main results. The variables $\kappa, \ell, \mathcal{L}, D$ and~$\delta$ are fixed throughout this section. Note that $\kappa$ is assumed to satisfy the conditions in Assumption~\ref{assumpt:degen}, and the domain $D$ is always considered to be bounded, but it is not necessarily smooth unless explicitly stated. Throughout this section, whenever an edge $e \in \varepsilon \mathcal{E}$ is considered, $f_1, f_2 \in \varepsilon \mathcal{F}$ are defined to be the faces adjacent to $e$. 

\subsection{Control of the `bad' events}

We first show how to control each of the `bad' events introduced in Section \ref{sec:overview} using the results developed in Sections \ref{sec:gauss}--\ref{sec:pert}. As explained in that section, for the proofs of Theorems~\ref{thm:main1}, \ref{thm:main2} and \ref{thm:main3} we need only control the existence of these events. On the other hand, for the proof of Theorem~\ref{thm:main4} we shall also need to control their \textit{multiplicity}.

\subsubsection*{Double-crossings}

We control double-crossings using the Kac-Rice formula in Proposition \ref{prop:kr2}. For each $\varepsilon > 0$ and edge $e \in \varepsilon \mathcal{E}$, let $D(e)$ be the event that $e$ has a double-crossing (as defined in Section \ref{sec:overview}), and define \textit{multiplicity of the double-crossing} of $e$ to be 
\[ \tilde{D}(e) := \max \left\{ 0, |\{\mathcal{N}_\ell \cap e\}| - 1 \right\} . \]

\begin{proposition}[Control of double-crossings]
\label{prop:condc}
There exists a $c = c(\kappa, \ell, \mathcal{L})$ such that, for each $\varepsilon \in (0, 1)$ and $e \in \varepsilon \mathcal{E}$,
\[    \mathbb{E}[ \tilde{D}(e) ] <  c \varepsilon^3   ,\]
and in particular,
\[    \mathbb{P}( D(e)) <  c \varepsilon^3  .\]
\end{proposition}
\begin{proof}
This is a direct application of Proposition~\ref{prop:kr2} together with Markov's inequality.
\end{proof}

\subsubsection*{Four-crossings and tubular-crossings}
We control four-crossings and tubular-crossings by linking these to simpler crossing events which we treat with Kac-Rice formulae. In particular, we show that four-crossings and tubular-crossings necessarily imply either (i) non-linear crossings of three line segments, or (ii) crossings that are extremely close to lattice vertices or to each other; we control the former using Proposition~\ref{prop:kr3}, and the latter using Proposition~\ref{prop:kr2}. 

For each $\varepsilon > 0$ and edge $e \in \varepsilon \mathcal{E}$, let $T(e)$ be the event that $e$ has a tubular crossing (as defined in Section \ref{sec:overview}), and define the \textit{multiplicity of the tubular-crossing} of $e$, denoted $\tilde{T}(e)$, to be the number of components of $f_1 \cup f_2 \setminus \mathcal{N}_\ell$ that lie between the endpoints of $e$ (i.e.\ two less than the smallest number of components crossed by any curve joining the endpoints).

Moreover, for each $\varepsilon > 0$ and face $f \in \varepsilon \mathcal{F}$, let $F(f)$ be the event that $f$ has a four crossing (as defined in Section \ref{sec:overview}), and define the \textit{multiplicity of the four-crossing} to be 
\[  \tilde{F}(f) := \max \left\{ 0, \frac{|\{ e \in \partial f : |\{\mathcal{N}_\ell \cap e\}| \ge 1 \}| - 2}{2} \right\} .\]

\begin{lemma}[Four-crossings and tubular-crossings imply non-degenerate triple-crossings]
\label{lem:triple}
Suppose that $\mathcal{L}$ is strictly-convex, i.e.\ the interior angles of the faces $\mathcal{F}$ are strictly less than~$\pi$. Then, there exists an angle $\mu = \mu(\mathcal{L}) \in [0, \pi)$ and a constant $c = c(\mathcal{L}) > 0$ such that, for any $\delta \in (0, 1)$ and any $e \in \mathcal{E}$ and $f \in \mathcal{F}$:
\begin{enumerate}
\item If $F(f)$ holds, then there are edges $e_1, e_2, e_3 \in \partial f$ and points $s = (s_1, s_2, s_3)$ satisfying $s_i \in e_i \cap \mathcal{N}_\ell$ such that either:
\begin{enumerate}
\item $\theta^-(s) \ge \delta$ and $\theta^+(s) \le  \mu$; or
\item Two vertices of $f$ lie within a distance $c \delta$ of a point in $s$. 
\end{enumerate}
Moreover, if $\tilde{F}(f) = n \in \mathbb{N}$, there are at least $n$ distinct such sets of points $s$; and
\item If $T(e)$ holds, then there are edges $e_1, e_2, e_3 \in \partial (f_1 \cup f_2)$ and points $s = (s_1, s_2, s_3)$ satisfying $s_i \in e_i \cap \mathcal{N}_\ell$ such that either:
\begin{enumerate}
\item $\theta^-(s) \ge \delta$ and $\theta^+(s) \le  \mu$; or
\item Each vertex of $e$ lies within a distance $c \delta$ of a point in $s$; or
\item The edge $e$ contains two points in $s$ within a distance $c \delta$.
\end{enumerate}
Moreover, if $\tilde{T}(e) = n \in \mathbb{N}$, there are at least $n$ distinct such sets of points $s$.
\end{enumerate}
\end{lemma}

\begin{proof}
This can be seen by inspection; see Figure \ref{fig:triple}.
\end{proof}

\begin{proposition}[Control of four-crossings and tubular-crossings]
\label{prop:contriple}
Suppose that $\mathcal{L}$ is strictly-convex and fix $\delta > 0$. Then there exists a $c = c(\kappa, \ell, \mathcal{L}, \delta)$ such that, for each $\varepsilon \in (0, 1)$ and $e \in \varepsilon \mathcal{E}$ and $f \in \varepsilon \mathcal{F}$,
\[    \mathbb{E}[ \tilde{T}(e) ]  <  c \varepsilon^{4 - \delta}  \quad \text{and} \quad  \mathbb{E}[ \tilde{F}(f) ]  < c \varepsilon^{4 - \delta}  ,\]
and in particular,
\[     \mathbb{P}( T(e) ) <  c \varepsilon^{4 - \delta}  \quad \text{and} \quad  \mathbb{P}( F(f) )  < c \varepsilon^{4 - \delta}  .\]
\end{proposition}
\begin{proof}
This can be proved by combining Lemma \ref{lem:triple} (setting $\delta := \varepsilon^{3/2}$) with the bounds in Proposition \ref{prop:kr3} (to control cases (1a) and (2a)) and Proposition \ref{prop:kr2} (to control cases (1b), (2b) and (2c)), and applying Markov's inequality. Notice that the choice of $\delta := \varepsilon^{3/2}$ is the maximum possible that will give the correct order $O(\varepsilon^4)$ in (1b) and (2b)). It is for this reason that we are able to impose the constraint $\theta^-(s) \ge \varepsilon^{3/2}$ in our statement of Proposition \ref{prop:kr3} (which eases the calculations), but no stronger constraint is possible. 
\end{proof}

\begin{figure}[ht]
\begin{center}
\includegraphics[scale=1]{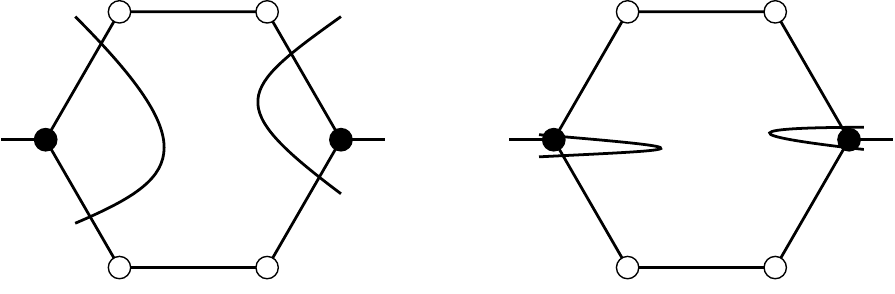}
\\
\includegraphics[scale=1]{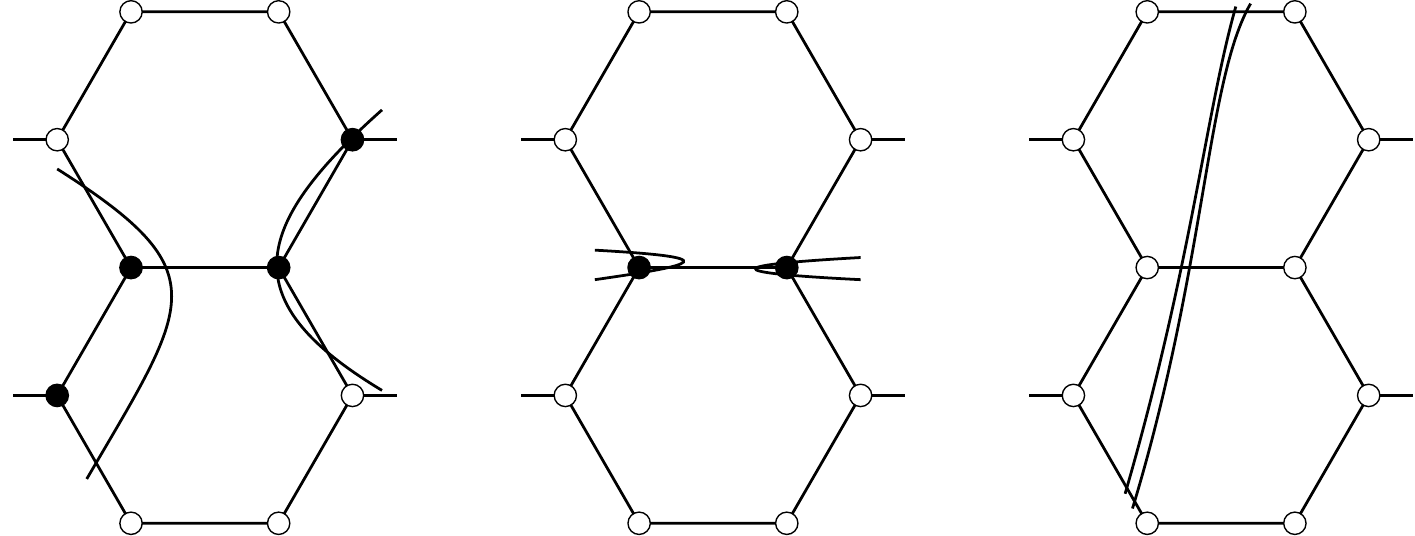}
\end{center}
\caption{Examples illustrating different case in Lemma \ref{lem:triple}, in particular, the two cases in part $(1)$ (top row) and the three cases in part $(2)$ (bottom row).}
\label{fig:triple}
\end{figure}

\subsubsection*{Small excursion domains}

We control small excursion domains by using the derivative bounds in Lemma \ref{lem:derbounds} and Proposition~\ref{prop:krmax}. For each $\varepsilon > 0$ and bounded domain $\bar D$, define a \textit{small excursion domain in} $\bar D$ to be a component of $\bar D \setminus \mathcal{N}_\ell $ whose boundary intersects at most one edge in $\varepsilon \mathcal{E}$ and does not intersect~$\partial \bar D$. For each $s > 0$, let $\tilde{S}(s, \varepsilon)$ to be the number of small excursion domains in~$sD$.

\begin{lemma}
\label{lem:snc}
For each $s > 0$ and $\varepsilon > 0$,
\[ \tilde{S}(s, \varepsilon) \le \left| \left\{ x \in sD : \nabla \Psi(x) = 0 , |\Psi(x) - \ell| <  2 d(\mathcal{L})^2 \varepsilon^2 \|\Psi\|_{C^2(B(x, 2\varepsilon d(\mathcal{L}))} \right\} \right| . \]
\end{lemma}
\begin{proof}
Let $\mathcal{S} \subseteq sD$ be a small excursion domain in $sD$. Observe that there must be a local maxima or minima contained within $\mathcal{S}$; suppose that $x \in \mathcal{S}$ is such a point. Let $f  \in \varepsilon \mathcal{F}$ denote the face that contains $x$, and $F$ denote the union of $f$ with all its adjacent faces in~$\varepsilon \mathcal{F}$. By Taylor's theorem, each $y \in F$ satisfies
\[ | \Psi(y) - \Psi(x)| <  \frac{1}{2} (2d(\mathcal{L}))^2 \varepsilon^2 \|\Psi\|_{C^2(B(x, 2\varepsilon d(\mathcal{L}))} . \]
Since $\mathcal{S}$ is a small excursion domain, the level set $\mathcal{N}_\ell$ intersects $F$, and so we may select a $y \in F$ such that $\Psi(y) = \ell$. We conclude that $x \in \mathcal{S}$ satisfies
\[  | \Psi(x) - \ell |  <  2 d(\mathcal{L})^2 \varepsilon^2 \|\Psi\|_{C^2(B(x, 2 \varepsilon  d(\mathcal{L}))} , \]
which implies the result.
 \end{proof}

\begin{proposition}[Control of small excursion domains]
\label{prop:consed}
There exists a $c = c(\kappa, \ell, \mathcal{L}, D)$ such that, for each $s > 0$ and $\varepsilon > 0$,
\[    \mathbb{E}[ \tilde{S}(s, \varepsilon) ]  <  c \varepsilon^2 s^2,\]
and in particular,
\[     \mathbb{P}( \tilde{S}(s, \varepsilon) \ge 1) <  c \varepsilon^2 s^2 .\]
\end{proposition}

\begin{proof}
Abbreviate $c_1 = 2d(\mathcal{L})^2$, and note that there exists a $c_2 = c_2(\mathcal{L}, D)$ such that the number of unit squares $I := [0, 1]^2$ needed to cover $sD$ is at most $c_2 s^2$. Applying Lemma \ref{lem:snc}, there exists a $c_3 = c_3(\kappa, \ell, \mathcal{L}, D) > 0$ such that
\begin{align*}
   \mathbb{E}[ \tilde{S}(s, \varepsilon) ]  & \le  \mathbb{E} \left[    \left| \left\{ x \in sD : \nabla \Psi(x) = 0 , |\Psi(x) - \ell| < c_1 \varepsilon^2 \|\Psi\|_{C^2(B(x, 2 \varepsilon d(\mathcal{L}) ) )} \right\} \right|   \right]   \\
   & \le c_2  s^2  \mathbb{E} \left[    \left| \left\{ x \in I : \nabla \Psi(x) = 0 , |\Psi(x) - \ell| < c_1  \varepsilon^2 \|\Psi\|_{C^2(B(1 + 2 d(\mathcal{L}) ) )} \right\} \right|   \right]  \\
   & \le  c_3  \varepsilon^2 s^2   \, \mathbb{E} [ \|\Psi\|_{C^2(B(1 + 2d(\mathcal{L}) ))} ] . 
   \end{align*}
where in the final step we conditioned on the $\sigma$-algebra generated by $ \|\Psi\|_{C^2(1 + 2d(\mathcal{L}) )} $ and used Proposition \ref{prop:krmax}. Applying Lemma~\ref{lem:derbounds} together with Markov's inequality gives the result. 
\end{proof}

\subsubsection*{Invisible errors}

Finally, we control invisible errors by combining the deterministic analysis in Corollary~\ref{cor:pert2} with the bounds in Lemma \ref{lem:derbounds} and Proposition \ref{prop:smalldevest}. As for Section \ref{sec:pert}, this subsection treats only the special case of the random plane wave, and may be skipped by a reader not concerned with this case.

In this subsection we assume that $\kappa$ satisfies~\eqref{eq:rpw} for a frequency parameter $k > 0$, $\mathcal{L}$ is a regular hexagonal lattice, and the level is set at $\ell= 0$. Without loss of generality we may rescale $\mathcal{L}$ so that $d(\mathcal{L}) = \sqrt{13}/4 \approx 0.9$ (this ensures that a circle of unit radius exactly circumscribes two adjacent hexagons).

First we give a formal definition of the `invisible errors' that we consider.  Recall the signed regions $\mathcal{S}^+$ and $\mathcal{S}^-$ defined in \eqref{def:signedregions}. Recall also the definition of a Type $2$ (see in particular Figure~\ref{fig:err}). Let $\varepsilon > 0$ and consider an edge $e \in \varepsilon \mathcal{E}$. Let $\bar D$ be a domain such that $f_1 \cup f_2 \subseteq \bar D$. We say that $e$ \textit{gives rise to an invisible error in} $\bar D$ if the faces $f_1$ and $f_2$ do not display a Type~$2$ error pattern, and either: (i) there is a component of $\bar D \setminus \mathcal{N}$ that separates the endpoints of $e$ in $\bar D$; or (ii) there is a component of $(f_1 \cup f_2) \setminus \mathcal{N}$ that separates the endpoints of $e$ in $f_1 \cup f_2$, and moreover this component intersects $\partial \bar D$ and at least one vertex in $f_1 \cup f_2$. See Figure \ref{fig:ie} for examples of these two cases.

\begin{figure}[t]
\begin{center}
\includegraphics[scale=1]{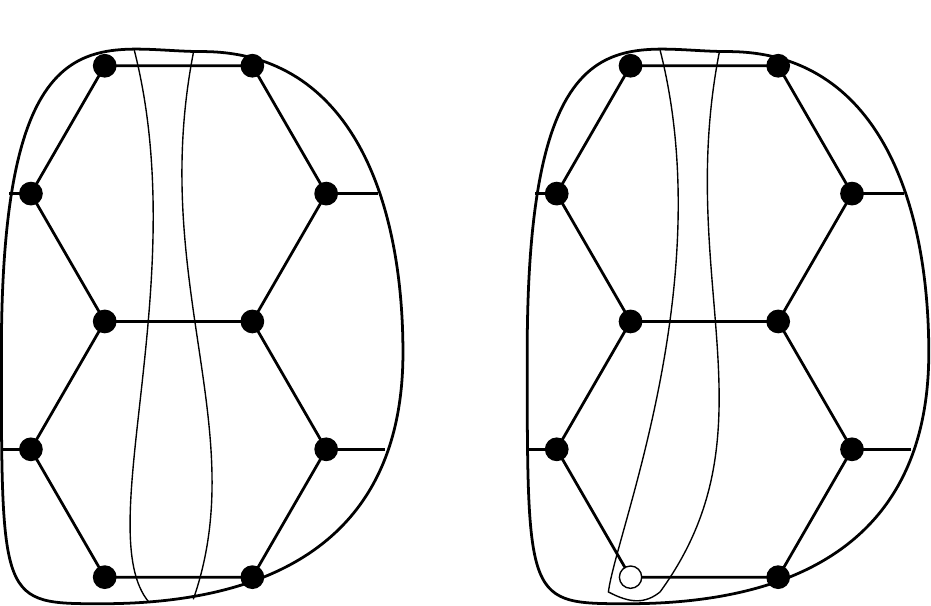}
\end{center}
\caption{Examples illustrating the two different cases in which an edge may give rise to an invisible error inside a domain.}
\label{fig:ie}
\end{figure}

Before discussing how to control invisible errors, we state a result that links them to the geometry of the regular hexagonal lattice, by way of the $\delta$-conic property; this is the only time the specific geometry of regular hexagons is needed (with one small caveat mentioned below). 

\begin{lemma}[Cones are `visible' on the vertices of adjacent hexagons] 
\label{lem:hex}
Let $H_1$ and $H_2$ be adjacent regular hexagons. Then there exists a sufficiently small $\delta > 0$ such that, for any cone~$\mathcal{C}$ with centre $x$ and angle at most $\delta$ away from $\pi/2$, if either:
\begin{enumerate}
 \item The endpoints of $e$ lie in opposite components of the exterior of $\mathcal{C}$; or
 \item One endpoint of $e$ lies in the exterior of $\mathcal{C}$ and the other is within a distance $\delta$ of $x$;
 \end{enumerate}
 then both $H_1$ and $H_2$ have vertices lying in the interior of $\mathcal{C}$.
\end{lemma}
\begin{proof}
This is clear from inspection; see Figure \ref{fig:hex2}.
\end{proof}

\begin{figure}[ht]
\begin{center}
\includegraphics[scale=1]{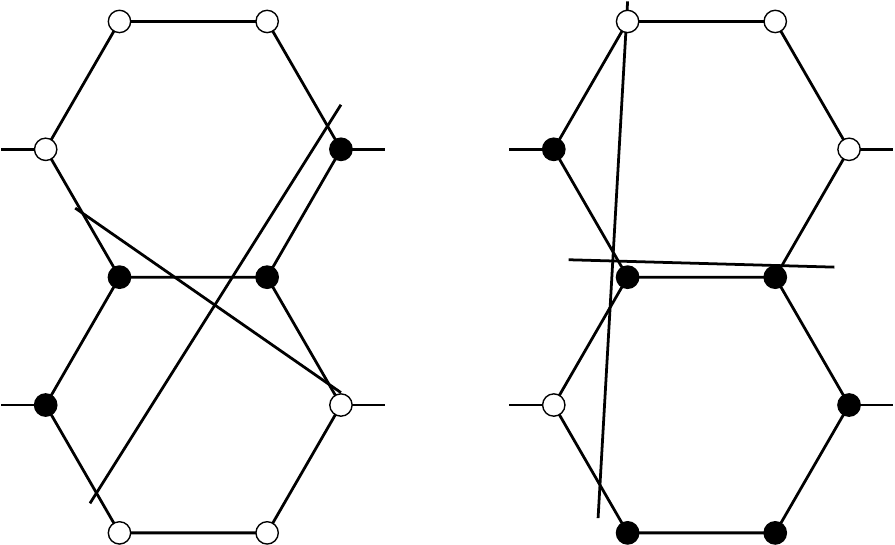}
\end{center}
\caption{Examples illustrating the two different cases in Lemma \ref{lem:hex}.}
\label{fig:hex2}
\end{figure}

\begin{corollary}[The $\delta$-conic property precludes invisible errors]
\label{cor:hex}
There exists a sufficiently small $\delta > 0$ such that, for each $\varepsilon >  0$ and each $e \in \varepsilon \mathcal{E}$, if $\mathcal{N}$ satisfies the $\delta \varepsilon$-conic property in a simply-connected domain $\bar D$ with $f_1 \cup f_2 \subseteq \bar D$, then $e$ does not give rise to an invisible error.
\end{corollary}
\begin{proof}
This can be deduced from Lemma \ref{lem:hex}; see Figure \ref{fig:hex3}. The most delicate case is when the regular conic section is a pair of lines intersecting at right-angles, in which case Lemma \ref{lem:hex} applies directly. Note that we may assume without loss of generality that $\varepsilon = 1$.
\end{proof}

\begin{figure}[ht]
\begin{center}
\includegraphics[scale=1]{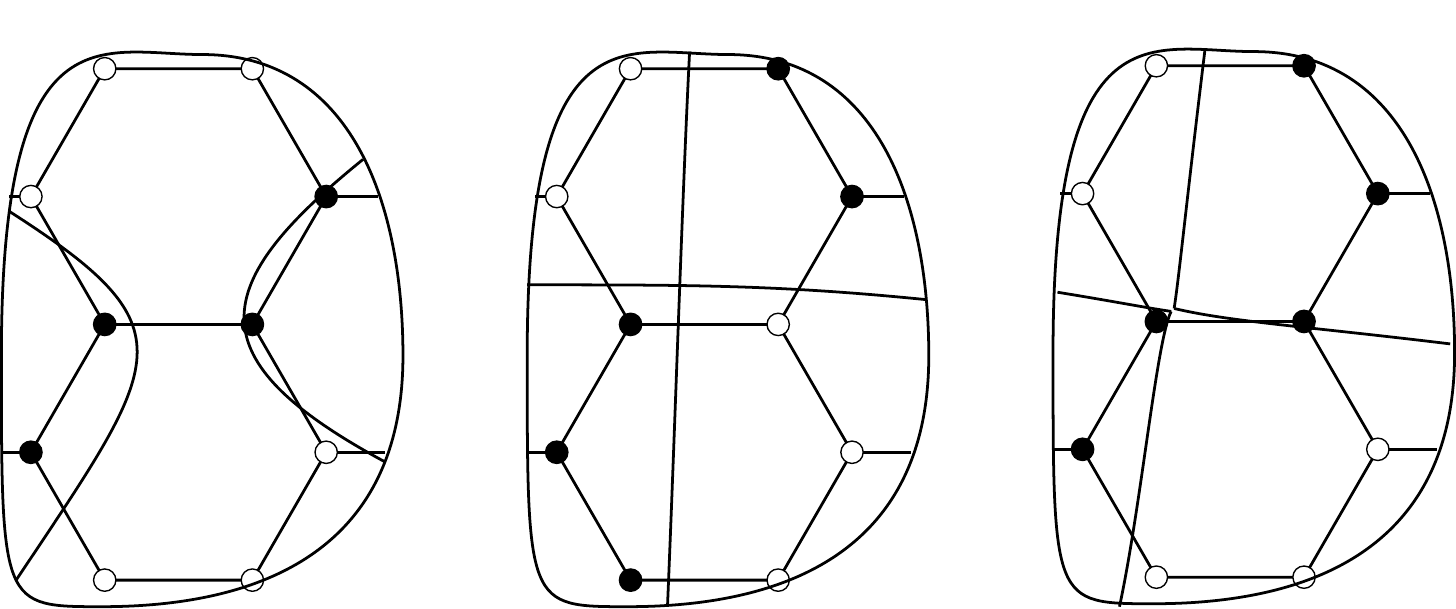}
\end{center}
\caption{Examples illustrating different cases in Corollary \ref{cor:hex}; the most delicate case is when the regular conic section is a pair of lines intersecting at right-angles (the two panels on the right).}
\label{fig:hex3}
\end{figure}

For each $\varepsilon > 0$ and edge $e \in \mathcal{E}$, let $m_e$ denote the mid-point of $e$. Let $a > 1$ be a constant such that (i) the ball $B(m_e, a \varepsilon)$ does not contain any other vertices in $\mathcal{V} \setminus (f_1 \cup f_2)$; and (ii) $2a < 3 d(\mathcal{L})$. It is easy to check that the first condition is possible (note that this uses a specific property of the regular hexagonal lattice, but one that is also shared by the square and regular triangular lattices), and the second condition is the origin of the bound $3 d(\mathcal{L})$ in the definition of $\varepsilon$\textit{-homeomorphic} in Section \ref{sec:intro}). Let $I(e)$ be the event that $e$ gives rise to an invisible error in every domain $\bar D$ such that $B(\varepsilon) \subseteq \bar D \subseteq B(a \varepsilon)$. 

\begin{proposition}
\label{prop:conie}
There exists a $c = c(k)$ such that, for each $\varepsilon > 0$ and $e \in \varepsilon \mathcal{E}$,
\[     \mathbb{P} ( I(e)  ) <  c \varepsilon^6   .\]
\end{proposition}
\begin{proof}
Assume without loss of generality that $m_e$ is the origin.  Let $c_1 = c_1(k) > 0$ be a constant satisfying Corollary~\ref{cor:pert2}, and select $\delta > 0$ sufficiently small such that, if $\mathcal{N}$ satisfies the $c_1 \delta$-conic property in a simply-connected domain $\bar D$ satisfying $f_1 \cup f_2 \subseteq \bar D$, then $e$ does not give rise to an invisible error. This is possible by Corollary~\ref{cor:hex}. Also choose $\delta$ small enough to ensure that $1 + c_1 \delta < a$, where $a > 1$ is the constant used in the definition of the event $I(e)$.

Abbreviate $\hat d = \| \Psi \|_{C^3(B(2))}$. By the conclusion of Corollaries~\ref{cor:pert2} and \ref{cor:hex},
\begin{align*}
 \mathbb{P} ( I(e) )  \le   \mathbb{P} \left(  \|\nabla \Psi(0)| < c_1 \varepsilon \hat d  , \,   \max\{ |\lambda_1|, |\lambda_2| \} < c_1 \varepsilon \hat d  , |\lambda_1 + \lambda_2| < c_1 \varepsilon^2 \hat d\right) <  c_2 \varepsilon^6  \hat{d}^4,
\end{align*}
for some $c_2 = c_2(k) > 0$, where $\lambda_1$ and $\lambda_2$ denote the two eigenvalues of the Hessian $\nabla^2 \Psi(0)$,  and in the last step we conditioned on the $\sigma$-algebra generated by $\hat{d}$ and applied Proposition~\ref{prop:smalldevest}. Applying Lemma~\ref{lem:derbounds} and Markov's inequality gives the result. 
\end{proof}

\subsection{Proof of the main results}

We complete the proof of the main results by combining our control on the `bad' events in Propositions \ref{prop:condc}, \ref{prop:contriple}, \ref{prop:consed} and \ref{prop:conie} with simple topological arguments.

\subsubsection*{Proof of Theorem \ref{thm:main1}}
Observe first that we may assume that $\mathcal{L}$ is strictly-convex, since if it is not we may simply apply the result to a strictly-convex graph $\mathcal{L}^\ast$ formed by adding additional edges to $\mathcal{L}$, and deduce the result. 

Note that the number of edges in $sD \cap \varepsilon \mathcal{E}$ is of order $ \varepsilon^{-2} s^2$. Hence, by Propositions~\ref{prop:condc}, \ref{prop:contriple} and~\ref{prop:consed} and the union bound, there exists a $c = c(\kappa, \ell, \mathcal{L}, D)$ such that, for each $\varepsilon \in (0, 1)$ and $s > 0$, there is an event $\mathcal{X}$ of probability $1 - c s^2 \varepsilon$ on which: (i) no edge in $sD \cap \varepsilon \mathcal{E}$ has a double-crossing; (ii) no face in $sD \cap \varepsilon \mathcal{F}$ has a four-crossing; and (iii) there are no small excursion domains inside~$sD$. We henceforth assume that the event $\mathcal{X}$ holds.
 
Consider a face $f \in  P^\varepsilon(s D) \cap \varepsilon \mathcal{F}$. Since we are working on $\mathcal{X}$, none of the boundary edges $e \in \partial f$ has a double-crossing, nor does $f$ have a four-crossing or a small excursion domain. Hence the level set $\mathcal{N}_\ell |_f$ is either empty or consists of exactly one curve which joins distinct edges of $f$. In either case, the level sets $\mathcal{N}_\ell^\varepsilon$ and $\mathcal{N}_\ell$ are $\varepsilon$-homeomorphic in $f$, as required.

Remark that, by this argument, we extract a bound of $c \varepsilon s^2$ on the rate of convergence in the statement of Theorem \ref{thm:main1}, for some $c = c(\kappa, \ell, \mathcal{L}, D)$.

\subsubsection*{Proof of Theorem \ref{thm:main2}}
As in the proof of Theorem \ref{thm:main1}, we may assume that $\mathcal{L}$ is strictly-convex. Note that, although the number of edges in $sD \cap \varepsilon \mathcal{E}$ is of order $\varepsilon^{-2} s^2$, since the domain $D$ is smooth the number of boundary edges in $\partial P^\varepsilon(s D)$ is only of order $\varepsilon^{-1} s$. Hence by Propositions~\ref{prop:condc}, \ref{prop:contriple} and~\ref{prop:consed} and the union bound, there exists a $c = c(\kappa, \ell, \mathcal{L}, D, \delta)$ such that, for each $\varepsilon \in (0, 1)$ and $s > 0$, there is an event $\mathcal{X}$ of probability $1 -  \varepsilon^{2 - \delta} s^{2}$ on which: (i) no boundary edge in $\partial P^\varepsilon(sD) \cap \varepsilon \mathcal{E}$ has a double-crossing; (ii) no face in $P^\varepsilon(sD) \cap \varepsilon \mathcal{F}$ has either a tubular-crossing or a four-crossing; and (iii) there are no small excursion domains inside $P^\varepsilon(s D)$. We henceforth assume that the event $\mathcal{X}$ holds.

As in the proof of Theorem \ref{thm:main1}, if a face $f \in P^\varepsilon(sD) \cap \varepsilon \mathcal{F}$ does not have a boundary edge $e \in \partial f$ with a double-crossing, then the level sets $\mathcal{N}_\ell^\varepsilon$ and $\mathcal{N}_\ell$ are $\varepsilon$-homeomorphic in $f$. So consider an edge $e \in P^\varepsilon(sD) \cap \varepsilon \mathcal{E}$ that has a double-crossing, and let $f_1, f_2 \in \varepsilon \mathcal{F}$ denote the adjacent faces. 

Since we are working on $\mathcal{X}$, any such edge is: (i) in the interior of $P^\varepsilon(sD)$; and (ii) does not have a tubular-crossing. This implies the existence of a pair of adjacent points in $\mathcal{N}_\ell \cap e$ that are connected by a single level line $c  \subseteq \mathcal{N}_\ell \cap ( f_1 \cup f_2 ) \subseteq P^\varepsilon(sD)$. Since there are no small excursion domains, there therefore exists a homeomorphism $h: f_1 \cup f_2 \to f_1 \cup f_2$ mapping the level line $c$ onto a second curve $\hat{c}$ such that (i) $c$ and $\hat{c}$ have the same endpoints on $\partial (f_1 \cup f_2)$, and (ii) $\hat{c}$ crosses the edge $e$ two fewer times than $c$. Moreover, this homeomorphism is local, in the sense that it moves points at most $2 d(\mathcal{L}) \varepsilon$. 

Sequentially applying such homeomorphisms to eliminate all double-crossing from $e$ (see Figure~\ref{fig:proof}), and then repeating the procedure for each $e \in P^\varepsilon(sD) \cap \varepsilon \mathcal{F}$ that has a double-crossing, we see that the the level sets $\mathcal{N}_\ell^\varepsilon$ and $\mathcal{N}_\ell$ are $\varepsilon$-homeomorphic in $sD$, as required.

Remark that, by this argument, we extract a bound of $c  \varepsilon^{2 - \delta} s^{2} $ on the rate of convergence in the statement of Theorem \ref{thm:main2}, for some $c = c(\kappa, \ell, \mathcal{L}, D, \delta)$.

\begin{figure}[ht]
\begin{center}
\includegraphics[scale=1]{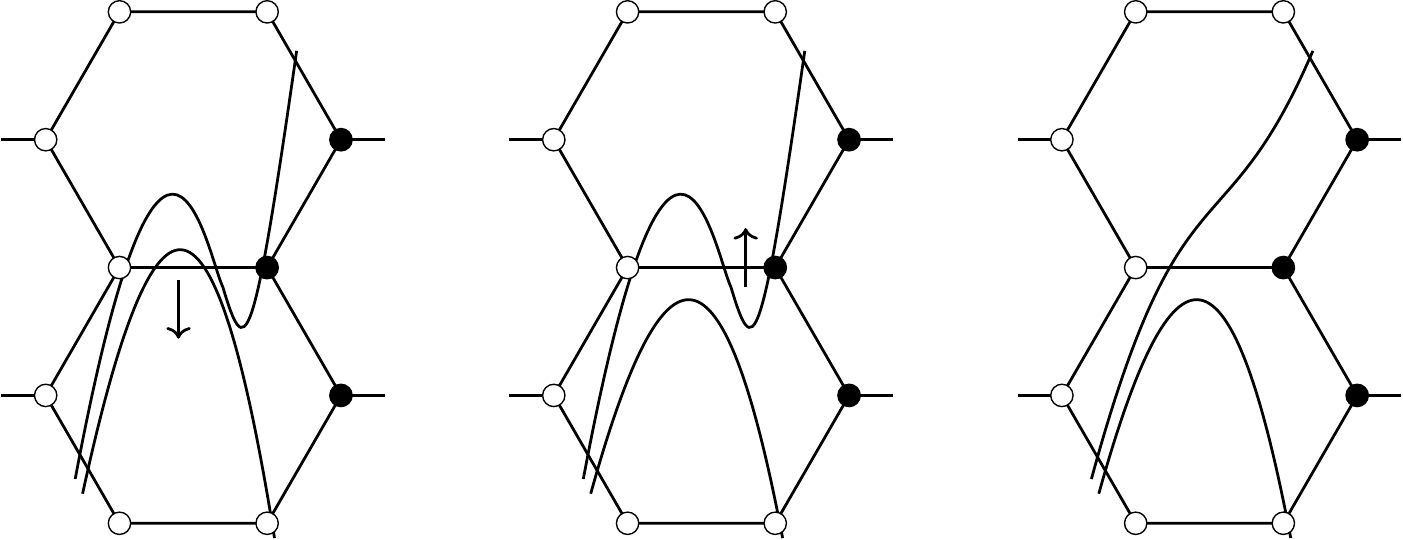}
\end{center}
\caption{Example of a sequence of local homeomorphisms to eliminate all double-crossings from an edge that does not have a tubular-crossing, as in the proof of Theorem \ref{thm:main2}.}
\label{fig:proof}
\end{figure}

\subsubsection*{Proof of Theorem \ref{thm:main3}}
Without loss of generality we may rescale $\mathcal{L}$ so that $d(\mathcal{L}) = \sqrt{13}/4$ (since $\mathcal{L}$ is regular, this is equivalent to rescaling the mesh-size $\varepsilon > 0$).
 
Assume first that $\varepsilon > 0$ is sufficiently small that nodal sets intersect at least two adjacent faces (i.e.\ there are no small nodal sets). Similarly to in the proofs of Theorem \ref{thm:main1} and \ref{thm:main2}, by Propositions~\ref{prop:condc} and \ref{prop:conie} and the union bound, there exists a $c = c(k, D)$ such that, for each $s > 0$ and $\varepsilon > 0$, there is an event $\mathcal{X}$ of probability $1 - c  \varepsilon^2 s$ on which: (i) no boundary edge in $\partial P^\varepsilon(sD) \cap \varepsilon \mathcal{E}$ has a double-crossing; (ii) no  edge $e \in sD \cap \varepsilon \mathcal{E}$ satisfies the invisible error event $I(e)$. Note that the error of order $\varepsilon^2 s$ arise from the control on double-crossings on the boundary; the control on invisible errors gives a term of order $\varepsilon^4 s^2$. Henceforth assume that the event $\mathcal{X}$ holds.

As in the proof of Theorem \ref{thm:main1}, the level sets $\mathcal{N}_\ell^\varepsilon$ and $\mathcal{N}_\ell$ are $\varepsilon$-homeomorphic in $f$ unless a face $f \in P^\varepsilon(sD) \cap \varepsilon \mathcal{F}$ either does not have a four-crossing, or does not have an edge $e \in \partial f$ with a double-crossing. We consider each of these cases in turn.

Consider an edge $e \in P^\varepsilon(sD) \cap \varepsilon \mathcal{E}$ that has a double-crossing, and let $f_1, f_2 \in \varepsilon \mathcal{F}$ denote the adjacent faces. Since we are working on $\mathcal{X}$, any such edge is (i) in the interior of $P^\varepsilon(sD)$, and (ii) does not satisfy the invisible error event $I(e)$. Using the definition of the event $I(e)$, and since there are no nodal domains, there are now two possibilities. Either $f_1$ and $f_2$ display a Type $2$ error, in which case they are contained in a component of the set of visible ambiguities $\pi^\varepsilon$. Or, there must exist a domain $\bar D \subseteq P^\varepsilon(sD)$ satisfying $2 d(\bar D) < 3 \varepsilon$ in which a homeomorphism can be constructed that eliminates all double-crossing of $e$ (in the same sense as in the proof of Theorem \ref{thm:main2}) and that fixes the vertices in $f_1 \cup f_2$ (for this, the second condition in the definition of `invisible errors' is crucial, as well as the fact that $\bar D$ does not contain any other vertices in $\mathcal{V}$ by construction). Note that by the bound on $d(\bar D)$, this homeomorphism moves points at most $3  \varepsilon$. See Figure \ref{fig:proof2} for an illustration.

Consider next a face $f \in P^\varepsilon(s D) \cap \varepsilon \mathcal{F}$ with a four-crossing, but whose boundary edges do not have double-crossings. Then $f$ must display a Type $1$ error pattern, and hence is also contained in a component of $\pi^\varepsilon$.

Combining the above, we have shown that the the level sets $\mathcal{N}_\ell^\varepsilon$ and $\mathcal{N}_\ell$ are $\varepsilon$-homeomorphic up to resolutions of the set of visible ambiguities, as required. Remark that, by this argument, we extract a bound of $c  \varepsilon^2 s $ on the rate of convergence in the statement of Theorem~\ref{thm:main3}, for some $c = c(k, \mathcal{L}, D)$.

\begin{figure}[ht]
\begin{center}
\includegraphics[scale=1]{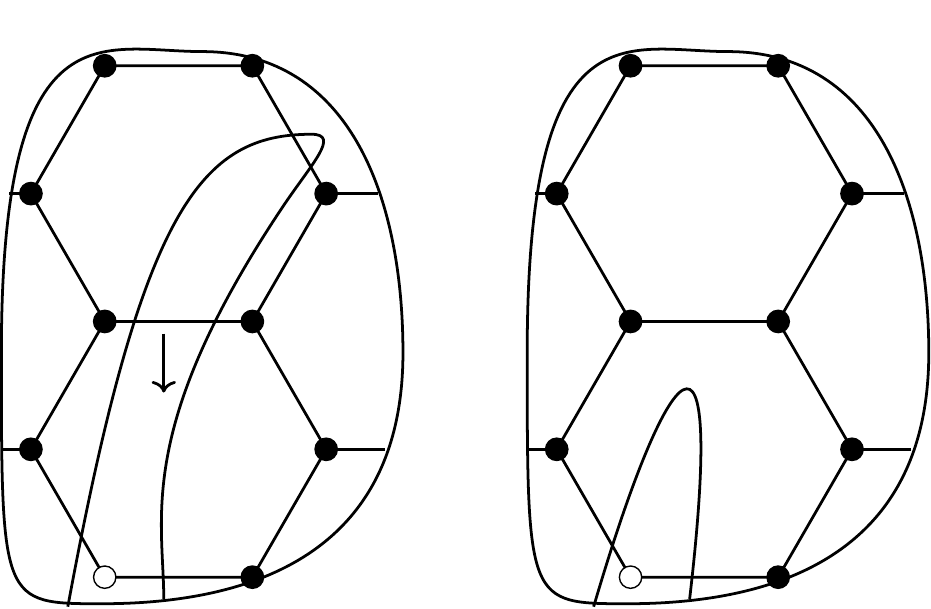}
\end{center}
\caption{Example of a of local homeomorphism to eliminate a double-crossing from an edge that does not give rise to an invisible error, as in the proof of Theorem \ref{thm:main3}.}
\label{fig:proof2}
\end{figure}

\subsubsection*{Proof of Theorem \ref{thm:main4}}
As in the proof of Theorem \ref{thm:main1}, we may assume that $\mathcal{L}$ is strictly-convex. Recall that $P^\varepsilon(sD)$ is the largest $\mathcal{L}$-compatible set that is contained in $sD$, and let $\bar{P}^\varepsilon(sD)$ be the smallest $\mathcal{L}$-compatible set that contains $sD$. Recall also that $N_\ell(sD)$ and $N_\ell^\varepsilon(sD)$ denote, respectively, the number of components of $sD \cap \mathcal{N}_\ell$ and~$P^\varepsilon(sD) \cap \mathcal{N}_\ell^\varepsilon$. 

Observe that discrepancies in $N_\ell(sD)$ and $N_\ell^\varepsilon(sD)$ come from three sources: (i) components of $sD \setminus \mathcal{N}_\ell$ that do not intersect any $v \in sD \cap \varepsilon \mathcal{V}$; (ii) vertices $v \in sD \cap \varepsilon \mathcal{V}$ that belong to the same component of $sD \setminus \mathcal{N}_\ell$ but different components of $sD \setminus \mathcal{N}_\ell^\varepsilon$; and (iii)  vertices $v \in sD \cap \varepsilon \mathcal{V}$ that belong to the same component of $sD \setminus \mathcal{N}_\ell^\varepsilon$ but different components of $sD \setminus \mathcal{N}_\ell$. 

The first of these is bound above by the number of small excursion domains in $\bar{P}^\varepsilon(sD)$, the total multiplicity of double-crossings on the boundary $\partial sD$, and the total multiplicity of tubular-crossings in $sD$; the second is bound by the total multiplicity of tubular-crossings in~$sD$; the third is bound by the total multiplicity of four-crossings in~$sD$. 

All in all, defining the random variables 
\[  A^\varepsilon(sD)  :=  \sum_{e \in \varepsilon \partial \bar{P}^\varepsilon(sD) \cap \mathcal{E}} \tilde{D}(e) , \quad   B^\varepsilon(sD) := \sum_{ f \in \varepsilon \bar{P}^\varepsilon(sD) \cap \mathcal{F}} \tilde{F}(f)  , \quad C^\varepsilon(sD) := \sum_{ e \in \varepsilon \bar{P}^\varepsilon(sD) \cap \mathcal{E}} \tilde{T}(f)  , \]
and
\[  D^\varepsilon(sD) := |\{  \text{small excursion domains in } \bar{P}^\varepsilon(sD) \}|  , \]
we see that
 \[ |N_\ell(sD) - N_\ell^\varepsilon(sD)| \le A^\varepsilon(sD) + B^\varepsilon(sD)  + 2C^\varepsilon(sD) +  D^\varepsilon(sD).\]

As in the proof of Theorem \ref{thm:main2}, by Propositions \ref{prop:condc}, \ref{prop:contriple} and \ref{prop:consed} and the union bound, there exists a $c = c(\kappa, \ell, \mathcal{L}, D) > 0$ such that, for $\varepsilon \in (0, 1)$ and $s > 0$,
\[ \max\{  \mathbb{E}[A^\varepsilon(sD)] , \mathbb{E}[B^\varepsilon(sD)], \mathbb{E} [ C^\varepsilon(sD)] , \mathbb{E}[D^\varepsilon(sD)]  <  c  \varepsilon^{2-\delta} s^{2}  .  \]

Putting this together, we have
 \begin{align*}
\left| \frac{ \mathbb{E}[N_\ell^\varepsilon(sD)]}{ \text{Area}(sD)} - \frac{ \mathbb{E}[N_\ell(sD)]}{ \text{Area}(sD)}  \right| &= \frac{ | \mathbb{E}[N_\ell(sD)] - \mathbb{E}[N_\ell^\varepsilon(sD)]|  }{s^2 \text{Area}(D)}  \le \frac{ \mathbb{E}[|N_\ell(sD) - N_\ell^\varepsilon(sD)|]  }{s^2 \text{Area}(D)} \\
& \le \frac{ \mathbb{E}[A^\varepsilon(sD)] + \mathbb{E}[B^\varepsilon(sD)] + 2\mathbb{E}[C^\varepsilon(sD)] + \mathbb{E}[D^\varepsilon(sD)]  }{s^2 \text{Area}(D)} \\
& < 5c \text{Area}(D)  \varepsilon^{2-\delta} . 
 \end{align*}
Since $\mathbb{E}[N_\ell(sD)] / \text{Area}(sD)$ converges to $c_{NS}$ by definition, we have the result.

Remark that, by this argument, we extract a bound of $c  \varepsilon^{2 - \delta} $ on the rate of convergence in the statement of Theorem \ref{thm:main4}, for some $c = c(\kappa, \ell, \mathcal{L}, D, \delta)$.


\smallskip

\appendix

\section{Appendix A: Proof of the Gaussian estimates}
\label{appendix1}

In this appendix we give the proof of Propositions \ref{prop:three} and \ref{prop:smalldevest}.

\subsection*{Proof of Proposition \ref{prop:three}}
Let $\rho$ denote the spectral measure of $\Psi$, defined by the relation
\[ \kappa(s) = \int_{\mathbb{R}^2} e^{2 \pi i  \langle s, \mu \rangle} \, d \rho(\mu) , \quad s \in \mathbb{R}^2 .\]
By Assumption \ref{assumpt:degen} (as well as the remarks that follow), $\rho$ is not supported on a line. Hence by linear rescaling we may assume, without loss of generality, that 
\[ 
P \subseteq \rm{supp}(\rho), \]
where $P= \{(1,0), (-1,0), (0,1), (0, -1)\} \subseteq \mathbb{R}^2$.

We next argue that it is sufficient to prove the result for the field $\hat \Psi$ whose spectral measure~$\hat \rho$ is the sum of unit point-masses at each point in $P$. To see why, assume the claimed result fails for a given field $\Psi$, i.e. for each $\delta > 0$ there exist distinct non-colinear points $(s_i)_{i = 1,2,3}$ inside a ball of radius $\delta$ such that the matrix $K := (\kappa(s_j-s_k))_{1 \le j,k \le 3}$ is singular. For each such set of points, the singularity of $K$ is equivalent to the existence of non-zero $(c_i)_{i = 1,2,3}$ such that
\[   \sum_{j = 1}^3 \sum_{k = 1}^3 c_j c_k \kappa(s_j - s_k) =   \int_{\mathbb{R}^2}  \bigg|  \sum_{j = 1}^3 c_j e^{2 \pi i \langle s_i, \mu \rangle} \bigg|^2 d \rho(\mu)  = 0,  \]
which in turn is equivalent to $\rm{supp}(\rho)$ being contained in the zero set of  $g_s(\mu) :=\sum_{j = 1}^3 c_j e^{2 \pi i \langle s_i, \mu \rangle}$.
Hence, since $\rm{supp}(\hat \rho) \subseteq \rm{supp}(\rho)$, the same is true for $\rm{supp}(\hat \rho)$, and so the claimed result must also fail for~$\hat \Psi$.

Finally we show that the result is true for $\hat \Psi$ by direct computation. By stationarity, it is enough to show the existence of a $\delta > 0$ such that the distribution of
\begin{align}
\label{e:nondegen}
\left( \hat \Psi(0) , \hat \Psi(s_1), \hat \Psi(s_2) \right)
\end{align}
is non-degenerate for all distinct non-zero $s_1 := (x_1, y_1)$ and $s_2 := (x_2, y_2)$ in $[0, 1]^2 \cap B(\delta)$ such that $s_1$ and~$s_2$ are not co-linear with the origin; we actually prove the stronger statement that~\eqref{e:nondegen} is non-degenerate for all distinct non-zero $s_1$ and $s_2$ inside the square $[0, \pi/2)^2$, as long as $s_1$ and~$s_2$ do not both lie on the line $y=x$. The field $\hat \Psi$ can be represented as 
\[ \hat \Psi (x, y)  = X_1 \cos(x) + X_2 \sin(x) + X_3 \cos(y) + X_4 \sin(y) ,\]
where $X_i$ are i.i.d.\ standard Gaussian random variables. Defining the vector 
\[ V(x,y)=(\cos(x),\sin(x),\cos(y),\sin(y)) \in \mathbb{R}^4, \]
 the distribution of \eqref{e:nondegen} is non-degenerate if and only if the vectors $V(0)$, $V(s_1)$, and $V(s_2)$ are linearly independent. Indeed, the $3\times 4$ matrix $W$ with these rows has full rank if and only if the covariance matrix of \eqref{e:nondegen}, $\Sigma = WW^T$, is non-singular. By Gaussian elimination, these vectors are linearly independent if and only if
\[    ( \cos(x_1) - \cos(y_1), \sin(x_1), \sin(y_1)  ) \quad \text{and} \quad  ( \cos(x_2) - \cos(y_2), \sin(x_2), \sin(y_2)  ) \]
are linearly independent. Making the substitutions
\[  \bar x = \sin(x)  \quad \text{and} \quad  \bar y= \sin(y) = \lambda \bar x,\]
the result follows once we show that, for $\lambda > 1$, the function 
\[ g(\bar x) := \sqrt{1 - \bar x^2} - \sqrt{1 -  \lambda^2 \bar x ^2} \] 
is strictly convex on the domain $0 \le \bar x < 1/\lambda^2$. This may be checked by direct computation, since
\[  g''(x) = \frac{\lambda^8 x^2}{(1 - \lambda^4 x^2)^{3/2}} -  \frac{ x^2}{(1 -x^2)^{3/2}}   + \frac{\lambda^4 }{(1 - \lambda^4 x^2)^{1/2}}  - \frac{1}{(1 - x^2)^{1/2}}  > 0 .\]

\subsection*{Proof of Proposition \ref{prop:smalldevest}}
By Lemma \ref{lem:oddeven}, the two components of $\nabla \Psi(0)$ and the random vector $(\lambda_1, \lambda_2)$ are independent. Moreover, the distribution of $\nabla \Psi(0)$ is non-degenerate and so there exists a $c > 0$ such that
\[  \mathbb{P} \left(|\nabla \Psi(0)| <  \delta_1 \right)  < c \delta_1^2 . \] 
Hence it is sufficient to show that there exists a $c$ such that
\[ \mathbb{P} \left( |\lambda_1| <  \delta_2 , |\lambda_2| <  \delta_2, |\lambda_1 + \lambda_2| < \delta_3 \right)  < c \delta_2^2 \delta_3  .\]
We use the fact that $\nabla^2 \Psi(0)$ is distributed as
\[ \frac{1}{2} k^2 M , \]
where $M$ is a $2 \times 2$ symmetric random matrix with all elements centred Gaussian random variables with covariance
\[  \mathbb{E}[M_{ij} M_{kl}] = \frac{1}{2}( \delta_{ik} \delta_{jl} + \delta_{il} \delta_{jk} + \delta_{ij} \delta_{kl} ),  \] 
with $\delta_{ij}$ the Kronecker delta function (see \cite[Eq. (5)]{F04}, for which it is useful to note that, in the notation of that paper (see, especially, eq.\ (3)), $J^2 = 2 k^4 J_0^{(4)}(0) / 3 = k^4/4$ in the case of the random plane wave, where $J_0^{(4)}$ is the fourth derivative of the zeroth Bessel function). The density of the (ordered) eigenvalues $\mu_1 \le \mu_2$ of the matrix $M$ is known \cite[Theorem 2.2]{CS15}, and satisfies 
\[  f(\mu_1, \mu_2) = c_1 \exp \big\{ - \frac{1}{2} \sum_{i = 1,2} \mu_i^2  + \frac{1}{8} (\sum_i \mu_i)^2\big\} |\mu_1 - \mu_2| \id_{\{ \mu_1 \le \mu_2 \}} ,\]
for a certain normalising constant $c_1  > 0$. Hence there exists a $c_2 > 0$ such that the density of $\lambda_1$ and $\lambda_2$ satisfies 
\[ f(\lambda_1, \lambda_1) < c_2 |\lambda_1 - \lambda_2|  .\]
Integrating over the region $\{ |\lambda_1| <  \delta_2 , |\lambda_2| <  \delta_2, |\lambda_1 + \lambda_2| < \delta_3\}$ yields the result.

\smallskip
\section{Appendix B: Kac-Rice computations}
\label{appendix2}

In this appendix we give details of the matrix computations that we used in the Kac-Rice arguments in Section \ref{sec:kacrice} above. These are used to control the effect of conditioning on certain Gaussian vectors involving the random field $\Psi$ and its derivatives $\Psi_v$.

\begin{proposition}
\label{prop:append1}
Fix $\delta > 0$ and let $f, g: \mathbb{R}^+ \to \mathbb{R}$ be functions such that, for each $x \in (0, 1)$,
\[ \left| f(x) - 1 +  x^2 \right| < \delta x^4 \quad \text{and} \quad \left| g(x) +  2x \right| < \delta x^3  . \]
For each $x > 0$, define the matrices 
\[ A^x :=  \left[ \begin{array}{cc}
1 & f(x) \\
f(x) & 1 \\
\end{array} \right]  \ , \quad  B^x :=  \left[ \begin{array}{cc}
0 & g(x) \\
- g(x)  & 0 \\
\end{array} \right] , \] 
a matrix $C = (c_{ij})_{1 \le i,j \le 1}$ satisfying $c_{11} = c_{22} = 2$, the vector
\[ \mu^x := (\mu^x_i)_{i = 1,2} = B^x (A^x)^{-1}  (1, 1)^T, \]
and the matrix 
\[  D^x = (d^x_{ij})_{1 \le i,j \le 2} := C - (B^x)^T (A^x)^{-1} B^x  .   \]
Then there exists a constant $c = c(\delta) > 0$ such that, for each $x \in (0, 1)$,
\[  \frac{ \max_i \max\{ (\mu^x_i)^4 ,(d^x_{ii})^2  \} }{|A^x| } <  c x^2. \]
\end{proposition}

\begin{proof}
By explicit calculation, 
\[ |A^x| = 1 - f(x)^2 , \quad \mu^x_i = \frac{ g(x)(1-f(x))}{1 - f(x)^2} ,  \]
and
\[ d^x_{ii}  = \frac{2(1 - f(x)^2) - g(x)^2}{1-f(x)^2}. \]
Putting the bounds on $f$ and $g$ into the above formulae yields that, as $x \to 0$,
\[ |A^x| = 2  x^2 + O(x^4),  \quad |\mu_i^x| =  \frac{O(  x^3 )}{2 x^2 + O(x^4)} = O(x) \]
and
\[ d^x_{11} = d^x_{22} = \frac{O(x^4)}{2 x^2 + O(x^4)} = O(x^2) .\]
Combining these gives the result.
\end{proof}

\begin{proposition}
\label{prop:append2}
Fix a constant $\delta > 0$, an angle $\mu \in (0, \pi)$, and a set of unit vectors $v = (v_1, v_2, v_3) \in (S^1)^3$. Let $s = (s_1, s_2, s_3) \in (\mathbb{R}^2)^3$ be a (non-degenerate) triangle. For each $1 \le i \neq j \le 3$, denote by $\eta_{ij}$ the angle between the line segment $s_i s_j$ and the vector $v_j$, let $e^{s}_{ij} \in \mathbb{R}$ be such that
\[   |e^s_{ij} - e^s_{ik}| < \delta |\eta_{ij} - \eta_{ik}| , \]
and let $f^s_{ij}, g^{s}_{ij}: \mathbb{R}^+ \to \mathbb{R}$ be such that, for each $x \in (0, 1)$,
\[ \left| f^s_{ij}(x) - 1 + x^2 + e^s_{ij} x^4\right| < \delta x^6 \quad \text{and} \quad  \left| g^{s}_{ij}(x) + \cos(\eta_{ij}) \left(2 x + 4 e^s_{ij} x^3 \right) \right| < \delta x^5. \]
Suppose that $f^s_{ii} = 1$, $g^s_{ii} = 0$ and $f^{s}_{ij} = f^{s}_{ji}$. Define the matrices
\[ A^{s} :=  \left( f^s_{ij}(|s_i - s_j|) \right)_{1 \le i,j \le 3} , \] 
\[ B^{s} :=  \left( g^s_{ij}(|s_i - s_j|)  \right)_{1 \le i,j \le 3} , \]
a matrix $C = (c_{ij})_{1 \le i,j \le 3}$ satisfying $c_{11} = c_{22} = c_{33} = 2$, the vector
\[ \mu^s := (\mu^s_i)_{1 \le i \le 3} = B^s (A^s)^{-1}  (1, 1, 1)^T, \]
and the matrix 
\[  D^{s} = (d^{s}_{ij})_{1 \le i,j \le 2} := C - (B^{s})^T (A^s)^{-1} B^{s}  .   \]
Abbreviate $\varepsilon := \max_{i,j}(|s_i - s_j|)$. Then, there exists a $c = c(\delta, \mu, v) > 0$ such that, uniformly over all $s$ such that $\varepsilon < 1$, $\theta^-(s) \ge \varepsilon^{3/2}$ and $\theta^+(s) \le \mu$,
\[  \frac{  \max_i    \max\{ (\mu^s_i)^6, (d^{s}_{ii})^3 \} }{|A^{s}|} < c \left( \frac{\varepsilon}{\theta^-(s)}\right)^2  .\]
\end{proposition}

\begin{proof}
To ease the notation, we henceforth drop all dependencies on $s$. We also abbreviate $f_1 := f_{23} = f_{32}$ and similarly for $f_2$ and $f_3$. 

For $i = 1,2,3$, let $\alpha_i$ denote the angle in triangle $s$ at vertex $s_i$, and let the opposing side-length be $\varepsilon a_i$, for some $a_i \in (0, 1]$. Without loss of generality we may assume that $a_3 = 1$, and notice that $\max\{a_1, a_2 \} \ge 1/2$ by the triangle inequality. We may also assume, without loss of generality, that $\alpha_1 \le \alpha_2$ and $\alpha_ 1 \le \alpha_3$.

By explicit calculation,
\[ |A| = 1-f_1(\varepsilon a_1)^2 - f_2(\varepsilon a_2)^2 - f_3(\varepsilon a_3)^2 + 2f_1(\varepsilon a_1)f_2(\varepsilon a_2)f_3(\varepsilon a_3) , \]
\begin{align*}
 \mu_1 =  \frac{1}{|A| } &  \bigg( g_{12}(\varepsilon a_3) ( f_2(\varepsilon a_2) - 1)(f_1(\varepsilon a_1) + f_3(\varepsilon a_3) - f_2 (\varepsilon a_2) - 1)  \\
 & \qquad \qquad + g_{13}(\varepsilon a_2) ( f_3(\varepsilon a_3) - 1)( f_1(\varepsilon a_1) + f_2(\varepsilon a_2) - f_3(\varepsilon a_3) - 1) \bigg)  , 
 \end{align*}
and
\begin{align*}
d_{11}  = \frac{1}{|A| } &  \bigg(   2|A|  - g_{21}(\varepsilon a_3)^2(1-f_2(\varepsilon a_2)^2)  - g_{31}(\varepsilon a_2)^2(1-f_3(\varepsilon a_3)^2)  \\
& \qquad  \qquad \qquad   - 2 g_{21}(\varepsilon a_2) g_{31}(\varepsilon a_3) \left( f_2(\varepsilon a_2)f_3(\varepsilon a_3) - f_1(\varepsilon a_1) \right)    \bigg) ,
\end{align*}
with $\mu_{22}, \mu_{33}, d_{22}$ and $d_{33}$ similar.

First suppose that the angle $\alpha_1$ is uniformly bounded away from zero (i.e.\ uniformly in $s$). Then, by applying the bounds on $f_{ij}$ and $g_{ij}$, as $\varepsilon \to 0$,
\[ |A| = 4 a_2^2 \sin(\alpha_1) \varepsilon^4 + o(\varepsilon^4), \quad   \mu_{i} |A| = O(\varepsilon^5)  , \]
and
\begin{align*}  
 d_{11} |A|   = 8 a_2^2 \left( -\sin^2(\alpha_1) + 2\cos(\eta_{21}) \cos(\eta_{31}) \cos(\alpha_1) - \cos^2(\eta_{21}) - \cos^2(\eta_{31})  \right) \varepsilon^4 + O(\varepsilon^6) ,
\end{align*}
where we make use of the cosine rule
\[  a_1^2 = a_2^2 + 1 - 2a_2 \cos(\alpha_1) . \]
Notice that, depending on the orientation of $s_2 s_3$ and $v_1$,
\[ \alpha_1 \in \{\eta_{21} - \eta_{31}, \eta_{31} - \eta_{21}, \eta_{21} + \eta_{31}, 2\pi - \eta_{21} - \eta_{31} \} .\]
Hence because of the identity 
\[  \sin^2(\theta) + 2 \cos(\alpha) \cos(\beta) \cos(\theta) - \cos^2(\alpha) - \cos^2(\beta) = 0 ,  \]
valid for each pair of angles $\alpha$ and $\beta$, and any third angle
\[\theta \in \{ \alpha - \beta, \beta - \alpha, \alpha + \beta, 2\pi - \alpha - \beta \} , \]
the term of order $O(\varepsilon^4)$ vanishes, which leaves us with $d_{11} |A| = O(\varepsilon^6)$, with $d_{22}$ and $d_{33}$ similar. Putting this together, we have 
\begin{align*}  \frac{  \max_i    \max\{ \mu_i^6, d_{ii}^3 \} }{|A|}  &= \max \left\{   \frac{  \max_i  (\mu_i |A|)^6 }{|A|^7}  +    \frac{  \max_i  (d_{ii} |A|)^3 }{|A|^4}  \right\}   \\
& = \frac{O(\varepsilon^{30})}{ (4 a_2^2 \sin(\alpha_1))^7 \varepsilon^{28} + o(\varepsilon^{28}) } +  \frac{O(\varepsilon^{18})}{ (4 a_2^2 \sin(\alpha_1))^4 \varepsilon^{16} + o(\varepsilon^{16}) }    = O(\varepsilon^2)  ,
\end{align*}
which yields the result in the case that the angle $\alpha_1$ is uniformly bounded away from zero.

So it remains to deal with the case in which angle $\alpha_1$ is not uniformly bounded away from zero, but may instead be as small as $\varepsilon^{3/2}$ (this constraint is crucial in restricting the order to which we need to control errors in the Taylor expansions below). Suppose then that $\alpha_1 < \mu/2$. Then, since $\theta^+(s) \le \mu$, and as the angles in the triangle sum to $\pi$, it must be the case that $\alpha_i \ge (\pi -\mu)(1-\delta) \ge (\pi-\mu)/2$ for each $i = 2,3$. Using the sine rule and the bound $\sin(x) < x$ for $x > 0$, it follows that there exist a $c_1 = c_1(\mu)$ such that
\begin{equation}
 \label{eq:rel1}
 a_1  < c_1 \alpha_1 .
 \end{equation}
By the triangle inequality, it must also be the case that there exists a $c_2 = c_2(\mu)$ such that
\begin{equation}
 \label{eq:rel2}
 |a_2 - a_3| < c_2 \alpha_1 .
 \end{equation}
Finally,  by the continuity of the sequences $e_{ij}$, there also exists a $c_3 = c_3(\delta)$ such that
\begin{equation}
 \label{eq:rel3}
 |e_{12} - e_{13}| < c_3 \alpha_1  \quad \text{and} \quad    |e_{21} - e_{31}| < c_3 \alpha_1  .
 \end{equation}

We now recompute the lower bound on $|A|$ and the upper bounds on $\mu_i |A|$ and $d_{ii} |A|$ using the relations \eqref{eq:rel1}--\eqref{eq:rel3}. It may be verified by Taylor expanding up to order $\varepsilon^7$ that there exists a $c_5 = c_5(\delta, \mu, v)$ such that 
\[ |A| > c_5 \alpha_1^2 \varepsilon^4   . \]
Moreover, applying the cosine rule 
\[  a_1^2 = a_2^2 + 1 - 2a_2 \cos(\alpha_1)  \]
and relations \eqref{eq:rel1} and \eqref{eq:rel3}, it can be verified by Taylor expanding up to order $\varepsilon^9$ that there exists a $c_6 = c_6(\delta, \mu, v)$ such that 
\[   \mu_1|A| <  c_6 \alpha_1^2 \varepsilon^5  \quad \text{and}  \quad  d_{11} |A| <  c_6 \alpha_1^2 \varepsilon^6 . \]
Similarly, applying the cosine rules
\[  a_2^2 = a_1^2 + 1 - 2a_1 \cos(\alpha_2)  \quad \text{and} \quad 1 = a_1^2 + a_2^2 - 2 a_1 a_2 \cos(\alpha_3)   \]
and relations \eqref{eq:rel2} and \eqref{eq:rel3}, there exists a $c_7 = c_7(\delta, \mu, v)$ such that 
\[   \max_{i = 2,3} \mu_{i} |A| <  c_7 \alpha_1^2 \varepsilon^5  \quad \text{and}  \quad  \max_{i = 2,3} \{ d_{ii} |A|  \} <  c_7 \alpha_1^2 \varepsilon^6 . \]
Combining these gives the result.
\end{proof}

\smallskip
\section{Appendix C: Proof of the Russo-Seymour Welsh estimates}
\label{appendix3}

In this appendix we give the proof of Theorem \ref{thm:rsw}; for this, it is helpful to have read both \cite{BG16} and \cite{Tas16}. 

We follow the proof in \cite{BG16}, but make three essentially distinct improvements to the arguments. The main improvement is to replace the discretisation scheme in \cite[Theorem]{BG16} with our enhanced scheme in Theorem~\ref{thm:main2}; this is already enough to reduce the required decay exponent from $\alpha \approx 325$ to $\alpha \approx 55$. The other two improvements are relatively minor. Before explaining them, we state a preliminary proposition on the quasi-independence of signs of Gaussian vectors, which is a strengthened form of \cite[Theorem 4.3]{BG16}. This bound is completely independent of other arguments, and may be of independent interest. 

\subsection{Quasi-independence of the signs of Gaussian vectors}

\begin{proposition}[Quasi-independence of the signs of Gaussian vectors; see {\cite[Theorem 4.3]{BG16}}]
\label{p:qi}
There exists a constant $c > 0$ such that, for any centred, normalised Gaussian vector $(X, Y)$ of dimension $m + n$ with covariance
 \[   \Sigma = \left[  
\begin{array}{cc}
\Sigma_X & \Sigma_{XY} \\
\Sigma_{XY}^T & \Sigma_Y \\
\end{array} \right]  ,\]
and any events $A$ and $B$ depending only on the signs of $X$ and $Y$ respectively, it holds that
\begin{align}
\label{e:bound}
 | \mathbb{P}(A \cap B) - \mathbb{P}(A) \mathbb{P}(B) | &< c \left( (m + n) \| \Sigma_{XY} \|_{2}  \log (1/ \| \Sigma_{XY} \|_{2}  ) \right)^{1/3} 
 \end{align}
where $\| \cdot \|_{2}$ denotes the Schatten $2$-norm (i.e.\ $\ell_2$-norm of the vector of matrix entries). 
\end{proposition}
\begin{remark}
If $\eta$ denotes the maximum absolute entry of $\Sigma_{XY}$, \eqref{e:bound} is at most
\[  c  (m + n)^{4/3} \eta^{1/3}  \log (1/\eta )^{1/3}   .  \]
Note that \cite[Theorem 4.3]{BG16} gave the weaker bound of  $c (m+n)^{8/5} \eta^{1/8}$ on this quantity.
\end{remark}

\begin{proof}
We follow the outline of the proof of \cite[Theorem 4.3]{BG16} but improve the estimates in each of the two main steps. Consider the spectral expansion of $X$ and $Y$
\[  
X = \sum_{i = 1}^m a_i \lambda_{1, i}^{1/2} u_{1,i} \quad \text{and} \quad Y =  \sum_{i = 1}^n b_i \lambda_{2, i}^{1/2} u_{2, i} ,  
\]
where $(a_i)_{i = 1}^m$ and $(b_i)_{i = 1}^n$ are vectors consisting of independent standard normal random variables, and $(\lambda_{1,i})_{i = 1}^m$ and $(\lambda_{2,i})_{i = 1}^n$ are the eigenvalues, in descending order, of $\Sigma_X$ and $\Sigma_Y$ respectively, with $(u_{1,i})_{i = 1}^m$ and $(u_{2, i})_{i = 1}^n$ the corresponding orthonormal eigenvectors. Fix a threshold $\lambda > 0$ to be set later, and truncate $X$ and $Y$ as
\[   X = \tilde{X} + \hat{X} = \sum_{i = 1}^{m_0} a_i \lambda_{1, i}^{1/2} u_{1,i} + \sum_{i = m_0 +1}^{m} a_i \lambda_{1, i}^{1/2} u_{1,i}     \]
and
\[   Y = \tilde{Y} + \hat{Y} = \sum_{i = 1}^{n_0} b_i \lambda_{2, i}^{1/2} u_{2,i} + \sum_{i = n_0 +1}^{m} b_i \lambda_{2, i}^{1/2} u_{2,i}     \]
where $m_0$ and $n_0$ are chosen so that $\lambda_{1,m_0} \ge \lambda > \lambda_{1, m_0 + 1}$ and $\lambda_{2, n_0} \ge \lambda > \lambda_{2, n_0 + 1}$.

The first step is to control the probability that $X$ and $\tilde{X}$ (respectively $Y$ and $\tilde{Y}$) have the same signs. Set $\varepsilon_1 :=  \sqrt{c_0 \lambda m}$ and $\varepsilon_2 := \sqrt{c_0 \lambda n}$ for a constant $c_0 > 1$ to be determined later. Notice that by the union bound
\begin{equation}
\label{e:x}
\mathbb{P} \bigg( \bigcap_{i = 1}^m |X_i| >  \varepsilon_1 \bigg) \ge   1 - m \varepsilon_1 = 1 - c_0^{1/2} \lambda^{1/2} m^{3/2} . 
\end{equation}
On the other hand, applying Parseval's theorem 
\begin{align*}
\mathbb{P} \bigg( \bigcup_{i = 1}^ m |\hat{X}_i| > \varepsilon_1 \bigg)    & \le   \mathbb{P} \left( \sum_{i = 1}^m \hat{X}_i^2 > \varepsilon_1^2 \right) = \mathbb{P} \left( \sum_{i = m_0 +1}^{m} a_i^2 \lambda_{1, i}   > \varepsilon_1^2 \right)   \le  \mathbb{P} \left( \sum_{i = m_0 + 1}^{m} a_i^2   > \frac{\varepsilon_1^2}{\lambda} \right)
\end{align*}
where in the last step we used the fact that $\lambda_{1, i} < \lambda$ for $i > m_0$. Applying the Markov inequality, for any $a \in (0, 1/2)$ the last probability is at most
\begin{align*}
 e^{-\frac{a \varepsilon_1^2}{\lambda} } \mathbb{E} \left[ e^{ a \sum_{i = m_0 + 1}^{m} a_i^2 } \right] = e^{-\frac{a \varepsilon_1^2}{\lambda} } (\sqrt{1 - 2a})^{-(m - m_0)}  \le  e^{-\frac{a \varepsilon_1^2}{\lambda} } (\sqrt{1 - 2a})^{-m}   .
\end{align*}
Setting 
\[  a := \frac{1}{2} \left( 1 - \frac{\lambda m }{\varepsilon_1^2}  \right)= \frac{1}{2} \left(1 - \frac{1}{c_0} \right) \in (0, 1/2) \]
and combining with \eqref{e:x} yields an upper bound of
\[  c_0^{1/2} \lambda^{1/2} m^{3/2} +  e^{-\frac{m}{2} \left( c_0 - \log(c_0) - 1 \right) } \]
on the probability that $X$ and $\tilde{X}$ do not have the same signs. Similarly, replacing $\varepsilon_1$ with $\varepsilon_2$,
\[  c_0^{1/2} \lambda^{1/2} n^{3/2} +  e^{-\frac{n}{2} \left( c_0 - \log(c_0) - 1 \right) } \]
is an upper bound on the probability that $Y$ and $\tilde{Y}$ do not have the same signs.

The second step is to control the total variation distance between $(\tilde{X}, \tilde{Y})$ and the vector $(\bar{X}, \bar{Y})$ formed from independent copies of $\tilde{X}$ and $\tilde{Y}$; on the event that $X$ and $\tilde{X}$ have the same signs, this gives an upper bound on the quantity in \eqref{e:bound}. By linear transformation, this total variation distance is the same as the total variation distance between the vector $Z = (a_1, \ldots, a_{m_0}, b_1, \ldots, b_{n_0})$ and a vector of $m_0 + n_0$ independent standard normal random variables. By Pinsker's inequality, this distance is bound above by
\[   \frac{1}{2} \sqrt{   \log (1/  | \Sigma_Z| )} \]
where $\Sigma_Z$ is the covariance matrix of $Z$. Observe that $\Sigma_Z$ has the block form
\[     \left[  
\begin{array}{cc}
\id_{m_0 \times m_0} &  \Lambda_1^{-1/2} U_1^T \Sigma_{XY} U_2^T \Lambda_2^{1/2} \\
  \Lambda_2^{-1/2} U_2 \Sigma_{XY}^T U_1 \Lambda_1^{1/2}  & \id_{n_0 \times n_0} \\
\end{array} \right]   \]
where $U_1 := (u_{1, i})_{i = 1}^{m_0}$, $U_2 := (u_{2, i})_{i = 1}^{n_0}$, and $\Lambda_1$ and $\Lambda_2$ are diagonal matrices formed from  the vectors $(\lambda_{1, i})_{i = 1}^{m_0}$ and $(\lambda_{2, i})_{i = 1}^{n_0}$ respectively. Applying the block determinant formula,  $|\Sigma_Z| =  |  \id_{m_0 \times m_0}    - E |$,
where $E$ is the symmetric positive-definite matrix
\[ E :=   \Lambda_1^{-1/2} \left( U_1^T \Sigma_{XY} U_2^T \right)  \Lambda_2^{-1}  \left( U_2 \Sigma_{XY}^T U_1  \right)\Lambda_1^{1/2} , \]
Applying the bound in Lemma~\ref{l:tracebound} below gives that
\[  0 \le \log( 1 / | \Sigma_Z | ) \le    2 \text{Tr}(E)   , \]
whenever $\text{Tr}(E) < 1/2$. Using the invariances of Schatten $p$-norms $\| \cdot \|_{p}$ under isometry and the H\"{o}lder inequality for these norms, 
\begin{equation*}
\text{Tr}(E) =  \| E \|_1 \le \|  \Lambda_1  \|_{\infty}^{-1/2}   \|   \Sigma_{XY}   \|_{2}  \|   \Lambda_2  \|_{\infty}^{-1}   \|   \Sigma_{XY}  \|_{2}  \|   \Lambda_1 \|_{\infty}^{1/2}    \le \lambda^{-2}  \|  \Sigma_{XY}   \|_{2}^2 .
\end{equation*}
To conclude, as long as $\lambda >  \sqrt{2} \|\Sigma_{XY} \|_{2} $, the total variation distance between  $(\tilde{X}, \tilde{Y})$ and $(\bar{X}, \bar{Y})$ is at most $ 2^{-1/2} \lambda^{-1}  \|\Sigma_{XY} \|_{2} $.

Putting both steps together, we have established a bound on \eqref{e:bound} of the form
\[ c_0^{1/2} \lambda^{1/2} m^{3/2} +  c_0^{1/2} \lambda^{1/2} n^{3/2} +   e^{-\frac{m}{2} \left( c_0 - \log(c_0) - 1 \right) } +  e^{-\frac{n}{2} \left( c_0 - \log(c_0) - 1 \right) }  +  \frac { \| \Sigma_{XY} \|_2 }{\sqrt{2} \lambda}  \]
for any choice of $c_0 > 1$ and $\lambda > \sqrt{2} \|\Sigma_{XY} \|_{2} $. Setting
\[  c_0 := 2 + \log (1/ \| \Sigma_{XY} \|_2 )   \quad \text{and} \quad   \lambda :=  \sqrt{2}  (m+n)^{-1} \| \Sigma_{XY} \|_2^{1/3}     \]
yields the result for sufficiently large $c > 0$ (in particular, this setting of $\lambda$ is valid whenever the statement in \eqref{e:bound} has any content).
\end{proof}

\begin{lemma}
\label{l:tracebound}
Let $E$ be a $n \times n$ symmetric positive-definite matrix such that $\text{Tr}(E) < 1/2$. Then
\[    -  \log | \id_{n \times n} - E|   \le 2 \text{Tr}(E) .\]
\end{lemma}
\begin{proof}
Let the eigenvalues of $E$ be $\lambda_i(E) \in [0, \text{Tr}(E)] < 1/2$. Since the spectral radius of $E$ is less than one, we may write
\[   - \log | \id_{n \times n} - E|  = \sum_i - \log(1 - \lambda_i(E) ) . \]
 The result then follows since $\log(1-x) \ge -2x$ if $x \in[0, 1/2]$.
 \end{proof}

\subsection{Proof of Theorem \ref{thm:rsw}}

Before we begin, we remark that the overall decay exponent of $\alpha = 16$ arises essentially from three distinct sources: it can be understood as $\alpha = d(1 + a)b$, where $d$ is the dimension of $\mathbb{R}^2$, $\varepsilon = s^{-a}$ is the required mesh-size in the discretisation scheme (in our case $a = 1$ by Theorem~\ref{thm:main2}), and $\eta = m^{-b}$ is the maximum size of the cross-correlation terms between two Gaussian vectors of size $m$ in order to guarantee quasi-independence of the law of the signs (in our case $b = 4$ by Proposition \ref{p:qi}). As such, any improvement in the exponents in either Theorem~\ref{thm:main2} or Proposition \ref{p:qi} would lower this decay exponent.

Assume that $\kappa \ge 0$ and that $\kappa$ is invariant under reflection through the horizontal axis and under rotation by $\pi/2$. Fix a periodic lattice $\mathcal{L}$ that is a triangulation, is invariant under reflection through the horizontal axis and under rotation by $\pi/2$, and is integral in the sense that $\mathcal{V} \subseteq (\frac{1}{N} \mathbb{Z})^2$ for some $N \in \mathbb{N}$. Fix a sequence $\varepsilon = \varepsilon(s)$, a smooth bounded domain $D$ and disjoint boundary arcs $\gamma$ and $\gamma'$ on $\partial D$. 

Consider the discretised excursion domain $sD \cap \mathcal{N}^\varepsilon$, and observe that there is a natural way to partition this set into a positive regime $\mathcal{O}^+ = \mathcal{O}^+(s, \varepsilon, D)$ and a negative region $\mathcal{O}^- = \mathcal{O}^-(s, \varepsilon, D)$ depending on the sign of $\Psi$ on the vertices lying inside each component of~$\mathcal{N}^\varepsilon$. The main result in \cite[Theorem 4.9]{BG16} is that, if 
\begin{equation}
\label{eq:rsw2}
\left(  s \varepsilon^{- (1 + \log(3/2)/\log(4/3) ) } \right)^{2 \times 8} s^{-\alpha} =  o(1)
 \end{equation}
 as $s \to \infty$, then there exists a $c = c(\kappa, D, \gamma, \gamma') > 0$ such that, for $s > 0$ sufficiently large,
 \begin{equation*}
 \mathbb{P} \left( \text{one component of each of } \mathcal{O}^+ \text{ and }  \mathcal{O}^- \text{ intersects both } s\gamma \text{ and } s \gamma' \right) > c .
 \end{equation*}
 To interpret the constituent parts of \eqref{eq:rsw2}, the term
 \begin{equation}
 \label{e:const}
  \tilde{ \varepsilon} = \varepsilon^{1 + \log(3/2)/\log(4/3)}
  \end{equation}
 is an adjustment to the lattice size to ensure that Tassion's argument in \cite{Tas16} applies, the term $m = \left(  s \tilde{\varepsilon}^{-1} \right)^2$ is the order of the number of vertices of $\varepsilon \mathcal{L}$ inside $sD$, the term $\eta = s^{-\alpha}$ gives the order of correlations on the scale $s$, and finally the relationship
 \begin{equation}
 \label{e:qi3}
  m^{8} \eta^{-1} = o(1)
  \end{equation}
is a sufficient condition, by \cite[Theorem 4.3]{BG16}, to ensure the asymptotic quasi-independent of the signs of $\varepsilon \mathcal{L} \cap sD_1$ and $\varepsilon \mathcal{L} \cap sD_2$ for disjoint domains $D_1$ and $D_2$. 
 
Applying the discretisation scheme in \cite[Theorem 1.5]{BG16}, valid if $\varepsilon = o(s^{-8-\delta})$ and under certain additional conditions on $\kappa$, the conclusion in \cite{BG16} is that
 \begin{equation}
 \label{eq:rsw3}
 \mathbb{P} \left( \text{there exists a component of } sD \cap \mathcal{N} \text{ that intersects } s\gamma \text{ and } s \gamma' \right) > c 
 \end{equation}
 as long as $\kappa(x) = o(|x|^{-\alpha - \delta})$ for
\[   \alpha > (1 + 8 \times (1 + \log(3/2)/\log(4/3) ) ) \times 2 \times 8 \approx 325 . \]

Our first improvement to the decay exponent is achieved by a direct substitution of the discretisation scheme in \cite[Theorem 1.5]{BG16} with the scheme in Theorem \ref{thm:main2}, which is valid if $\varepsilon = o(s^{-1-\delta})$ and under Assumption \ref{assumpt:degen}. Then it is immediate that \eqref{eq:rsw3} holds as long as $\kappa(x) = o(|x|^{-\alpha - \delta})$ for
\begin{equation*}
\alpha > (1 + 1 \times (1 + \log(3/2)/\log(4/3) ) ) \times 2 \times 8   \approx 55 . 
\end{equation*}

To further improve the decay exponent, we make two additional enhancements to the argument. First, we remark that the exponent 
\[ c := 1 + \log(3/2)/\log(4/3) \]
in \eqref{e:const} originates from the choice of the constant $\rho = 2/3$ in \cite[Lemma 2.2]{Tas16}, and in particular can be written as
\[ c(\rho) := 1 + \log(1/\rho)/\log(2 \rho) . \]
We then observe that \cite[Lemma 2.2]{Tas16} still holds, with an identical proof, if we replace the constant $2/3$ with any $\rho \in (0, 1)$, and so the exponent $c$ can be replaced with $1 + \delta_1$ for any $\delta_1 > 0$ by taking $\rho$ close enough to one.

Second, we substitute the strengthened quasi-independence result in Proposition \ref{p:qi} in place of \cite[Theorem 4.3]{BG16}. The consequence is that the exponent in \eqref{e:qi3} may be reduced from $8$ to $4 + \delta_2$, for any $\delta_2 > 0$.

Combining these improvements, we conclude that \eqref{eq:rsw3} holds as long as  $\kappa(x) = o(|x|^{-\alpha - \delta})$ for
\[ \alpha > (2 + \delta_1) \times 2 \times (4 + \delta_2)  ,\]
which can be made arbitrarily close to $16$ by choosing $\delta_1$ and $\delta_2$ small enough.


\bigskip


\bibliography{BM_bibliography}{}

\begin{thebibliography}{10}

\bibitem{Adler}
R.J. Adler.
\newblock {\em The geometry of random fields}.
\newblock Classics in Applied Mathematics. SIAM, Philadelphia, 2010.

\bibitem{Alex96}
K.S. Alexander.
\newblock Boundedness of level lines for two-dimensional random fields.
\newblock {\em Ann. Probab.}, 24:1653--1674, 1996.

\bibitem{AW}
J.~Aza{\"{\i}}s and M.~Wschebor.
\newblock {\em Level sets and extrema of random processes and fields}.
\newblock John Wiley \& Sons, Inc., Hoboken, NJ, 2009.

\bibitem{BG16}
V.~Beffara and D.~Gayet.
\newblock Percolation of random nodal lines.
\newblock {\em Publ. Math. IHES}, 2017.
\newblock https://doi.org/10.1007/s10240-017-0093-0.

\bibitem{BK13}
D.~Beliaev and Z.~Kereta.
\newblock On the {B}ogomolny--{S}chmit conjecture.
\newblock {\em Journal of Physics A: Mathematical and Theoretical},
  46(45):455003, 2013.

\bibitem{Berry77}
M.V. Berry.
\newblock Regular and irregular semiclassical wavefunctions.
\newblock {\em Journal of Physics A: Mathematical and General}, 10(12):2083,
  1977.

\bibitem{BS02}
E.~Bogomolny and C.~Schmit.
\newblock Percolation model for nodal domains of chaotic wave functions.
\newblock {\em Phys. Rev. Lett.}, 88:114102, Mar 2002.

\bibitem{CS15}
D.~Cheng and A.~Schwartzman.
\newblock Expected number and height distribution of critical points of smooth
  isotropic {G}aussian random fields.
\newblock {\em arXiv:1511.06835 (to appear in Bernoulli)}, 2015.

\bibitem{F04}
Y.V. Fyodorov.
\newblock Complexity of random energy landscapes, glass transitions and
  absolute value of spectral determinant of random matrices.
\newblock {\em Phys. Rev. Lett.}, 92:240601, 2004.

\bibitem{Konrad12}
K.~Konrad.
\newblock Asymptotic statistics of nodal domains of quantum chaotic billiards
  in the semiclassical limit.
\newblock Senior Thesis, Dartmouth College, 2012.

\bibitem{MW07}
K.~Mischaikow and T.~Wanner.
\newblock Probabilistic validation of homology computations for nodal domains.
\newblock {\em Ann. Appl. Probab.}, 17:980--1018, 2007.

\bibitem{NS09}
F.~Nazarov and M.~Sodin.
\newblock On the number of nodal domains of random spherical harmonics.
\newblock {\em Amer. J. Math.}, 131(5):1337--1357, 2009.

\bibitem{NS15}
F.~Nazarov and M.~Sodin.
\newblock Asymptotic laws for the spatial distribution and the number of
  connected components of zero sets of {G}aussian random functions.
\newblock {\em J. Math. Phys. Anal. Geo.}, 12(3):205--278, 2016.

\bibitem{Sun93}
X.P. Sun.
\newblock Conditionally positive definite functions and their application to
  multivariate interpolations.
\newblock {\em J. Approx. Theory}, 74(2):159--180, 1993.

\bibitem{Tas16}
V.~Tassion.
\newblock Crossing probabilities for {V}oronoi percolation.
\newblock {\em Ann. Probab.}, 44(5):3385--3398, 2016.

\bibitem{Wendland}
H.~Wendland.
\newblock {\em Scattered Data Approximation}.
\newblock Cambridge Monographs on Applied and Computational Mathematics.
  Cambridge University Press, 2005.

\end{thebibliography}
\bibliographystyle{plain}

\end{document}